\documentclass[11pt,reqno]{amsart}
\usepackage[utf8]{inputenc}

\usepackage{amsmath,amsfonts,amsthm,amscd,amssymb,mathabx}
\usepackage{mathrsfs}
\usepackage{bm}
\usepackage{cases}
\usepackage[hypertexnames=false,hyperfootnotes=false]{hyperref}
\usepackage{cite}
\usepackage{cleveref}
\usepackage{enumitem}
\usepackage{graphicx}
\usepackage{etoolbox}

\numberwithin{equation}{section}

\crefname{equation}{}{}

\usepackage[dvipsnames]{xcolor}

\newtheorem{theorem}{Theorem}[section]
\newtheorem{lemma}[theorem]{Lemma}
\newtheorem{proposition}[theorem]{Proposition}

\theoremstyle{definition}

\newtheoremstyle{mystyle-remark}{6pt}{6pt}{\normalfont}{}
{\normalfont\bfseries}{.}{0.5em}{}
\theoremstyle{mystyle-remark}
\newtheorem{remark}[theorem]{Remark}
\AtEndEnvironment{remark}{\unskip\nobreak\hfill\(\diamond\)}

\newcommand{\RR}{\mathbb{R}}

\newcommand{\ii}{\mathrm{i}}
\newcommand{\dd}{\mathrm{d}}
\newcommand{\ee}{\mathrm{e}}

\newcommand{\comm}[2]{[#1,#2]}

\newcommand{\Flam}{\widehat{\lambda}}
\newcommand{\Fomega}{\widehat{\omega}}

\DeclareMathOperator{\Tr}{Tr}

\begin{document}

\title{Modified Scattering for Infinite-Rank Hartree--Fock Equations}

\author[Kawamoto]{Masaki Kawamoto}
	\address[M. Kawamoto]{
		Research Institute for Interdisciplinary Science,
		Okayama University,
		3-1-1, Tsushimanaka, Kita-ku,
		Okayama City, Okayama, 700-8530, Japan
	}
	\email{\href{mailto:kawamoto.masaki@okayama-u.ac.jp}{kawamoto.masaki@okayama-u.ac.jp}}

\author[Lee]{Jinyeop Lee}
	\address[J. Lee]{
		Department of Applied Mathematics,
		Kyung Hee University,
		1732 Deogyeong-daero, Giheung-gu,
		Yongin-si, Gyeonggi-do, South Korea
	}
	\email{\href{mailto:jinyeop.lee@khu.ac.kr}{jinyeop.lee@khu.ac.kr}}

\author[Yang]{Changhun Yang}
	\address[C. Yang]{
		Department of Mathematics,
		Chungbuk National University,
		Cheongju 28644, South Korea
	}
	\email{\href{mailto:chyang@chungbuk.ac.kr}{chyang@chungbuk.ac.kr}}

\author[You]{Chanjin You}
	\address[C. You]{
		Department of Mathematics,
		Yale University,
		219 Prospect Street, New Haven, CT 06511, USA
	}
	\email{\href{mailto:chanjin.you@yale.edu}{chanjin.you@yale.edu}}

\date{\today}

\begin{abstract}
    We study the long-time behavior of the Hartree--Fock equation with Coulomb interaction in dimensions $d\geq2$ for small initial density matrices of possibly infinite rank. We prove global well-posedness and modified scattering governed by a resonant Hartree-Fock equation. Our approach is based on the half-density formulation and a final-state construction for the resonant dynamics. We also characterize the stationary states of the resonant Hartree--Fock equation and prove that ordinary linear scattering fails when the initial density matrix has rank at least two. These results extend long-range modified scattering to genuinely operator-valued infinite-rank Hartree-Fock evolution.
    Despite cancellation between the direct and exchange terms, the long-range effect persists, with resonant dynamics evolving on the same logarithmic time scale as in the Hartree equation.
\end{abstract}

\maketitle

{\let\thefootnote\relax
\footnotetext{All authors contributed equally; authors are listed in alphabetical order.}}

%\tableofcontents

\section{Introduction}

\subsection{Hartree--Fock equations}
Let $\gamma(t)$ be a nonnegative trace-class operator on $L^2(\RR^d)$ with integral kernel $\gamma(t,x,y)$.
We consider the time-dependent Hartree--Fock equation
\begin{equation}\label{eq:HF}
	\ii\partial_t\gamma = \comm{H_\gamma}{\gamma},
\end{equation}
where
\begin{equation*}
	H_\gamma = -\frac{1}{2}\Delta  + V*\rho_\gamma - {\mathsf{X}}_\gamma.
\end{equation*}
Here $\rho_\gamma(t,x)=\gamma(t,x,x)$, $V \ast \rho_\gamma$ is a multiplication operator, and $\mathsf{X}_A$ denotes the exchange operator with integral kernel
\begin{equation}\label{def:exchnage-operator}
    {\mathsf{X}}_A(t,x,y)=V(x-y)A(t,x,y),
\end{equation}
where $A(t,x,y)$ denotes the integral kernel of $A(t)$.

We mainly consider the Coulomb potential
\[
	V(x)=|x|^{-1}
\]
in $d \ge 2$. 
The small data global-wellposedness and modified scattering results also hold for $V(x) = C|x|^{-1}$ with $C \in \mathbb{R} \setminus \{ 0\}$. For simplicity, we take $C=1$.

\medskip

The time-dependent Hartree--Fock equation \eqref{eq:HF} gives a self-consistent description of fermionic dynamics in terms of the one-particle density matrix. The Hartree potential $V*\rho_\gamma$ represents the average interaction generated by the particle density, while the exchange operator $\mathsf X_\gamma$ accounts for the antisymmetry of fermionic wave functions.

For $N$ fermions, the many-body Hamiltonian is
\[
    H_N=-\frac12\sum_{j=1}^N\Delta_{x_j}+\sum_{1\le i<j\le N}V(x_i-x_j)
\]
acting on $L^2(\RR^d)^{\otimes_{\mathrm{asym}} N}$. An $N$-particle density operator $\Gamma_N$ with $\Tr\Gamma_N=1$ evolves according to $\ii\partial_t\Gamma_N=[H_N,\Gamma_N]$. Its one-particle density matrix is $\gamma_N=N\Tr_{2,\ldots,N}\Gamma_N$. The Dirac--Frenkel approximation by Slater determinants gives \eqref{eq:HF} for the associated one-particle density matrix $\gamma=\sum_{j=1}^N|\varphi_j\rangle\langle\varphi_j|$, where the orbitals $\{\varphi_j\}_{j=1}^{N}$ form an orthonormal family \cite{BenedikterSokSolovej2018}. In this case, $\gamma$ is a rank-$N$ orthogonal projection.

The Hartree--Fock approximation has been justified in various mean-field regimes for initial states close to Slater determinants \cite{BardosGolseMauser2003,BenedikterPortaSchlein2014,PetratPickl2016} and for mixed quasi-free states \cite{BenedikterJaksicPortaSaffirioSchlein2016}. The Coulomb interaction was treated under additional assumptions on the Hartree--Fock evolution in \cite{PortaRademacherSaffirioSchlein2017}. Further results for mixed states with singular interactions, including convergence in Schatten norms, were obtained in \cite{ChongLaflecheSaffirio2024}. At the level of ground-state energies, the leading correction to Hartree--Fock theory has been studied by bosonization of particle--hole excitations \cite{BenedikterNamPortaSchleinSeiringer2021,ChristiansenHainzlNam2023RPA}. For the Coulomb interaction, matching upper and lower bounds yield the mean-field analogue of the Gell-Mann--Brueckner formula for the correlation energy of the electron gas \cite{ChristiansenHainzlNam2023Upper,ChristiansenHainzlNam2027}.

The density-matrix formulation also applies to mixed quasi-free states without pairing, whose one-particle density matrices satisfy $0\le\gamma\le1$ and $\Tr\gamma<\infty$ but need not be projections \cite{BenedikterJaksicPortaSaffirioSchlein2016}. Their occupation numbers may be nonzero for infinitely many orbitals while remaining summable. Allowing infinite rank is therefore natural even when the expected number of particles is finite.

Such examples arise from the grand-canonical Gibbs states of noninteracting fermions in a harmonic trap at positive temperature \cite{ChongLaflecheSaffirio2023Thermal}. As shown in Remark~\ref{rem:thermal-data}, their one-particle density matrices satisfy the assumptions of our main theorem when the chemical potential is sufficiently negative. In this paper, we study the long-time behavior of small, localized trace-class density matrices without imposing a finite-rank assumption.

\subsection{Previous results}
The Cauchy problem for the Hartree--Fock equation has been studied extensively.
Early results include \cite{Bove1974,Chadam1975, Bove1976, Chadam1976,Zagatti1992}, under various assumptions on the interaction potential and the initial data.

The large time behavior of finite-rank Hartree--Fock equations was studied by Wada \cite{Wada1999,Wada2002}.
For the short-range interaction $V(x)=|x|^{-\theta}$, $4/3 < \theta < \min \{ 4, d \}$,  Wada \cite{Wada1999} constructed wave operators and proved asymptotic stability for large initial data.
For the long-range interaction $V(x) = |x|^{-1}$ in dimensions $d \ge 2$, Wada \cite{Wada2002} constructed modified wave operators for small final states and proved modified scattering for small initial data.
Ikeda \cite{Ikeda2012} subsequently constructed the modified scattering operator for small data. Recently, Mal\'ez\'e \cite{Maleze2025} proved small data scattering for localized trace class initial density matrices in one dimension with interaction potentials given by finite measures. The key observation is a leading order cancellation between the direct and exchange terms, which eliminates the expected long-range effect.

The asymptotic stability of translation-invariant steady states for the Hartree--Fock equation has also been studied for a class of short-range potentials. Collot, Danesi, de Suzzoni, and Mal\'ez\'e \cite{Collot2025} proved asymptotic stability in dimensions $d\ge 4$, while Nguyen and You \cite{NguyenYou2026} proved nonlinear Landau damping and asymptotic stability in dimensions $d \ge 3$. These results, however, do not cover long-range potentials such as the Coulomb potential.

\subsection{Modified scattering and resonant dynamics}
To motivate our approach, we first consider the scalar Hartree equation
\begin{equation}\label{NLS}
    \ii\partial_t u
    =-\frac12\Delta u+(V_\theta*|u|^2)u,
    \qquad V_\theta(x)=|x|^{-\theta},
    \qquad 0<\theta<d,
\end{equation}
where $u:\mathbb R\times\mathbb R^d\to\mathbb C$ and $d\ge2$.
This equation can be viewed as a simplified model of the Hartree--Fock equation \eqref{eq:HF}, obtained by neglecting the exchange term and restricting the resulting density-matrix evolution to rank-one states $\gamma(t)=|u(t)\rangle\langle u(t)|$.\footnote{In the full Hartree--Fock equation, by contrast, the direct and exchange terms cancel on rank-one states.} Equation \eqref{NLS} has been extensively studied, and our analysis in the operator setting builds on
the dispersive techniques developed for this model.

The spatial decay of the potential $V_\theta$ plays an important role in the long-time behavior of solutions to the Hartree equation.
We analyze this behavior by studying the interaction profile $f(t)=\ee^{-\ii t\Delta/2}u(t)$.
If $f(t)$ converges to a fixed state $u_+$ in $H^s$,
we say that $u(t)$ scatters to the free solution
$\ee^{\ii t\Delta/2}u_+$, in the sense that
\[
    \|f(t)-u_+\|_{H^s}
    =\|u(t)-\ee^{\ii t\Delta/2}u_+\|_{H^s}\longrightarrow0
    \qquad\text{as }t\to\infty.
\]
This property is called \emph{linear scattering}.

For sufficiently regular and localized profiles, the stationary phase method yields the asymptotic formula
\begin{equation}\label{eqn:free-schrodinger-asymptotics}
 u(t,x)\approx (2\pi \ii t)^{-d/2}\ee^{\ii|x|^2/(2t)}
 \widehat f\bigl(t,x/t\bigr),
 \qquad t\to\infty.
\end{equation}
Substituting this expansion into the profile equation and extracting the leading order contribution, we formally obtain
\begin{equation}\label{eqn:profile-heuristic}
    \ii\partial_t\widehat f(t,\xi)
    =
    (2\pi)^{-d} t^{-\theta}
    \left(
      \int_{\mathbb R^d}
      |\xi-\eta|^{-\theta}
      |\widehat f(t,\eta)|^2\,d\eta
    \right)\widehat f(t,\xi)
    +R_f(t,\xi),
\end{equation}
where $R_f(t,\xi)$ is a remainder term denoting the difference between the exact profile equation and its leading term with sufficient time decay.

This heuristic computation suggests that the threshold between short-range and long-range interactions is determined by the time integrability of $t^{-\theta}$. When $\theta>1$, this factor is integrable at infinity, and linear scattering is expected. Scattering results in this regime have been established under assumptions on the regularity, localization, and size of the data that depend on the range of $\theta$ \cite{HayashiTsutsumi1987,Mochizuki1989OnSD,miao2008}.
On the other hand, $0<\theta\le1$ corresponds to the long-range regime, where  the leading time factor is no longer integrable, and nonexistence of ordinary scattering has been proved under appropriate hypotheses.
For the three-dimensional Coulomb interaction, Glassey \cite{Glassey1997a} proved that a solution cannot scatter in $L^2$ to a nonzero free state belonging to $L^1\cap L^2$.
Hayashi and Tsutsumi \cite{HayashiTsutsumi1987} subsequently excluded nontrivial $L^2$ scattering for suitably localized initial data throughout $0<\theta\le1$.
The borderline exponent $\theta=1$ is thus scattering-critical and corresponds to the physical Coulomb potential in three dimensions.

For the scalar Hartree equation in the long-range regime, one instead seeks convergence of the profile after a nonlinear correction. This is referred to as \emph{modified scattering}.
It typically involves a real-valued phase $\Phi(t,\xi)$ such that
\[
    \ee^{\ii\Phi(t,-\ii\nabla)}f(t)\longrightarrow u_+
    \qquad\text{as }t\to\infty
\]
in an appropriate function space, where $u_+$ is a suitable modified scattering state. For $\theta=1$, sufficiently small, regular, and localized solutions admit a logarithmic phase correction of the form
\begin{equation}\label{eq:phase}
    \Phi(t,\xi)
    =(2\pi)^{-d} (\log t)\bigl(V_1*|\widehat{u_+}|^2\bigr)(\xi).
\end{equation}
Modified scattering with this logarithmic phase correction was established by Hayashi and Naumkin \cite{Hayashi1998a} and by Kato and Pusateri \cite{Kato2011}; see also \cite{Nakanishi2002} for modified wave operators.

To explain the form of this correction, we introduce the slow time $\tau=\log t$. Formally neglecting the remainder in \eqref{eqn:profile-heuristic}, we obtain the scalar resonant equation for a profile $g$:
\begin{equation}\label{eq:resonance scalar}
    \ii\partial_\tau\widehat g
    = (2\pi)^{-d} (V_1*|\widehat g|^2)\widehat g.
\end{equation}
Since the potential is real-valued, $|\widehat g(\tau,\xi)|^2$ is independent of $\tau$. Choosing the resonant initial datum $\widehat g(0)=\widehat{u_+}$, we therefore obtain
\[
    \widehat g(\tau,\xi)
    =
    \ee^{-\ii\tau(2\pi)^{-d} (V_1*|\widehat{u_+}|^2)(\xi)}
    \widehat{u_+}(\xi).
\]
This explicit solution explains the logarithmic phase correction in \eqref{eq:phase}. Here the resonant initial datum is selected through the asymptotic analysis and need not coincide with the profile $\widehat f(1)$.

Below the critical exponent, modified wave operators with power-type phase corrections were constructed for $1/2<\theta<1$ in \cite{GinibreVelo2000,Nakanishi2002b}. Ginibre and Velo \cite{GinibreVelo2000-2} treated the full range $0<\theta\le1$ in dimensions $d\ge3$ by incorporating further terms in the asymptotic construction. Recently, Huang and Xie \cite{HX2026arxiv} proved modified scattering for the three-dimensional Hartree equation in the range $0 < \theta \le 1/2$ for small initial data with sufficient spatial decay. Their asymptotic phase contains finitely many correction terms, determined recursively.

\medskip

The same heuristic, based on the time integrability of the leading nonlinear term, extends to operator-valued Schr\"odinger equations. Pusateri and Sigal proved trace-norm scattering for weak self-interactions in a class of short-range time-dependent Kohn--Sham equations \cite{Pusateri2021}. For Coulomb interactions, modified scattering was established for the three-dimensional infinite-rank Hartree equation by Nguyen and You \cite{NguyenYou2025}, and for critical Kohn--Sham equations in dimensions two and three by the authors of the present paper \cite{KLYY2026arxiv}. These works treat infinite-rank density matrices, although their functional frameworks and scattering topologies differ. For Hartree--Fock systems with finitely many orbitals, short-range scattering was studied in \cite{Wada1999}, while Coulomb long-range scattering and modified wave operators were investigated in \cite{Wada2002,Ikeda2012}.

The purpose of the present work is to describe the long-range dynamics of the Hartree--Fock equation with Coulomb interaction for small and localized density matrices, without a finite-rank restriction. We now carry out the corresponding formal computation for the Hartree--Fock equation \eqref{eq:HF}, allowing the more general homogeneous interaction $V_\theta(x)=|x|^{-\theta}$, $0<\theta<d$.
We introduce the interaction profile
\[
    \Lambda(t)
    =\ee^{-\ii t\Delta/2}\gamma(t)\ee^{\ii t\Delta/2}.
\]
The same rescaling as in the scalar case formally gives
\begin{equation}\label{eqn:hf-profile-heuristic}
\begin{aligned}
    \ii\partial_t\Lambda(t)
    =&t^{-\theta}\mathcal F^{-1}
    \Big[
      V_\theta*\rho_{\mathcal F\Lambda(t)\mathcal F^{-1}}
      -\mathsf X_{\mathcal F\Lambda(t)\mathcal F^{-1}}^{(\theta)}, \, 
      \mathcal F\Lambda(t)\mathcal F^{-1}
    \Big]\mathcal F
    +\mathcal E_\Lambda(t),
\end{aligned}
\end{equation}
where $\mathsf X_A^{(\theta)}$ denotes the exchange operator with kernel $|\xi-\eta|^{-\theta}A(\xi,\eta)$. Here $\mathcal E_\Lambda$ is the remainder, which has to be shown to decay sufficiently rapidly. The spatial decay of the potential thus gives the same time factor $t^{-\theta}$ as in the scalar Hartree model, with both the direct and exchange terms contributing at this same order. Because of the exchange term, we have two additional issues for the asymptotic analysis.

The first issue is whether the cancellation between the direct and exchange terms removes the long-range obstruction to linear scattering.
The direct and exchange terms appear with opposite signs in the Hartree--Fock Hamiltonian, and the two contributions cancel exactly in the evolution of rank-one states. For long-range interactions, including the Coulomb potential, Proposition~\ref{prop:stationary-coulomb} characterizes the states for which the leading commutator vanishes: for a positive trace-class operator $A$ with bounded density,
\[
    [V_\theta*\rho_A-\mathsf X_A^{(\theta)},A]=0
    \quad\Longleftrightarrow\quad
    \operatorname{rank} A \le1, \quad \text{ when} \quad 0<\theta\le1.
\]
Under the hypotheses of Theorem~\ref{thm:no-linear-scattering}, including the optimal density decay $\|\rho_{\gamma(t)}\|_{L^\infty}\lesssim\langle t\rangle^{-d}$, convergence $\Lambda(t)\to\gamma_+$ in trace norm forces asymptotic stationarity:
\[
    [V_\theta*\rho_{\mathcal F\gamma_+\mathcal F^{-1}}-\mathsf X_{\mathcal F\gamma_+\mathcal F^{-1}}^{(\theta)}, \, \mathcal F\gamma_+\mathcal F^{-1}]=0.
\]
The density decay also ensures that $\rho_{\mathcal F\gamma_+\mathcal F^{-1}}$ is bounded, so the stationary-state characterization gives $\operatorname{rank}\gamma_+\le1$. On the other hand, preservation of the spectrum, together with convergence in trace norm, excludes this possibility when $\operatorname{rank}\gamma(0)\ge2$. Consequently, exchange does not restore linear scattering for solutions of rank at least two satisfying the optimal density decay bound. For finite-measure potentials in one dimension, however, cancellation can lead to linear scattering; see Mal\'ez\'e \cite{Maleze2025}.

The second issue concerns the effective dynamics that
describe the solution when linear scattering fails.
Here the nonlocal structure of exchange becomes essential. Whereas the direct Hartree term depends only on the density $\rho_\gamma(x)=\gamma(x,x)$, the exchange operator depends on the full integral kernel $\gamma(x,y)$, including its off-diagonal part. For the Hartree resonant equation, the density is conserved, so the associated multiplication potential is independent of time and yields an explicit phase correction. In Hartree--Fock, the leading Hamiltonian also contains an exchange operator coupled to the evolving kernel. The effective dynamics therefore cannot in general be described by a fixed scalar Fourier phase.

The explicit Hartree phase correction can be interpreted in terms of the scalar resonant flow \eqref{eq:resonance scalar}. For Hartree--Fock, we describe modified scattering through the full resonant evolution, retaining the leading commutator in \eqref{eqn:hf-profile-heuristic}.  In the Coulomb case $\theta=1$, introducing the logarithmic time $\tau=\log t$ and writing $\mathfrak{F}(\Gamma):=\mathcal F\,\Gamma\mathcal F^{-1}$,  we obtain the autonomous resonant equation
\begin{equation}\label{eqn:res-eqn-gamma1}
    \ii\partial_\tau\mathfrak{F}(\Gamma)
    =
    \bigl[
        V_1*\rho_{\mathfrak{F}(\Gamma)}
        -\mathsf X_{\mathfrak{F}(\Gamma)}^{(1)},
        \mathfrak{F}(\Gamma)
    \bigr],
\end{equation}
which admits global solutions with time-independent density $\rho_{\mathfrak{F}(\Gamma)}$; see Lemma~\ref{lem:resonant-GWP}. Thus, the asymptotic behavior is described by a nonlinear resonant flow on a logarithmic time scale. More precisely, there exists a resonant profile $\Gamma(\tau)$, with $\mathfrak{F}(\Gamma)$ solving \eqref{eqn:res-eqn-gamma1},
such that
\[
    \|\Lambda(t)-\Gamma(\log t)\|_{\mathfrak S^1}
    =
    \|\gamma(t)-\ee^{\ii t\Delta/2}\Gamma(\log t)\ee^{-\ii t\Delta/2}\|_{\mathfrak S^1}
    \longrightarrow 0
    \qquad\text{as }t\to\infty.
\]
This asymptotic description is also referred to as
\emph{modified scattering}. Modified scattering governed by nonlinear resonant dynamics has also been established for cubic Schr\"odinger equations on product spaces by Hani, Pausader, Tzvetkov, and Visciglia \cite{hani2015modified}, and under partial harmonic trapping by Hani and Thomann \cite{HaniThomann2016}.

\medskip

In Appendix~\ref{app:Subcritical HFE}, we also prove linear scattering for the Hartree--Fock equation with short-range interactions and small, localized initial data. Together with the nonexistence of linear scattering and the modified scattering result in the Coulomb case (Theorems~\ref{thm:no-linear-scattering} and~\ref{thm:main}), this identifies $\theta=1$ as the scattering-critical threshold suggested by the time factor $t^{-\theta}$. These results extend the classical scalar scattering theory to the infinite-rank operator framework.

\subsection{Main results}

To state the scattering result, we first introduce the resonant Hartree--Fock equation. Write $\mathfrak{F}(\Gamma) := \mathcal{F} \Gamma \mathcal{F}^{-1}$ and set $\rho_{\mathfrak{F}(\Gamma)}(k) := \mathfrak{F}(\Gamma)(k,k)$.
We recall that the direct potential is given by
\[
	\bigl(V* \rho_{\mathfrak{F}(\Gamma)} \bigr)(k) = \int_{\RR^d} V(k-\ell) \rho_{\mathfrak{F}(\Gamma)}(\ell) \, \dd \ell,
\]
and the exchange operator has the integral kernel
\[
	\mathsf{X}_{\mathfrak{F}(\Gamma)}(k,p) := V(k-p)\mathfrak{F}(\Gamma)(k,p).
\]
We write the resonant Hartree--Fock equation in logarithmic time $\tau=\log t$:
\begin{equation}\label{eqn:res-eqn-gamma}
	\ii\partial_\tau\mathfrak{F}(\Gamma) = \bigl[ V*\rho_{\mathfrak{F}(\Gamma)} - \mathsf X_{\mathfrak{F}(\Gamma)}, \mathfrak{F}(\Gamma) \bigr],\qquad \tau\ge0,
\end{equation}
see Section \ref{sec:halfdensity} for a discussion.

\begin{theorem}\label{thm:main}
	Let $V(x) = |x|^{-1}$, $d\ge2$ and $\sigma=\lfloor d/2\rfloor+1$. 
	There exist $\delta >0$ and $\epsilon_*>0$, depending only on $d$, with the following property.
	Let $\gamma_0\ge0$ and assume
	\[
		\epsilon:=\|\langle\nabla\rangle^\sigma\gamma_0 \langle\nabla\rangle^\sigma\|_{\mathfrak S^1} +\|\langle x\rangle^\sigma\gamma_0 \langle x\rangle^\sigma\|_{\mathfrak S^1} \le\epsilon_*.
	\]
	Then \eqref{eq:HF} has a unique nonnegative global solution $\gamma\in C([0,\infty);\mathfrak S^{\sigma,1})$ with $\gamma(0)=\gamma_0$.
	It also belongs to $C^1([0,\infty);\mathfrak S^1)$ and satisfies $|\mathrm J(t)|^\sigma\gamma(t)|\mathrm J(t)|^\sigma \in C([0,\infty);\mathfrak S^1)$. Its density satisfies
	\begin{equation}\label{time decay of rho}
		\|\rho_\gamma(t)\|_{L^p}\lesssim \epsilon\langle t\rangle^{-d(1-1/p)},\qquad 1\le p\le\infty, \quad t\ge0.
	\end{equation}
	Moreover, there exist a nonnegative $\Gamma_*\in\mathfrak S^1$ and a nonnegative $\Gamma\in C^1([0,\infty);\mathfrak S^1)$, with $\Gamma(0)=\Gamma_*$, such that $G(\tau):=\mathfrak F(\Gamma(\tau))$ solves \eqref{eqn:res-eqn-gamma} in $\mathfrak S^1$, $\rho_{G(\tau)}=\rho_{G(0)}\in L^1\cap L^\infty$ and $\operatorname{Tr}\Gamma(\tau)=\operatorname{Tr}\gamma_0$ for $\tau\ge0$, and
    \[
		\|\gamma(t)-\ee^{\ii t\Delta/2}\Gamma(\log t ) \ee^{-\ii t\Delta/2}\|_{\mathfrak S^1} \lesssim\epsilon t^{-3\delta},\qquad t\ge1.
	\]
\end{theorem}

\begin{remark}
For Theorem \ref{thm:main}, we remark the following.
\begin{enumerate}
\item
If $\operatorname{rank}\gamma_0=1$, the direct and exchange contributions cancel in the evolution equation, and hence
\[
    \gamma(t)
    =\ee^{\ii t\Delta/2}\gamma_0\ee^{-\ii t\Delta/2}.
\]
The corresponding resonant solution is stationary: $\Gamma(\tau)\equiv\gamma_0$.
Thus the modified scattering error is identically zero.
\item
We do not claim that the regularity and localization assumptions in Theorem~\ref{thm:main} are optimal. For the Kohn--Sham equation in dimensions two and three, \cite{KLYY2026arxiv} uses weighted control of the free profile together with an $L^\infty$ bound on the initial density, without a separate Sobolev regularity assumption. In the same dimensions, modified scattering for the scalar Hartree equation with small weighted $L^2$ data was proved in \cite{TimVanHoose2024arxiv}.

\item
The resonant equation \eqref{eqn:res-eqn-gamma} does not require a smallness assumption. For any nonnegative $G_0\in\mathfrak S^1$ with $\rho_{G_0}\in L^\infty$, it admits a global nonnegative solution. Indeed, the density is independent of time, so the bounds needed for global continuation remain uniform, as in the proof of Lemma~\ref{lem:resonant-GWP}.

\item
The decay exponent $d(1-1/p)$ in \eqref{time decay of rho} agrees with that of the density under the free Schr\"odinger evolution. This decay rate is sharp, as can already be seen from rank-one free solutions.

\end{enumerate}
\end{remark}

The following remark give us an example of a physically natural initial data of infinite rank satisfying the assumptions of Theorem \ref{thm:main}.
\begin{remark}[Thermal initial data]\label{rem:thermal-data}
Define the following operator, so called Gibbs measure,
	\[
	\gamma_0
	=\left(1+\ee^{\beta(\frac12(-\Delta+|x|^2)-\mu)}\right)^{-1},
	\qquad \beta>0,\quad \mu\in\RR.
	\]
	The eigenvalues of $\gamma_0$ are $(1+\ee^{\beta(n_1+\cdots+n_d+d/2-\mu)})^{-1}$ for non-negative integers $n_1, \dots,n_d$. So $\gamma_0$ has infinite rank. By spectral calculus, we have
	\[
	0\le\gamma_0
	\le\ee^{\beta\mu}\ee^{-\frac{\beta}{2}(-\Delta+|x|^2)}.
	\]
	Using the Fourier invariance of $\frac12(-\Delta+|x|^2)$ and Mehler's formula, we obtain
	\[
	\|\langle\nabla\rangle^\sigma\gamma_0\langle\nabla\rangle^\sigma\|_{\mathfrak S^1}
	+\|\langle x\rangle^\sigma\gamma_0\langle x\rangle^\sigma\|_{\mathfrak S^1}
	\le C_{\beta,d,\sigma}\ee^{\beta\mu},
	\]
	where the constant $C_{\beta,d,\sigma}$ is independent of $\mu$. Thus, for each fixed $\beta>0$, these initial data satisfy the assumptions of Theorem~\ref{thm:main} when $\mu$ is sufficiently negative.
\end{remark}

The long-range modification in Theorem~\ref{thm:main} is essential
when $\operatorname{rank}\gamma_0\ge2$, as the following theorem shows.
\begin{theorem}[Nonexistence of linear scattering] \label{thm:no-linear-scattering}
Let $V(x) = |x|^{-1}$, $d\ge2$, and let $\gamma\in C([0,\infty);\mathfrak{S}^1)$ be a global solution to \eqref{eq:HF} with initial data $\gamma_0$ satisfying 
\[
 \gamma_0 \ge0, \quad \operatorname{rank} \gamma_0\ge 2.
\]
Assume $\rho_\gamma(t)$ satisfies the time decay estimate
\begin{equation}\label{eq:optimal-density-decay}
\langle t\rangle^{d}\|\rho_\gamma(t)\|_{L^\infty}\lesssim 1.
	\end{equation}Then there is no $\gamma_+\in \mathfrak{S}^1$ such that 
\begin{equation}\label{eq:LS}
  \|\gamma(t)-\ee^{\ii t\Delta/2}\gamma_+ \ee^{-\ii t\Delta/2}\|_{\mathfrak S^1} \longrightarrow0\quad\text{as }\;t\to\infty.
\end{equation}
    \end{theorem}

\begin{remark}[Extension to $0<\theta<1$]
\label{rem:power-extension-intro}
The conclusion of Theorem~\ref{thm:no-linear-scattering} extends to $V(x)=V_\theta(x)=|x|^{-\theta}$, $0<\theta<1$, 
under the same assumptions
on the solution. More precisely, let $d\ge2$, $0<\theta<1$, and set
\[
 (\mathsf X_A^{(\theta)}f)(x)
 =\int_{\mathbb R^d}|x-y|^{-\theta}A(x,y)f(y)\,dy.
\]
Let $\gamma\in C([0,\infty);\mathfrak S^1)$ be a global mild solution of
\begin{equation}
 \ii\partial_t\gamma
 =[-\Delta/2+V_\theta*\rho_\gamma-\mathsf X_\gamma^{(\theta)},\gamma],
 \qquad \gamma(0)=\gamma_0,
 \label{eq:HF-theta}
\end{equation}
with $\gamma_0\ge0$ and $\operatorname{rank}\gamma_0\ge2$.
Assume also that $\gamma$ satisfies \eqref{eq:optimal-density-decay}. Then there is no
$\gamma_+\in\mathfrak S^1$ satisfying \eqref{eq:LS}.
The modifications of the proof are given in
Remark~\ref{rem:power-extension-proof}.
Here global existence and the density decay are assumed, rather than established. It would be interesting to construct global solutions satisfying
these assumptions.
\end{remark}

\subsection{Half-density formulation}\label{sec:halfdensity}

	We prove Theorem \ref{thm:main} using $\gamma = \kappa \kappa^{\ast}$, where the half-density $\kappa$ satisfies
	\[
	\ii \partial_t \kappa = H_{\kappa\kappa^{\ast}} \kappa.
	\]
	To describe the asymptotic behavior of the profile $\lambda(t):= \ee^{-\ii t \Delta_x/2}\kappa(t)$, we introduce the resonant equation at the level of half-density as
	\begin{equation}\label{eqn:res-eqn-00}
		\begin{aligned}
			\ii \partial_t \widehat{\nu}(t,k,y) &= (2\pi)^{-d} t^{-1} \bigl(V *_k \|\widehat{\nu}(t,\cdot)\|_{L^2_y}^2\bigr)(k) \widehat{\nu}(t,k,y) \\
			&\quad - (2\pi)^{-d} t^{-1} \int_{\mathbb R^d}\int_{\mathbb R^d} V(k-\ell) \overline{\widehat{\nu}(t,\ell,a)} \widehat{\nu}(t,\ell,y) \widehat{\nu}(t,k,a) \,\dd a\dd \ell,
		\end{aligned}
	\end{equation}
	where $V(k) =|k|^{-1}$.
	For each fixed $k$, the operator in the direct term is a scalar multiple of the identity on $L_y^2$, while the operator in the exchange term is self-adjoint on $L_y^2$.
	The equation \eqref{eqn:res-eqn-00} captures the leading contribution of order $t^{-1}$ to the profile equation in Fourier space.

	For $\nu=\mathcal F^{-1}\widehat\nu$, define
    \[
        \Gamma(\tau)=\nu(\ee^\tau)\nu(\ee^\tau)^*,\qquad \tau\ge0.
    \]
    By our convention, for $t \ge 1$, 
    \[
        \mathfrak F(\Gamma(\log t))
        =(2\pi)^{-d}\widehat\nu(t)\widehat\nu(t)^*.
    \]
    If $\widehat\nu$ solves \eqref{eqn:res-eqn-00}, the change of variables
    $\tau=\log t$ shows that $\mathfrak F(\Gamma)$ solves
    \eqref{eqn:res-eqn-gamma}.
	We first establish the following theorem for $\kappa$ and deduce Theorem \ref{thm:main} from it.

    We use the Banach space
    \begin{equation}\label{def:X}
        X:= L^2_{k,y} \cap C_b(\mathbb{R}^d; L^2_y), \qquad \| f \|_X:= \| f \|_{L^2_{k,y}} + \sup_{k\in\RR^d} \| f(k)\|_{L^2_y}.
    \end{equation}
    Here $C_b(\RR^d;L^2_y)$ denotes the space of bounded continuous $L^2_y$-valued functions, with the supremum norm.

	\begin{theorem}\label{thm:main-kappa}
		Let $V(x) = |x|^{-1}$, $d\ge2$ and $\sigma=\lfloor d/2\rfloor+1$. Take $\delta>0$ as in Theorem~\ref{thm:main}.
		Assume
		\[
		\|\langle\nabla\rangle^\sigma\kappa_0\|_{\mathfrak S^2} +\|\langle x\rangle^\sigma\kappa_0\|_{\mathfrak S^2} \le\epsilon_0
		\]
		with $\epsilon_0$ sufficiently small.
		Then \eqref{eqn:kappa} has a unique global solution $\kappa\in C([0,\infty);H_x^\sigma L_y^2)$ with $\kappa(0)=\kappa_0$ and $|\mathrm J(t)|^\sigma\kappa(t)\in C([0,\infty);\mathfrak S^2)$.
		Moreover, it satisfies
		\[
		\|\kappa(t)\|_{L_x^pL_y^2} \lesssim\epsilon_0\langle t\rangle^{-d(1/2-1/p)}
		\]
        for $2 \le p \le \infty$ and $t \ge 0$.
        In addition, there exists $\widehat\nu_1\in X$ with $\|\widehat\nu_1\|_X\lesssim\epsilon_0$ such that the resonant equation \eqref{eqn:res-eqn-00} has a unique global-in-time solution $\widehat\nu\in C^1([1,\infty);X)$ with $\widehat{\nu}(1)=\widehat{\nu}_1$. The equation holds in $L^2_y$ for every $k\in\mathbb{R}^d$ and $t\ge 1$. With $\nu(t)=\mathcal F^{-1}\widehat\nu(t)$, we have
		\[
		\|\widehat\lambda(t)-\widehat\nu(t)\|_X \lesssim\epsilon_0t^{-3\delta},\qquad \|\kappa(t)-\ee^{\ii t\Delta/2}\nu(t)\|_{\mathfrak S^2} \lesssim\epsilon_0t^{-3\delta},\qquad t\ge1.
		\]
	\end{theorem}

    The equation for $\widehat\lambda$ contains the oscillatory
    factor $\ee^{\ii t(k-r)\cdot\ell}$. For the Coulomb potential, we use Fourier inversion to write
    \[
        (2\pi)^{-d}\int_{\mathbb R^d}
        \ee^{\ii t(k-r)\cdot\ell}\widehat V(\ell)\,\dd\ell
        =
        V(t(k-r))
        =
        t^{-1}V(k-r).
    \]
    Formally, setting $\ell=0$ in the profile factors leads to 
    \eqref{eqn:res-eqn-00}. We will also use the form
    \[
        \ii\partial_t\widehat\nu(t,k)
        =
        t^{-1}A[\widehat\nu(t)](k)\widehat\nu(t,k),
    \]
    where $A[f](k)$ is defined in \eqref{def:A}.

    Put $G(\tau)=\mathfrak F(\Gamma(\tau))$. We can write
    \eqref{eqn:res-eqn-00} as
    \[
        \ii\partial_t\widehat\nu(t)
        =t^{-1}\bigl(V*_k\rho_{G(\log t)}-\mathsf X_{G(\log t)}\bigr)
        \widehat\nu(t).
    \]
    Using $G(\log t)=(2\pi)^{-d}\widehat\nu(t)\widehat\nu(t)^*$
    and $\partial_\tau=t\partial_t$, we obtain \eqref{eqn:res-eqn-gamma}.

    Finally, by the Schatten--H\"older inequality, 
    \[
    \begin{aligned}
    \|\gamma(t)
    -\ee^{\ii t\Delta_x/2}\Gamma(\log t)\ee^{-\ii t\Delta_x/2}
    \|_{\mathfrak S^1}
    \le
    \bigl(\|\kappa(t)\|_{\mathfrak S^2}
    +\|\nu(t)\|_{\mathfrak S^2}\bigr)
    \|\kappa(t)-\ee^{\ii t\Delta_x/2}\nu(t)\|_{\mathfrak S^2}.
    \end{aligned}
    \]
    Thus the scattering estimate for $\gamma$ follows from
    Theorem~\ref{thm:main-kappa}.

\subsubsection{Gauge freedom in the half-density formulation}
	Let $\gamma \ge 0$ be a trace-class operator and choose a half-density $\kappa \in \mathfrak{S}^2$ such that $\gamma = \kappa \kappa^*$.
	This factorization is not unique.
	For any unitary operator $U$, the half-density $\kappa U$ gives the same density matrix, since
    $(\kappa U)(\kappa U)^* = \kappa U U^* \kappa^* = \gamma$.
	We refer to this as gauge freedom.

	We first describe the gauge transformation formally. Let $U(t)$ be a family of unitary operators satisfying
	\begin{equation}\label{eq:gauge-ev}
		\ii \partial_t U^* = A(t) U^* \quad \Longleftrightarrow \quad \ii \partial_t U = -U A(t)
	\end{equation}
    where $A(t) = A(t)^{\ast}$.
	By the product rule, right multiplication by $U(t)$ transforms the left-gauge equation into the general gauge equation
    \begin{equation}\label{eq:general-gauge}
        \ii\partial_t\kappa = H_\gamma\kappa-\kappa A(t), \qquad \gamma=\kappa\kappa^*.
    \end{equation}
    Since $U(t)$ is unitary, this transformation preserves the density matrix.
	Using $H_{\gamma}^{\ast} = H_{\gamma}$ and $A(t)^{\ast} = A(t)$, we compute
	\begin{equation}\label{eq:gamma-solves-HF}
		\ii \partial_t \gamma = (\ii \partial_t \kappa)\kappa^* - \kappa (\ii \partial_t \kappa)^* = (H_\gamma \kappa - \kappa A)\kappa^* - \kappa (H_\gamma \kappa - \kappa A)^* = H_\gamma \gamma - \gamma H_\gamma.
	\end{equation}
	Thus $\gamma$ satisfies the time-dependent Hartree--Fock equation \eqref{eq:HF} independently of the choice of $A(t)$.

\subsubsection{The left gauge}
We now specify the gauge used for the estimates in the paper.

We call the choice $A(t) =0$ the \textit{left gauge}. 
	Then
	\begin{equation}\label{eqn:kappa}
		\ii\partial_t\kappa=H_\gamma\kappa,
	\end{equation}
where the commutator structure is removed.	
	At the kernel level, this becomes
	\begin{align}\label{eqn:kappa-kernel}
		\ii\partial_t\kappa(t,x,y) &= -\frac12\Delta_x\kappa(t,x,y) + (V*\rho_{\kappa\kappa^*})(t,x)\kappa(t,x,y) - ({\mathsf{X}}_{\kappa\kappa^*}\kappa)(t,x,y).
	\end{align}
	Here
	\[
	\rho_{\kappa\kappa^*}(t,x) = \int_{\RR^d}|\kappa(t,x,a)|^2\,\dd a
	\]
	and
	\[
	({\mathsf{X}}_{\kappa\kappa^*}\kappa)(t,x,y) = \int_{\RR^d} V(x-z)\gamma(t,x,z)\kappa(t,z,y)\,\dd z.
	\]
	The kernel of $\gamma = \kappa \kappa^{\ast}$ is given by
	\[
	\gamma(t,x,z) = \int_{\RR^d} \kappa(t,x,a)\overline{\kappa(t,z,a)}\,\dd a.
	\]
    Substituting this into the exchange term, we obtain
	\begin{equation}\label{eq:exchange-left}
		({\mathsf{X}}_{\kappa\kappa^*}\kappa)(t,x,y) = \int_{\RR^d}\int_{\RR^d} V(x-z) \kappa(t,x,a) \overline{\kappa(t,z,a)} \kappa(t,z,y) \,\dd a\,\dd z.
	\end{equation}
	Thus the equation closes in $\kappa$.

	In this gauge, the linear Schr\"odinger evolution acts only in $x$, so
	$\kappa$ satisfies an $L^2_y$-valued Schr\"odinger equation.
A similar half-density formulation has proved useful for
extending scalar dispersive techniques to the operator
setting of the Kohn--Sham equation \cite{KLYY2026arxiv}.

\begin{remark}
One may also choose $A(t)=H_\gamma(t)$, which yields
a nonlinear von Neumann-type equation for the half-density:
\[
    \ii\partial_t\kappa=[H_\gamma,\kappa].
\]
This commutator formulation is the Hartree--Fock analogue of the half-density evolution used by Pusateri and Sigal \cite[Section~3.1]{Pusateri2021}.
\end{remark}

\begin{remark}
A related square-root formulation is useful for kinetic equations.
For instance, Ionescu, Pausader, Wang, and Widmayer \cite{IPWW2022}
introduce $\mu=\sqrt f$ and use $L^2$-based energy estimates to prove small-data modified scattering for the
three-dimensional Vlasov--Poisson system.
There, the identity $\|\mu\|_{L^2}^2=\|f\|_{L^1}$
plays a role analogous to the trace identity above.
\end{remark}

\subsubsection{Why the half-density formulation is useful}

Although $\gamma$ is trace class and hence Hilbert--Schmidt,
its density is not, in general, equal to
$\|\gamma(x,\cdot)\|_{L^2_y}^2$.
The factorization $\gamma=\kappa\kappa^*$ instead gives
\[
    \rho_\gamma(x)=\|\kappa(x,\cdot)\|_{L^2_y}^2,
    \qquad
    \operatorname{Tr}\gamma=\|\kappa\|_{\mathfrak S^2}^2,
\]
where the first identity holds for almost every $x$.
Thus positivity is built into the factorization, and the density
can be estimated through mixed norms of $\kappa$ while working
in the Hilbert space of Hilbert--Schmidt operators.

The exchange term also admits an $L^2_y$-valued formulation.
Suppressing the time variable, set $u(x):=\kappa(x,\cdot)$ and write
\[
    (|g\rangle\langle g|)h
    :=g\langle g,h\rangle_{L^2_y}
    =g\int_{\mathbb R^d}\overline{g(a)}h(a)\,\dd a.
\]
For $V(x)=|x|^{-1}$ and $d\ge2$, define
\[
    M_\kappa(x)
    :=\int_{\mathbb R^d}
        V(x-z)|u(z)\rangle\langle u(z)|\,\dd z.
\]
Since
$\||u(z)\rangle\langle u(z)|\|_{\mathfrak S^1(L^2_y)}
=\|u(z)\|_{L^2_y}^2$,
the assumptions $\rho_\gamma\in L^1$ and
$V\in L^1+L^\infty$ ensure that this integral converges
in $\mathfrak S^1(L^2_y)$ for almost every $x$.
It defines a positive self-adjoint operator satisfying
\[
    \|M_\kappa(x)\|_{\mathcal B(L^2_y)}
    \le (V*\rho_\gamma)(x).
\]
Moreover, since
$\gamma(x,z)=\langle u(z),u(x)\rangle_{L^2_y}$,
the kernel formula \eqref{eq:exchange-left} gives
\[
    (\mathsf X_\gamma\kappa)(x,\cdot)
    =\int_{\mathbb R^d}
        V(x-z)u(z)\langle u(z),u(x)\rangle_{L^2_y}\,\dd z
    =M_\kappa(x)u(x).
\]
Consequently,
\[
    \|(\mathsf X_\gamma\kappa)(x,\cdot)\|_{L^2_y}
    \le (V*\rho_\gamma)(x)\|\kappa(x,\cdot)\|_{L^2_y}
\]
for almost every $x$.
Thus the direct and exchange terms are controlled by the same
scalar Hartree potential.
The trilinear version of this estimate, established in the proof
of Lemma~\ref{lem:exchange-term}, is used in the nonlinear
and weighted energy estimates.

The same operator-valued structure appears in the resonant
generator $A[f](k)$ defined in \eqref{def:A}.
For $f\in X$, it takes the form
\[
    A[f](k)
    =(2\pi)^{-d}
    \left[
        \bigl(V*\|f(\cdot)\|_{L^2_y}^2\bigr)(k)I
        -\int_{\mathbb R^d}
            V(k-\ell)|f(\ell)\rangle\langle f(\ell)|\,\dd\ell
    \right],
\]
where $I$ is the identity on $L^2_y$.
In particular, for every $k\in\mathbb R^d$ and $h\in L^2_y$,
\[
    \|A[f](k)h\|_{L^2_y}
    \lesssim
    \bigl(V*\|f(\cdot)\|_{L^2_y}^2\bigr)(k)\|h\|_{L^2_y},
    \qquad
    A[f](k)^*=A[f](k).
\]
The bound gives the local Lipschitz estimates needed
for the resonant evolution and the final-state construction.
The self-adjointness of $A[f](k)$ yields conservation of
$\|\widehat\nu(t,k)\|_{L^2_y}$ along the resonant flow
and unitarity of the associated propagator.
This conservation law allows the resonant solution to be extended globally.
These properties are used to construct the modified profile and solve the final-state problem in the proof of modified scattering.

\begin{remark}\label{rmk:infinite-system}
An alternative formulation is obtained from the spectral
decomposition
\[
    \gamma_0=\sum_{j=1}^{\infty}
    a_j|u_{0,j}\rangle\langle u_{0,j}|,
\]
where $\{u_{0,j}\}_{j\ge1}$ is an orthonormal system in
$L^2(\mathbb R^d)$, $a_j\ge0$, and $\sum_j a_j<\infty$.
The corresponding infinite coupled system is
\[
\begin{aligned}
    \ii\partial_t u_j
    &=-\frac12\Delta u_j
      +\sum_{k=1}^{\infty}a_k
      \left[
        (V*|u_k|^2)u_j
        -(V*(u_j\overline{u_k}))u_k
      \right],\\
    u_j(0)&=u_{0,j},
    \qquad j\ge1,
\end{aligned}
\]
with
$\gamma(t)=\sum_{j=1}^{\infty}
a_j|u_j(t)\rangle\langle u_j(t)|$.
Wada \cite{Wada2002,Wada1999} studied the long-time behavior of solutions to finite coupled systems of this type.
The half-density formulation used here provides a systematic framework for analyzing the dynamics directly at the operator level, without introducing an explicit spectral decomposition.
\end{remark}

\subsection{Strategy of proof}

Our analysis combines the half-density formulation with the space-time resonance approach of \cite{Kato2011}.
The bootstrap argument also builds on the energy and decay estimates for the infinite-rank Hartree equation in \cite{NguyenYou2025}.
We regard $\kappa(t,x,\cdot)$ as an $L_y^2$-valued function of $x\in\mathbb R^d$, keeping the auxiliary variable $y$ in $L^2$ throughout the analysis.
This allows us to extend the scalar Hartree bootstrap
framework to the corresponding Hilbert-space-valued spaces.
Thus, the bootstrap norm \eqref{def:bootstrapnorm} combines
the $L^2_{x,y}$ norm, Sobolev regularity in $x$, spatial
moments of the interaction profile
$\lambda(t)=\ee^{-\ii t\Delta_x/2}\kappa(t)$ measured through
$\mathrm J_x(t)$, and dispersive decay in $L_x^\infty L_y^2$.
The Sobolev and weighted norms are allowed to grow like
$t^\delta$, while the dispersive norm retains the optimal decay rate $t^{-d/2}$.

To obtain the dispersive estimate needed to close the bootstrap, we analyze the interaction profile in Fourier space.
Following the heuristic decomposition in
\eqref{eqn:hf-profile-heuristic}, we separate the leading resonant contribution of order $t^{-1}$ from the remainder:
\[
    \ii\partial_t\widehat\lambda(t,k,y)
    =
    t^{-1}\mathcal R[\widehat\lambda](t,k,y)
    +\mathcal E[\widehat\lambda](t,k,y),
\]
where $\widehat\lambda$ denotes the Fourier transform of $\lambda$ in the $x$ variable. Both the direct and exchange terms contribute at this order; see Subsection~\ref{sec:Decomp of profile eqn}. Here $\mathcal R[f](k)=A[f](k)f(k)$, where $A[f](k)$ is the self-adjoint operator on $L_y^2$ defined in \eqref{def:A}, and $\mathcal E$ denotes the remainder. Under the bootstrap assumption $\mathcal B_T(\kappa)\lesssim\epsilon_0$, the weighted estimates give
\[
    \|\mathcal E[\widehat\lambda](t)\|_X
    \lesssim\epsilon_0^3t^{-1-3\delta},
    \qquad 1\le t\le T,
\]
where $X$ is defined in \eqref{def:X}.
In particular, this bound is integrable in time.
Since $A[\widehat\lambda(t)](k)$ is self-adjoint, the associated propagator is unitary on $L_y^2$ at each frequency and preserves the $X$ norm. Duhamel's formula therefore gives a uniform-in-time bound
for the Fourier profile in $X$. Together with the weighted estimates, this bound yields the optimal dispersive decay and closes the bootstrap for sufficiently small $\epsilon_0$. We thus obtain a global solution, and the remainder estimate holds for all $t\ge1$.

For the resonant equation \eqref{eqn:res-eqn-00}, global well-posedness for any initial datum in $X$ follows directly from the local Lipschitz bounds and a conservation law. Indeed, self-adjointness preserves the $L_y^2$ norm at each frequency. Consequently, both the $L^2_{k,y}$ and $L^\infty_kL_y^2$ norms, and thus the $X$ norm, are conserved. This conservation law then extends the local solution globally.

To remove the leading evolution from the full profile equation, we use the unitary propagator $\mathcal U_{\widehat\lambda}(t,s;k)$ generated by the leading-order operator $t^{-1}A[\widehat\lambda(t)](k)$, with $\mathcal U_{\widehat\lambda}(s,s;k)=I$ (see \eqref{def:Uf-propagator}). Suppressing the frequency variable, set
\[
    \widehat\omega(t)
    =
    \mathcal U_{\widehat\lambda}(1,t)\widehat\lambda(t).
\]
Then
\[
    \ii\partial_t\widehat\omega(t)
    =
    \mathcal U_{\widehat\lambda}(1,t)
    \mathcal E[\widehat\lambda](t).
\]
Since the propagator preserves the $X$ norm and the remainder is integrable, $\widehat\omega(t)$ converges to some $\widehat\omega_\infty\in X$, with
\[
    \|\widehat\omega(t)-\widehat\omega_\infty\|_X
    \lesssim\epsilon_0^3t^{-3\delta}.
\]
This propagator plays the role of the scalar phase correction discussed above.
The resulting approximation $\mathcal U_{\widehat\lambda}(t,1)\widehat\omega_\infty$ is generated by the full profile $\widehat\lambda$ and
need not itself solve the resonant equation.

For the Hartree equation, one passes from this approximation to an exact resonant solution by separating the accumulated
phase into a logarithmic term and a convergent correction. After absorbing the limiting phase correction into the
asymptotic profile, one obtains the explicit resonant solution described in \eqref{eq:resonance scalar}. This is possible because the resonant Hartree potential depends only on the conserved density, given by $|\widehat g|^2$ in the scalar case.

For Hartree--Fock, conservation of the density does not determine the exchange operator, which also depends on the evolving off-diagonal kernel. Thus, this scalar phase argument does not in general
provide an exact resonant solution. We instead construct such a solution through a final-state problem. Write
\[
    \widehat\nu(t)
    =
    \mathcal U_{\widehat\lambda}(t,1)\widehat\eta(t),
    \qquad
    \widehat\eta(t)\longrightarrow\widehat\omega_\infty
     \quad \text{in } X \text{ as }t\to\infty.
\]
The resonant equation for $\widehat\nu$ is then equivalent to
\[
    \widehat\eta(t)
    =\widehat\omega_\infty+\ii\int_t^\infty s^{-1}
    \mathcal U_{\widehat\lambda}(1,s)
    \bigl(A[\widehat\nu(s)]-A[\widehat\lambda(s)]\bigr)
    \widehat\nu(s)\,\dd s.
\]
The difference estimates for $A$, the convergence rate of $\widehat\omega$, and the smallness of the data allow us to solve this integral equation by a contraction argument in a time-weighted $X$ space.
This determines the resonant initial datum
$\widehat\nu(1)=\widehat\eta(1)$ and gives
\[
    \|\widehat\lambda(t)-\widehat\nu(t)\|_X
    \lesssim\epsilon_0t^{-3\delta}.
\]
Finally, setting
$\nu(t)=\mathcal F^{-1}\widehat\nu(t)$ and
$\Gamma(\tau)=\nu(\ee^\tau)\nu(\ee^\tau)^*$, with $\tau=\log t$, we obtain the resonant density matrix $\mathfrak F(\Gamma)$ solving \eqref{eqn:res-eqn-gamma}. The half-density approximation then yields the trace-norm modified scattering statement by the Schatten--H\"older inequality.

\medskip

For the nonexistence of linear scattering
(Theorem~\ref{thm:no-linear-scattering}), we adapt Glassey's long-range contradiction argument \cite{Glassey1997a} to the operator setting. Under the optimal density decay bound, linear scattering would allow us to replace the evolving profile by its limit in the leading $t^{-1}$ commutator.
Convergence of the profile then requires this limiting commutator to vanish.
Proposition~\ref{prop:stationary-coulomb} excludes such stationarity for nonnegative states of rank at least two:
a localized virial identity, together with homogeneity, forces the Hartree--Fock interaction energy of a stationary
state to vanish, whereas this energy is strictly positive for such states.

	\subsection{Notation}
	Our convention for the \emph{Fourier transform} is
	\[
	(\mathcal{F}_x V)(k)\equiv\widehat V(k):=\int_{\mathbb R^d}\ee^{-\ii k\cdot x}V(x)\,\dd x, \qquad V(x)=\frac1{(2\pi)^d}\int_{\mathbb R^d}\ee^{\ii k\cdot x}\widehat V(k)\,\dd k .
	\]

	For an integral operator $A$ with kernel $A(x,y)$, define $\widehat{A}$ by $\widehat{A}f= \mathcal{F} (Af)$ for $f\in L^2(\mathbb{R}^d)$.
	Its integral kernel is given by
	\begin{equation}\label{def:Fourier-right-variable}
		\widehat{A}(k,y)= \int_{\RR^d} \ee^{-\ii k \cdot x} A(x,y)  \, \dd x .
	\end{equation}
	For kernels of operators, we often omit the subscript $x$ and write $\mathcal F$ for the Fourier transform in the left (output) variable only.

	For an integral operator $A$ with kernel $A(x,y)$, let $\mathfrak{F}(A):= \mathcal{F} A\mathcal{F}^{-1}$.
	Its integral kernel is given by
	\begin{equation}\label{def:Fourier-operator}
		(\mathfrak{F}(A))(k,p)= (2\pi)^{-d}\int_{\RR^d \times \RR^d} \ee^{-\ii k \cdot x + \ii p \cdot y } A(x,y)  \, \dd x \, \dd y.
	\end{equation}

	Throughout the paper, all complex inner products on Hilbert spaces are taken to be linear in the second argument.
	In particular,
	\[
	\langle f,g\rangle_{L^2_{x,y}} := \iint \overline{f(x,y)}g(x,y)\, \dd x\, \dd y.
	\]

    For a multi-index $\beta$, write
    \[
        \mathrm{J}_x (t):= x + \ii t \nabla_x , \qquad \mathrm{J}_x^{\beta}(t) := \ee^{\ii t \Delta_x/2} x^{\beta} \ee^{-\ii t \Delta_x/2},
    \]
    which are obtained by conjugating the position operators along the free Schr\"odinger flow.

	For $m\ge0$, we introduce the operators
	\[
	|\mathrm J_x (t)|^m:= \ee^{\ii t\Delta_x/2}\, |x|^m \,\ee^{-\ii t\Delta_x/2} \quad \text{and} \quad \langle\mathrm J_x (t)\rangle^m:= \ee^{\ii t\Delta_x/2}\, \langle x\rangle^m \,\ee^{-\ii t\Delta_x/2}.
	\]

    For a Banach space $E$, we denote by $\mathcal B(E)$ the space of
    bounded linear operators on $E$, equipped with the operator norm.
    For $B\in\mathcal B(L_y^2)$, we write
    \[
        \|B\|_{\mathcal B(L_y^2)}=\|B\|_{\mathrm{op}}
        :=\sup_{\|h\|_{L_y^2}=1}\|Bh\|_{L_y^2}.
    \]

	For $1\le p <\infty$, we denote by $\mathfrak{S}^{m,p}$ the Schatten--Sobolev class on $L^2(\mathbb{R}^d)$, with norm
	\[
	\| A \|_{\mathfrak{S}^{m,p}} = \| \langle \nabla \rangle^m A \langle \nabla \rangle^m \|_{\mathfrak{S}^p}.
	\]

	\subsection{Organization of the paper}
    Section~2 presents the basic estimates and local theory. In Section~3, we derive the weighted energy bounds. Section~4 identifies the resonant equation and constructs the associated unitary propagator, and Section~5 estimates the remainder and proves the bootstrap estimate. In Section~6, we solve the resonant final-state problem and prove modified scattering. Section~\ref{sec:no-scattering} classifies stationary solutions of the resonant equation and proves the absence of linear scattering. Appendix~\ref{app:Subcritical HFE} treats short-range interactions.

\subsection*{Acknowledgment}
M. Kawamoto was partially supported by JSPS KAKENHI Grant Numbers JP24K06796 and JP26K00612. 
J. Lee was partially supported by the Global-Learning \& Academic research institution for Master's-PhD students, and Postdocs (G-LAMP) Program of the National Research Foundation of Korea (NRF) grant funded by the Ministry of Education (No. RS-2025-25442355), the NRF grant funded by the Korea government (MSIT) (No. RS-2026-25589357), and a grant from Kyung Hee University in 2026 (KHU-20262263). C.Yang was supported by the National Research Foundation of Korea(NRF) grant funded by the Korea government(MSIT) (RS-2026-25587425).

\subsection*{AI disclosure}
For the present work, we used AI-tools GPT-6 Astra and GPT-6.1 Sol for proofreading and language editing. The authors reviewed all revisions independently, and take full responsibility for the mathematical content of the manuscript.

\section{Preliminaries}
\subsection{Some Lemmas}
	\begin{lemma}\label{lem:decay}
		Let $\lambda(t) = \ee^{-\ii t \Delta_x/2 } \kappa(t)$ be the profile of $\kappa(t)$.
		Fix $ 0 < \theta \le 1$ and $m > d/2 + 2\theta$.
		Then we have
		\[
		\| \kappa(t) \|_{L^{\infty}_x L^2_y} \lesssim t^{-d/2} \| \widehat{\lambda}(t) \|_{L^{\infty}_k L^2_y} + t^{-d/2 - \theta} \| \langle x \rangle^m \lambda(t) \|_{L^2_{x,y}},
		\]
		uniformly for $t \ge 1$.
	\end{lemma}
	\begin{proof}
		For $t>0$, we use the free Schr\"odinger kernel to write
		\[
		\begin{aligned}
			\kappa(t,x,y)&=(2\pi\ii t)^{-d/2}\ee^{\ii|x|^2/(2t)} \left (\widehat\lambda(t,x/t,y) +\int\ee^{-\ii x\cdot z/t} \bigl(\ee^{\ii|z|^2/(2t)}-1\bigr)\lambda(t,z,y)\,\dd z\right).
		\end{aligned}
		\]
        For $0\le\theta\le1$, we have
        $|\ee^{\ii|z|^2/(2t)}-1|\lesssim
        t^{-\theta}|z|^{2\theta}$. Hence
        \[
        \begin{aligned}
        \|\kappa(t)\|_{L_x^\infty L_y^2}
        &\lesssim t^{-d/2}
        \|\widehat\lambda(t)\|_{L_k^\infty L_y^2}
        +t^{-d/2-\theta}
        \||x|^{2\theta}\lambda(t)\|_{L_x^1L_y^2}\\
        &\lesssim t^{-d/2}
        \|\widehat\lambda(t)\|_{L_k^\infty L_y^2}
        +t^{-d/2-\theta}
        \|\langle x\rangle^m\lambda(t)\|_2,
        \end{aligned}
        \]
        where the last inequality follows from Cauchy--Schwarz since
        $m>d/2+2\theta$.
	\end{proof}

    \begin{lemma}\label{lem:potential}
        Let $d\ge 2$, $V(x) = |x|^{-1}$, and $f\in L^1(\mathbb{R}^d) \cap L^{\infty}(\mathbb{R}^d)$. Then
        \[
            \|V*f\|_{L^\infty}
            \lesssim \|f\|_{L^1}^{1-1/d}\|f\|_{L^\infty}^{1/d}.
        \]
    \end{lemma}

    \begin{proof}
        We may assume $f \neq 0$. For $R>0$, write
        \[
            V=V_{<R}+V_{\ge R},\qquad
            V_{<R}:=V\mathbf1_{\{|x|<R\}},\qquad
            V_{\ge R}:=V\mathbf1_{\{|x|\ge R\}}.
        \]
        Since $\|V_{<R}\|_{L^1}\lesssim R^{d-1}$ and
        $\|V_{\ge R}\|_{L^\infty}=R^{-1}$, Young's inequality gives
        \[
            \|V*f\|_{L^\infty}
            \lesssim R^{d-1}\|f\|_{L^\infty}+R^{-1}\|f\|_{L^1}.
        \]
        Choosing $R=(\|f\|_{L^1}/\|f\|_{L^\infty})^{1/d}$ proves the estimate.
    \end{proof}

    For two integral kernels $f$ and $g$, define
    \[
        \rho_{fg^*}(x) := \int_{\RR^d} f(x,a) \overline{g(x,a)}  \, \dd a.
    \]
	We next control the nonlinear direct term.
	\begin{lemma}[Direct term]
		\label{lem:direct}
		Let $d \ge 2$.
		For $f, g \in (L^2_x \cap L^{\infty}_x) L^2_y$ and $h \in L^2_{x,y}$, we have
		\begin{equation}\label{est:direct-0}
			\| (V \ast \rho_{fg^{\ast}})h \|_{L^2_{x,y}} \lesssim \| f \|_{L^2_{x,y}}^{(d-1)/d} \| g \|_{L^2_{x,y}}^{(d-1)/d} \| f \|_{L^{\infty}_xL^2_y}^{1/d} \| g \|_{L^{\infty}_x L^2_y}^{1/d} \|h \|_{L^2_{x,y}}.
		\end{equation}
		Let $\sigma = \lfloor d/2 \rfloor +1$.
		If, in addition, $f,g,h \in H^{\sigma}_x L^2_y$, then
		\begin{equation}\label{est:direct-Hsigma}
			\| (V \ast \rho_{fg^{\ast}})h \|_{H^{\sigma}_x L^2_y} \lesssim \| f \|_{H^{\sigma}_x L^2_y} \| g \|_{H^{\sigma}_x L^2_y} \| h \|_{H^{\sigma}_x L^2_y}.
		\end{equation}
	\end{lemma}

\begin{proof}
    Set $F(x)=\|f(x,\cdot)\|_{L^2_y}$ and
    $G(x)=\|g(x,\cdot)\|_{L^2_y}$. Since $|\rho_{fg^*}(x)|\le F(x)G(x)$,
    Lemma~\ref{lem:potential} and Cauchy--Schwarz imply \eqref{est:direct-0}.

    Write $\rho=\rho_{fg^*}$. Since $\sigma>d/2$, we use the algebra property
    of $H^\sigma$, Leibniz's rule, and Cauchy--Schwarz to estimate
    \[
        \|\rho\|_{H^\sigma_x}+\|\rho\|_{W^{\sigma,1}_x}
        \lesssim \|f\|_{H^\sigma_xL^2_y}\|g\|_{H^\sigma_xL^2_y}.
    \]
    Noting that $V_{<1} \in L^1$ and $V_{\ge 1} \in L^{\infty}$, Young's inequality implies
    \[
    \begin{aligned}
        \|V_{<1}*\rho\|_{H^\sigma_x}
        &\le \|V_{<1}\|_{L^1}\|\rho\|_{H^\sigma_x}
        \lesssim \|f\|_{H^\sigma_xL^2_y}\|g\|_{H^\sigma_xL^2_y},\\
        \|V_{\ge1}*\rho\|_{W^{\sigma,\infty}_x}
        &\le \|V_{\ge1}\|_{L^\infty}\|\rho\|_{W^{\sigma,1}_x}
        \lesssim \|f\|_{H^\sigma_xL^2_y}\|g\|_{H^\sigma_xL^2_y}.
    \end{aligned}
    \]
    Combining these bounds with the algebra property of $H^\sigma$
    and Leibniz's rule, we obtain \eqref{est:direct-Hsigma}.
    This completes the proof of Lemma~\ref{lem:direct}.
\end{proof}

	Next, define the trilinear exchange term by
	\[
	\bigl(\mathsf{X}_{fg^{\ast}}h\bigr)(x,y) := \int_{\mathbb R^d}\int_{\mathbb R^d} V(x-z) f(x,a) \overline{g(z,a)} h(z,y) \,\dd a\,\dd z.
	\]
	We use the following estimate.
	\begin{lemma}[Exchange term]\label{lem:exchange-term}
		Let $d \ge 2$.
		If $f \in L^2_{x,y}$ and $g,h \in (L^2_x \cap L^{\infty}_x) L^2_y$, then
		\begin{equation}\label{est:exchange-L2}
			\|\mathsf{X}_{fg^{\ast}}h\|_{L^2_{x,y}} \lesssim \|f\|_{L^2_xL^2_y} \|g\|_{L^2_xL^2_y}^{(d-1)/d} \|h\|_{L^2_xL^2_y}^{(d-1)/d} \|g\|_{L^\infty_xL^2_y}^{1/d} \|h\|_{L^\infty_xL^2_y}^{1/d}.
		\end{equation}
		Let $\sigma = \lfloor d/2 \rfloor + 1$. If, in addition, $f,g,h\in H^\sigma_xL^2_y$, then
		\begin{equation}\label{est:exchange-Hsigma}
			\|\mathsf{X}_{fg^{\ast}}h\|_{H^\sigma_xL^2_y} \lesssim \|f\|_{H^\sigma_xL^2_y} \|g\|_{H^\sigma_xL^2_y} \|h\|_{H^\sigma_xL^2_y}.
		\end{equation}
	\end{lemma}

\begin{proof}
    Set $F(x)=\|f(x,\cdot)\|_{L^2_y}$,
    $G(x)=\|g(x,\cdot)\|_{L^2_y}$, and
    $H(x)=\|h(x,\cdot)\|_{L^2_y}$.
    By Cauchy--Schwarz in $a$ and Minkowski's inequality,
    \[
        \|(\mathsf X_{fg^*}h)(x,\cdot)\|_{L^2_y}
        \le F(x)(V*(GH))(x).
    \]
    Thus
    \[
        \|\mathsf X_{fg^*}h\|_{L^2_{x,y}}
        \le \|F\|_{L^2_x}\|V*(GH)\|_{L^\infty_x}.
    \]
    Lemma~\ref{lem:potential} and Cauchy--Schwarz give \eqref{est:exchange-L2}.

    To prove \eqref{est:exchange-Hsigma}, set
    $Q(x,a,y)=\overline{g(x,a)}h(x,y)$.
    As in the proof of Lemma~\ref{lem:direct},
    we use the algebra property of
    $H^\sigma$, Leibniz's rule, and Cauchy--Schwarz to obtain
    \[
        \|Q\|_{H^\sigma_xL^2_{a,y}}+\|Q\|_{W^{\sigma,1}_xL^2_{a,y}}
        \lesssim \|g\|_{H^\sigma_xL^2_a}\|h\|_{H^\sigma_xL^2_y}.
    \]
    By Young's inequality, 
    \[
    \begin{aligned}
        \|V_{<1}*Q\|_{H^\sigma_xL^2_{a,y}}
        &\le \|V_{<1}\|_{L^1}\|Q\|_{H^\sigma_xL^2_{a,y}}
        \lesssim \|g\|_{H^\sigma_xL^2_a}\|h\|_{H^\sigma_xL^2_y},\\
        \|V_{\ge1}*Q\|_{W^{\sigma,\infty}_xL^2_{a,y}}
        &\le \|V_{\ge1}\|_{L^\infty}\|Q\|_{W^{\sigma,1}_xL^2_{a,y}}
        \lesssim \|g\|_{H^\sigma_xL^2_a}\|h\|_{H^\sigma_xL^2_y}.
    \end{aligned}
    \]
    Since
    \[
        (\mathsf X_{fg^*}h)(x,y)
        =\int_{\RR^d}f(x,a)(V*_xQ)(x,a,y)\,\dd a,
    \]
    we obtain \eqref{est:exchange-Hsigma} by Cauchy--Schwarz in $a$, the algebra property of $H^{\sigma}$ and Leibniz's rule.
    This completes the proof of Lemma~\ref{lem:exchange-term}.
\end{proof}

For $K\in L^1(\RR^d)+L^\infty(\RR^d)$, define
\[
\begin{aligned}
    \bigl(\mathsf X^K_{fg^*}h\bigr)(x,y)
    &:=\int_{\RR^d}\int_{\RR^d}K(x-z)f(x,a)
    \overline{g(z,a)}h(z,y)\,\dd a\,\dd z,\\
    \mathcal T_K(f,g,h)
    &:=(K*\rho_{fg^*})h-\mathsf X^K_{fg^*}h.
\end{aligned}
\]
We write $\mathcal T=\mathcal T_V$, and $\mathsf X^V_{fg^*}h=\mathsf X_{fg^*}h$.

\begin{lemma}\label{lem:trilinear-cutoff}
Let $d\ge2$, $V(x)=|x|^{-1}$, and $\sigma=\lfloor d/2\rfloor+1$.
For every $u\in H_x^\sigma L_y^2$ and $R>0$, we have
\begin{equation}\label{est:trilinear-cutoff-L2}
    \|\mathcal T(u,u,u)\|_{L^2_{x,y}}
    \lesssim
    \left(R^{d-1}\|u\|_{L_x^\infty L_y^2}^2
    +R^{-1}\|u\|_{L^2_{x,y}}^2\right)\|u\|_{L^2_{x,y}},
\end{equation}
and
\begin{equation}\label{est:homogeneous-nonlinear-energy}
    \|\mathcal T(u,u,u)\|_{\dot H_x^\sigma L_y^2}
    \lesssim
    \left(R^{d-1}\|u\|_{L_x^\infty L_y^2}^2
    +R^{-1}\|u\|_{L^2_{x,y}}^2\right)
    \|u\|_{\dot H_x^\sigma L_y^2}.
\end{equation}
\end{lemma}

\begin{proof}
The pointwise bounds in the proofs of Lemmas~\ref{lem:direct} and
\ref{lem:exchange-term} hold with $|K|$ in place of $V$.
By Young's and H\"older's inequalities,
\[
    \|\mathcal T_K(f,g,h)\|_{L^2_{x,y}}
    \lesssim\|K\|_{L^1}
    \|f\|_{L_x^{p_1}L_y^2}\|g\|_{L_x^{p_2}L_y^2}
    \|h\|_{L_x^{p_3}L_y^2}
\]
for $K\in L^1$ and $2\le p_j\le\infty$ with
$\sum_{j=1}^3p_j^{-1}=1/2$. For $K\in L^\infty$, the same bounds give
\[
    \|\mathcal T_K(f,g,h)\|_{L^2_{x,y}}
    \lesssim\|K\|_{L^\infty}
    \|f\|_{L^2_{x,y}}\|g\|_{L^2_{x,y}}\|h\|_{L^2_{x,y}}.
\]
As in Lemma~\ref{lem:potential}, split $V=V_{<R}+V_{\ge R}$. Since
\[
    \|V_{<R}\|_{L^1}\lesssim R^{d-1},\qquad
    \|V_{\ge R}\|_{L^\infty}=R^{-1},
\]
applying the bound for $K\in L^1$ to $V_{<R}$ with
$(p_1,p_2,p_3)=(\infty,\infty,2)$ and the bound for $K\in L^\infty$
to $V_{\ge R}$ proves
\eqref{est:trilinear-cutoff-L2}.

For smooth $u$, we expand by Leibniz's rule:
\[
    \partial_x^\beta\mathcal T_K(u,u,u)
    =\sum_{\beta_1+\beta_2+\beta_3=\beta}
    C_{\beta_1,\beta_2,\beta_3}
    \mathcal T_K(\partial_x^{\beta_1}u,
    \partial_x^{\beta_2}u,\partial_x^{\beta_3}u).
\]
Fix $|\beta|=\sigma$ and take $p_j=2\sigma/|\beta_j|$, with
$p_j=\infty$ when $\beta_j=0$.
The $L_y^2$-valued Gagliardo--Nirenberg inequality and interpolation
between $L^2_{x,y}$ and $\dot H_x^\sigma L_y^2$ give
\[
\begin{aligned}
    \prod_{j=1}^3\|\partial_x^{\beta_j}u\|_{L_x^{p_j}L_y^2}
    \lesssim
    \|u\|_{L_x^\infty L_y^2}^2
    \|u\|_{\dot H_x^\sigma L_y^2},\qquad 
    \prod_{j=1}^3\|\partial_x^{\beta_j}u\|_{L^2_{x,y}}
    \lesssim
    \|u\|_{L^2_{x,y}}^2\|u\|_{\dot H_x^\sigma L_y^2}.
\end{aligned}
\]
For each term in the Leibniz expansion of $\partial_x^{\beta} \mathcal T_{V_{<R}}(u,u,u)$,
use the estimate for $K\in L^1$ with the exponents $p_j$ chosen above.
For $\mathcal T_{V_{\ge R}}(u,u,u)$, use the estimate for $K\in L^\infty$
with all $\partial_x^{\beta_j} u$ in $L^2_{x,y}$.
Summing over $|\beta|=\sigma$ proves
\eqref{est:homogeneous-nonlinear-energy}.
The estimates extend to $H_x^\sigma L_y^2$ by density.
\end{proof}

	\subsection{Local well-posedness}
	We set
	\[
	\Sigma_x^{\sigma}L^2_y = \{ f \in H^{\sigma}_x L^2_y : \, |x|^\sigma f \in L^2_{x,y} \}, \qquad \| f \|_{\Sigma^{\sigma}_x L^2_y}:= \| f \|_{H^{\sigma}_x L^2_y} + \| |x|^{\sigma} f \|_{L^2_{x,y}}.
	\]

	\begin{lemma}\label{lem:LWP}
		Let $d\ge 2$ and $\sigma = \lfloor d/2 \rfloor +1$.
		Suppose that $\kappa_0 \in \Sigma^{\sigma}_x L^2_y$.
		Then the Cauchy problem \eqref{eqn:kappa} with $\kappa(0) = \kappa_0$ has a unique maximal solution $ \kappa \in C([0,T_*);\Sigma^{\sigma}_x L^2_y)$ for some $T_* \in (0,\infty]$.
		Moreover, $|\mathrm{J}_x|^{\sigma}\kappa \in C([0,T_*); L^2_{x,y})$.
		If $T_* < \infty$, then
		\begin{equation}\label{eq:lwp-continuation}
			\limsup_{t \rightarrow T_*} \| \kappa(t) \|_{H^\sigma_x L^2_y} = \infty.
		\end{equation}
	\end{lemma}

	\begin{proof}
		We first note that, for every fixed $T_0>0$,
		\begin{equation}
			\label{eq:free-Sigma-bound}
			\sup_{|t|\leq T_0} \|\ee^{\ii t\Delta_x/2}f\|_{\Sigma_x^\sigma L_y^2} \le C_{T_0}\|f\|_{\Sigma_x^\sigma L_y^2}.
		\end{equation}
		Indeed, the $H^\sigma_x L^2_y$ norm is preserved by the free flow.
		For the weighted part, we use the identity
		$
		x_j \ee^{\ii t\Delta_x/2} = \ee^{\ii t\Delta_x/2}(x_j-\ii t\partial_{x_j})
		$
		and the standard interpolation estimate
		$
		\|x^{\beta_1}\partial_x^{\beta_2} f\|_{L_{x,y}^2} \lesssim \|f\|_{\Sigma_x^\sigma L_y^2},
		$
		for $|\beta_1| + |\beta_2| \le \sigma$.
		Expanding the corresponding identity for $x^{\alpha} \ee^{\ii t \Delta_x/2}$, with $|\alpha|\le\sigma$, gives \eqref{eq:free-Sigma-bound}.
		By a density argument and \eqref{eq:free-Sigma-bound}, the free Schr\"odinger group is strongly continuous on $\Sigma^{\sigma}_x L^2_y$.

		Applying Lemmas \ref{lem:direct} and \ref{lem:exchange-term}, we get
		\begin{equation}
			\label{eq:T-Hsigma-local}
			\|\mathcal T(f,g,h)\|_{H_x^\sigma L_y^2} \lesssim \|f\|_{H_x^\sigma L_y^2} \|g\|_{H_x^\sigma L_y^2} \|h\|_{H_x^\sigma L_y^2}.
		\end{equation}
		For the weighted part, observe that
		\[
		|x|^\sigma \bigl((V \ast \rho_{fg^{\ast}})h \bigr) = (V \ast \rho_{fg^{\ast}})\bigl(|x|^\sigma h\bigr), \qquad |x|^\sigma \mathsf{X}_{fg^\ast}h = \mathsf{X}_{(|x|^\sigma f)g^\ast}h.
		\]
		Applying the $L^2$ estimates for the direct and exchange terms, we obtain 
		\begin{equation}\label{eq:T-xsigma-local}
			\bigl\||x|^\sigma\mathcal T(f,g,h)\bigr\|_{L_{x,y}^2} \lesssim \|f\|_{H_x^\sigma L_y^2} \|g\|_{H_x^\sigma L_y^2} \||x|^\sigma h\|_{L_{x,y}^2} + \||x|^\sigma f\|_{L_{x,y}^2} \|g\|_{H_x^\sigma L_y^2} \|h\|_{H_x^\sigma L_y^2}.
		\end{equation}
		Combining \eqref{eq:T-Hsigma-local} and \eqref{eq:T-xsigma-local}, we get
		\begin{equation}\label{eq:T-Sigma-local}
			\begin{aligned}
				\|\mathcal T(f,g,h)\|_{\Sigma_x^\sigma L_y^2} &\lesssim \|f\|_{H_x^\sigma L_y^2} \|g\|_{H_x^\sigma L_y^2} \|h\|_{\Sigma_x^\sigma L_y^2} + \|f\|_{\Sigma_x^\sigma L_y^2} \|g\|_{H_x^\sigma L_y^2} \|h\|_{H_x^\sigma L_y^2} \\
				&\lesssim \|f\|_{\Sigma_x^\sigma L_y^2} \|g\|_{\Sigma_x^\sigma L_y^2} \|h\|_{\Sigma_x^\sigma L_y^2}.
			\end{aligned}
		\end{equation}
		Taking $f=g=h=\kappa$ also yields
		\begin{equation}\label{eq:T-Sigma-diagonal}
			\|\mathcal T(\kappa,\kappa,\kappa)\|_{\Sigma_x^\sigma L_y^2} \lesssim \|\kappa\|_{H_x^\sigma L_y^2}^2 \|\kappa\|_{\Sigma_x^\sigma L_y^2}.
		\end{equation}

		Fix $T_0 >0$.
		For $0\le T \le T_0$, define
		\[
		X_T := C([0,T];\Sigma_x^\sigma L_y^2), \qquad \|\kappa\|_{X_T} := \sup_{0\le t\le T} \|\kappa(t)\|_{\Sigma_x^\sigma L_y^2}.
		\]
		Consider the map
		\begin{equation}
			\label{eq:lwp-map}
			\Phi[\kappa](t) := \ee^{\ii t \Delta_x/2}\kappa_0 -\ii\int_0^t \ee^{\ii (t-s) \Delta_x/2}\mathcal T(\kappa,\kappa,\kappa)(s)\,\dd s.
		\end{equation}
		By \eqref{eq:free-Sigma-bound} and \eqref{eq:T-Sigma-diagonal},
		\begin{equation}
			\label{eq:lwp-self-map}
			\|\Phi[\kappa]\|_{X_T} \le C_0 \|\kappa_0\|_{\Sigma_x^\sigma L_y^2} + C_1 T\|\kappa\|_{X_T}^3,
		\end{equation}
		where $C_0, \, C_1 >0$ depend only on $d$ and $T_0$.

		By \eqref{eq:T-Sigma-local}, 
		\begin{align}
			\|\Phi[\kappa^{(1)}]-\Phi[\kappa^{(2)}]\|_{X_T} \le C_1 T \Bigl( \|\kappa^{(1)}\|_{X_T}^2 + \|\kappa^{(1)}\|_{X_T} \|\kappa^{(2)}\|_{X_T} + \|\kappa^{(2)}\|_{X_T}^2 \Bigr) \|\kappa^{(1)} -\kappa^{(2)}\|_{X_T}.
			\label{eq:lwp-contraction}
		\end{align}

		Set $R:=2C_0\|\kappa_0\|_{\Sigma_x^\sigma L_y^2}$ and $B_R := \left\{ \kappa\in  X_T: \|\kappa\|_{X_T}\le R \right\}. $ Choose $0 < T \le T_0$ so that $3C_1 T R^2 \le 1/2$.
		By \eqref{eq:lwp-self-map}, for every $\kappa \in B_R$, we have
		\[
		\|\Phi[\kappa]\|_{X_T} \le \frac{R}{2}+C_1 TR^3 \le \frac{R}{2}+\frac{R}{6} \le R.
		\]
		Thus $\Phi$ maps $B_R$ into itself.
		Moreover, for $\kappa^{(1)},\kappa^{(2)} \in B_R$, \eqref{eq:lwp-contraction} shows that
		\[
		\|\Phi[\kappa^{(1)}]-\Phi[\kappa^{(2)}]\|_{X_T} \leq \frac12 \|\kappa^{(1)}-\kappa^{(2)}\|_{ X_T}.
		\]
		The contraction mapping principle gives a fixed point $\kappa\in B_R$, which solves \eqref{eqn:kappa} with $\kappa(0) = \kappa_0$.

		To obtain uniqueness in $X_T$, let $\kappa^{(1)}, \, \kappa^{(2)} \in X_T$ be two solutions with the same initial data.
		Applying Duhamel's formula and \eqref{eq:T-Sigma-local}, we get
		\[
		\begin{aligned}
			&\|\kappa^{(1)}(t)-\kappa^{(2)}(t)\|_{\Sigma_x^\sigma L_y^2} \le C_{T_0}\int_0^t \left( \|\kappa^{(1)}(s)\|_{\Sigma_x^\sigma L_y^2}^2 + \|\kappa^{(2)} (s)\|_{\Sigma_x^\sigma L_y^2}^2 \right) \|\kappa^{(1)}(s)-\kappa^{(2)}(s)\|_{\Sigma_x^\sigma L_y^2} \,\dd s.
		\end{aligned}
		\]
		By Gr\"onwall's inequality, $\kappa^{(1)}(t) =\kappa^{(2)}(t)$ for $t \in [0,T]$.
		Let $[0,T_*)$ denote the maximal interval of existence.
        For every $0<T<T_*$, \eqref{eq:free-Sigma-bound} and strong continuity of the free group give $\ee^{-\ii t \Delta_x/2} \kappa(t) \in C([0,T];\Sigma^{\sigma}_x L^2_y)$.
		By the definition of $|\mathrm{J}_x(t)|^\sigma$, we obtain $|\mathrm{J}_x|^\sigma\kappa \in C([0,T_*);L_{x,y}^2)$.

		It remains to prove \eqref{eq:lwp-continuation}.
		Suppose $T_* < \infty$ and \eqref{eq:lwp-continuation} fails. Since $\kappa \in C([0,T_*);H^{\sigma}_x L^2_y)$, we then have
		$
		M:= \sup_{t \in[0,T_*)} \| \kappa(t) \|_{H^{\sigma}_x L^2_y}<\infty
		$.
		Using Duhamel's formula, \eqref{eq:T-Sigma-diagonal}, and \eqref{eq:free-Sigma-bound}, we obtain
		\[
		\|\kappa(t)\|_{\Sigma_x^\sigma L_y^2} \le C_{T_*}\|\kappa_0\|_{\Sigma_x^\sigma L_y^2} +C_{T_*}M^2 \int_0^t\|\kappa(s)\|_{\Sigma_x^\sigma L_y^2}\,\dd s
		\]
		for $0 \le t <T_*$.
		By Gr\"onwall's inequality,
		\[
		\sup_{0\le t<T_*} \|\kappa(t)\|_{\Sigma_x^\sigma L_y^2} \le C_{T_*}\|\kappa_0\|_{\Sigma_x^\sigma L_y^2} \, e^{C_{T_*} M^2 T_*} <\infty.
		\]
		The local existence argument therefore extends $\kappa$ beyond $T_*$, contradicting maximality.
		This proves \eqref{eq:lwp-continuation}.
	\end{proof}

	\begin{remark}
		The above proof also shows that for every fixed $T_0>0$, there exist $\epsilon_0 = \epsilon_0 (T_0) >0$ and $C_{T_0} >0$ such that, whenever $\| \kappa_0 \|_{\Sigma^{\sigma}_x L^2_y} \le \epsilon_0$, the solution exists on $[0,T_0]$ and satisfies
		\[
		\sup_{0\leq t\leq T_0} \left( \|\kappa(t)\|_{H_x^\sigma L_y^2} + \||\mathrm{J}_x|^\sigma\kappa(t)\|_{L_{x,y}^2} \right) \le C_{T_0} \| \kappa_0 \|_{\Sigma^{\sigma}_x L^2_y}.
		\]
		In particular, if $\| \kappa_0 \|_{\Sigma^{\sigma}_x L^2_y} \le \epsilon_0(2)$, the solution exists on $[0,2]$  and satisfies these bounds.
	\end{remark}

    \begin{remark}\label{rmk:lambda-C1}
        For the profile $\lambda(t)=\ee^{-\ii t\Delta_x/2}\kappa(t)$, Duhamel's formula,
        \eqref{eq:T-Sigma-local}, and strong continuity of the free group imply that
        $\lambda\in C^1([0,T];\Sigma_x^\sigma L_y^2)$ for every $0<T<T_*$, with
        \[
            \ii \partial_t \lambda(t) =  \ee^{-\ii t \Delta_x/2} \mathcal{T}(\kappa,\kappa,\kappa)(t).
        \]
        Since $\sigma>d/2$, the Fourier transform and Sobolev embedding imply
        $\widehat\lambda\in C^1([0,T];H_k^\sigma L_y^2)$. In particular, $ \widehat\lambda \in C^1([0,T];X)$.
    \end{remark}

	\section{Energy Estimates}
	For sufficiently small initial data, we treat $0\leq t\le 1$ by the local bounds following Lemma~\ref{lem:LWP} and use the bootstrap argument on $1\le  t\le  T$.
	Fix $\delta = (\sigma - d/2)/24$. Let $T>1$ be such that the solution $\kappa$ given by Lemma \ref{lem:LWP} exists on $[0,T]$. We define the bootstrap norm
	\begin{equation}\label{def:bootstrapnorm}
		\mathcal B_T(\kappa) := \sup_{1\le t\le T} \left( \|\kappa(t)\|_{L^2_{x,y}} + t^{-\delta}\|| \mathrm J_x(t)|^\sigma \kappa(t)\|_{L^2_{x,y}} + t^{-\delta}\|\kappa(t)\|_{H^\sigma_x L^2_y} + t^{d/2}\|\kappa(t)\|_{L^\infty_xL^2_y} \right).
	\end{equation}

	\begin{proposition}\label{prop:bootstrap}
		Fix $\delta =(\sigma - d/2)/24$. Let $\kappa_0$ satisfy
		\[
		\|\kappa_0\|_{H_x^\sigma L_y^2} + \||x|^\sigma\kappa_0\|_{L_{x,y}^2} \leq\epsilon_0,
		\]
		where $\epsilon_0>0$ is sufficiently small.
		Let $T>1$, and let $\kappa$ be the solution to \eqref{eqn:kappa} on $[0,T]$ with $\kappa(0)=\kappa_0$, supplied by Lemma~\ref{lem:LWP}.
		Then
		\[
		\mathcal B_T(\kappa) \leq C_0\epsilon_0+C_1\mathcal B_T(\kappa)^3,
		\]
		where $C_0,C_1>0$ depend only on $d$.
	\end{proposition}
    We will prove this in Section~\ref{sec:decayestimates}.

	\begin{lemma}
		\label{lem:energy}
		Let $T>1$ and let $\kappa$ be a solution given by Lemma \ref{lem:LWP} on $[0,T]$ with $\kappa(0)=\kappa_0$.
		Then
		\[
		\|\kappa(t)\|_{L_{x,y}^2} = \|\kappa_0\|_{L_{x,y}^2}
		\]
		for every $0\leq t\leq T$.
		Moreover, for every multi-index $\beta$ with $|\beta|=\sigma$,
		\[
		\frac{\dd}{\dd t} \left( \|\partial_x^\beta\kappa(t)\|_{L_{x,y}^2} + \|\mathrm J_x^\beta\kappa(t)\|_{L_{x,y}^2} \right) \lesssim t^{-1+\delta}\mathcal B_T(\kappa)^3
		\]
		for every  $1\leq t\leq T$.
	\end{lemma}
	\begin{proof}
		Since $V$ is real-valued and $\mathsf X_{\kappa\kappa^\ast}$ is self-adjoint, we get
		\[
		\frac12\frac{\dd}{\dd t}\|\kappa(t)\|_{L_{x,y}^2}^2 = \Im \langle \kappa(t),\ii\partial_t\kappa(t) \rangle_{L_{x,y}^2} = 0,
		\]
		which proves mass conservation.
		Equation \eqref{eqn:kappa} reads
		\[
		\ii\partial_t\kappa = -\frac12\Delta_x\kappa + \mathcal T(\kappa,\kappa,\kappa).
		\]
		Let $\beta$ be a multi-index with $|\beta|=\sigma$.
        By Remark~\ref{rmk:lambda-C1},  we have $\lambda\in C^1([0,T];\Sigma^{\sigma}_x L^2_y)$.
		Since the free Schr\"odinger group is unitary and commutes with $\partial_x^{\beta}$,
        \[
          \| \partial_x^{\beta} \kappa(t) \|_{L^2_{x,y}}
          =
          \| \partial_x^{\beta} \lambda(t)\|_{L^2_{x,y}},
          \qquad
          \| \mathrm{J}_x^{\beta}\kappa (t) \|_{L^2_{x,y}}
          =
          \| x^{\beta} \lambda(t) \|_{L^2_{x,y}}.
        \]
        Differentiating $\|\partial_x^{\beta} \lambda(t) \|_{L^2_{x,y}}^2$ and using the equation for $\lambda$, we obtain
		\[
		\begin{aligned}
			\frac12\frac{\dd}{\dd t} \|\partial_x^\beta\kappa(t)\|_{L_{x,y}^2}^2 &= \Im \left\langle \partial_x^\beta\kappa(t), \partial_x^\beta \mathcal T(\kappa,\kappa,\kappa)(t) \right\rangle_{L_{x,y}^2}.
		\end{aligned}
		\]
		Consequently,
		\begin{equation}
			\label{est:ordinary-energy-reduction}
			\frac{\dd}{\dd t} \|\partial_x^\beta\kappa(t)\|_{L_{x,y}^2} \lesssim \| \partial_x^\beta \mathcal T(\kappa,\kappa,\kappa)(t) \|_{L_{x,y}^2}.
		\end{equation}

		For $t\ge1$, set
		\[
		\widetilde\kappa(t,x,y) = \ee^{-\ii|x|^2/(2t)}\kappa(t,x,y).
		\]
		By the definitions of the direct and exchange terms,
        \[
            \mathcal T(\kappa,\kappa,\kappa)
            =\ee^{\ii|x|^2/(2t)}
            \mathcal T(\widetilde\kappa,\widetilde\kappa,\widetilde\kappa).
        \]
		Since $\mathrm J_x^\beta(t)=\ee^{\ii|x|^2/(2t)}
        (\ii t\partial_x)^\beta\ee^{-\ii|x|^2/(2t)}$ commutes with
        $\ii\partial_t+\Delta_x/2$, the same calculation gives
        \begin{equation} \label{est:weighted-energy-reduction}
            \frac{\dd}{\dd t} \|\mathrm J_x^\beta\kappa(t)\|_{L^2_{x,y}}
            \lesssim t^\sigma
            \|\partial_x^\beta\mathcal T(\widetilde\kappa,\widetilde\kappa,\widetilde\kappa)(t)\|_{L^2_{x,y}}.
        \end{equation}
		Therefore, it suffices to prove
		\begin{equation}
			\label{est:two-nonlinear-energy-bounds}
			\begin{aligned}
				\| \partial_x^\beta \mathcal T(\kappa,\kappa,\kappa)(t) \|_{L_{x,y}^2} &\lesssim t^{-1+\delta}\mathcal B_T(\kappa)^3,\\
				t^\sigma \| \partial_x^\beta \mathcal T( \widetilde\kappa, \widetilde\kappa, \widetilde\kappa )(t) \|_{L_{x,y}^2} &\lesssim t^{-1+\delta}\mathcal B_T(\kappa)^3.
			\end{aligned}
		\end{equation}

        Set $F(x)=\|f(x,\cdot)\|_{L^2_y}$, and define $G,H$ similarly. Cauchy--Schwarz in $a$ and Minkowski's inequality give
        \[
        \begin{aligned}
            \|(K*\rho_{fg^*})(x)h(x,\cdot)\|_{L^2_y}
            &\le (|K|*(FG))(x)H(x),\\
            \|\mathsf X^K_{fg^*}h(x,\cdot)\|_{L^2_y}
            &\le F(x)(|K|*(GH))(x).
        \end{aligned}
        \]
        For $K\in L^1$ and $2\le p_j\le\infty$ with $\sum_{j=1}^3p_j^{-1}=1/2$,
        Young's and H\"older's inequalities imply
        \[
            \|\mathcal T_K(f,g,h)\|_{L^2_{x,y}}
            \lesssim\|K\|_{L^1}
            \|f\|_{L^{p_1}_xL^2_y}\|g\|_{L^{p_2}_xL^2_y}
            \|h\|_{L^{p_3}_xL^2_y}.
        \]
        For $K\in L^\infty$, the same pointwise bounds give
        \[
            \|\mathcal T_K(f,g,h)\|_{L^2_{x,y}}
            \lesssim\|K\|_{L^\infty}
            \|f\|_{L^2_{x,y}}\|g\|_{L^2_{x,y}}\|h\|_{L^2_{x,y}}.
        \]

        For smooth $u$, we compute
        \[
            \partial_x^\beta\mathcal T_K(u,u,u)
            =\sum_{\beta_1+\beta_2+\beta_3=\beta}
            C_{\beta_1,\beta_2,\beta_3}
            \mathcal T_K(\partial_x^{\beta_1}u,
            \partial_x^{\beta_2}u,\partial_x^{\beta_3}u).
        \]
        Take $p_j$ as in the proof of Lemma~\ref{lem:trilinear-cutoff}. Since $\sum_j|\beta_j|=\sigma$,
        we have $\sum_jp_j^{-1}=1/2$.
        We apply the $L_y^2$-valued Gagliardo--Nirenberg inequality and interpolation to obtain
        \[
        \begin{aligned}
            \|\partial_x^{\beta_j}u\|_{L^{p_j}_xL^2_y}
            &\lesssim\|u\|_{L^\infty_xL^2_y}^{1-|\beta_j|/\sigma}
            \|u\|_{\dot H^\sigma_xL^2_y}^{|\beta_j|/\sigma},\\
            \|\partial_x^{\beta_j}u\|_{L^2_{x,y}}
            &\lesssim\|u\|_{L^2_{x,y}}^{1-|\beta_j|/\sigma}
            \|u\|_{\dot H^\sigma_xL^2_y}^{|\beta_j|/\sigma}.
        \end{aligned}
        \]
        Split $V=V_{<t}+V_{\ge t}$. We have
        $\|V_{<t}\|_{L^1}\lesssim t^{d-1}$ and
        $\|V_{\ge t}\|_{L^\infty}=t^{-1}$.
        Applying the bound for $K\in L^1$ to the terms with $V_{<t}$
        and the bound for $K\in L^\infty$ to those with $V_{\ge t}$, then summing over
        $|\beta|=\sigma$, we obtain
        \begin{equation}\label{est:homogeneous-nonlinear-energy-time}
            \|\mathcal T(u,u,u)\|_{\dot H^\sigma_xL^2_y}
            \lesssim
            \left(t^{d-1}\|u\|_{L^\infty_xL^2_y}^2
            +t^{-1}\|u\|_{L^2_{x,y}}^2\right)
            \|u\|_{\dot H^\sigma_xL^2_y}.
        \end{equation}
        By the density argument, this estimate holds for every $u\in H^\sigma_xL^2_y$.

        By the definition of
        $\mathcal B_T(\kappa)$, we have
        \[
        \begin{aligned}
            t^{d-1} \| \kappa(t) \|_{L^{\infty}_x L^2_y}^2 + t^{-1} \| \kappa(t) \|_{L^2_{x,y}}^2 
            &\lesssim t^{-1}\mathcal B_T(\kappa)^2,
        \end{aligned}
        \]
        and similarly for $\widetilde\kappa(t)$ since $|\widetilde\kappa(t)|=|\kappa(t)|$.
        Moreover,
        \[
        \begin{aligned}
            \|\kappa(t)\|_{\dot H^\sigma_xL^2_y}
            &\lesssim t^\delta\mathcal B_T(\kappa),\\
            \|\widetilde\kappa(t)\|_{\dot H^\sigma_xL^2_y}
            &\sim t^{-\sigma}\||\mathrm J_x(t)|^\sigma\kappa(t)\|_{L^2_{x,y}}
            \lesssim t^{-\sigma+\delta}\mathcal B_T(\kappa).
        \end{aligned}
        \]
        Applying \eqref{est:homogeneous-nonlinear-energy-time} to $\kappa$ and $\widetilde\kappa$ proves
        \eqref{est:two-nonlinear-energy-bounds}.
        Combining \eqref{est:ordinary-energy-reduction} and
        \eqref{est:weighted-energy-reduction} completes the proof.

	\end{proof}

	\section{Resonant equation}

	Combining Lemma \ref{lem:energy} with local bounds on $[0,1]$, we obtain
    $$
    \|  \kappa (t) \|_{L^2_{x,y}}  + t^{-\delta} \left( \| |\mathrm{J}_x(t)|^{\sigma} \kappa (t) \|_{L^2_{x,y}} +\|  \kappa (t) \|_{H^{\sigma}_xL^2_{y}}  \right) \lesssim \|\kappa_0\|_{H_x^\sigma L_y^2}+\||x|^\sigma\kappa_0\|_{L^2_{x,y}}+\mathcal{B}_T(\kappa)^3
    $$
    for $1\le t\le T$.
    It remains to estimate $t^{d/2}\|\kappa(t)\|_{L_x^\infty L_y^2}$.
    For this purpose, we identify the leading cubic term in the equation for $\widehat\lambda$. The corresponding resonant equation preserves the $L^2_y$ norm for each $k$. We estimate the remainder in Section~\ref{sec:decayestimates}. Motivated by the arguments of Nguyen--You \cite{NguyenYou2025} and Kato--Pusateri \cite{Kato2011}, we set $\eta=t\ell$ and formally replace $r-\eta/t$ and $k-\eta/t$ by $r$ and $k$. 

	\subsection{Decomposition of the profile equation}\label{sec:Decomp of profile eqn}
    Recall that $d\ge 2$ and $V(x)=|x|^{-1}$. The equation for the kernel of $\kappa$ is
	  \begin{equation}
        \begin{aligned}
			\ii \partial_t \kappa(t,x,y) &= - \frac{1}{2}\Delta_x \kappa(t,x,y) + (V \ast_x \rho_{\kappa\kappa^{\ast}} )(t,x) \kappa(t,x,y)\\
			&\quad- \int V(x-z) \kappa(t,x,a) \overline{\kappa(t,z,a)} \kappa(t,z,y) \, \dd a \, \dd z.
		\end{aligned}
	\end{equation}
	Taking the Fourier transform in $x$, we obtain
	\begin{equation}
		\begin{aligned}
			\ii \partial_t \widehat{\kappa}(t,k,y) &= \frac{|k|^2}{2} \widehat{\kappa}(t,k,y) + (2\pi)^{-2d} \int \widehat{V}(\ell) \widehat{\kappa}(t,r,a) \overline{\widehat{\kappa}(t,r-\ell, a)} \widehat{\kappa}(t,k-\ell, y) \, \dd r \, \dd a \, \dd \ell\\
			&\quad - (2\pi)^{-2d} \int \widehat{V}(\ell) \widehat{\kappa}(t,k-\ell, a) \overline{\widehat{\kappa}(t,r-\ell, a)} \widehat{\kappa}(t,r,y) \, \dd r \, \dd a \, \dd \ell.
		\end{aligned}
	\end{equation}
	Recalling that $\widehat{\lambda}(t,k,y)= \ee^{\ii t|k|^2/2} \widehat{\kappa}(t,k,y)$, the equation becomes
	\begin{equation}\label{eqn:profile}
		\begin{aligned}
			\ii \partial_t \widehat{\lambda}(t,k,y) &= (2\pi)^{-2d} \int \ee^{\ii t (k-r)\cdot \ell} \widehat{V}(\ell) \Flam(t,r,a) \overline{\Flam(t,r-\ell, a)} \Flam(t,k-\ell,y) \, \dd r \, \dd a \, \dd \ell \\
			&\quad - (2\pi)^{-2d} \int \ee^{\ii t (k-r)\cdot \ell} \widehat{V}(\ell) \Flam(t,k-\ell,a) \overline{\Flam(t,r-\ell, a)} \Flam(t,r,y) \, \dd r \, \dd a \, \dd \ell.
		\end{aligned}
	\end{equation}

	For $t >0$ and Schwartz functions $f$, $g$, $h$ on $\mathbb{R}^d$, set $\eta=t\ell$. Since $\widehat V(\eta/t)=t^{d-1}\widehat V(\eta)$, we have
	\begin{equation}\label{eqn:profile-RHS}
		\begin{aligned}
			&(2\pi)^{-2d} \int \ee^{\ii t(k-r)\cdot\ell}\widehat V(\ell)
			\widehat f(r)\overline{\widehat g(r-\ell)}\widehat h(k-\ell) \, \dd r \, \dd\ell \\
			&\quad=(2\pi)^{-2d}t^{-1}\int \ee^{\ii(k-r)\cdot\eta}\widehat V(\eta) 
			\widehat f(r) \overline{\widehat g(r-\eta/t)} \widehat h(k-\eta/t) \, \dd r \, \dd\eta.
		\end{aligned}
	\end{equation}
    Replacing $\widehat g(r-\eta/t)$ and $\widehat h(k-\eta/t)$ by $\widehat g(r)$ and $\widehat h(k)$ suggests the leading term
	\begin{equation}
		\begin{aligned}
			\widehat{\mathcal R}_0(t,k)
			&:=(2\pi)^{-2d}t^{-1}\widehat h(k)\int
			\ee^{\ii(k-r)\cdot\eta}\widehat V(\eta)
			\widehat f(r)\overline{\widehat g(r)}\,\dd r\,\dd\eta\\
			&=(2\pi)^{-d}t^{-1}\big(V*(\widehat f\,\overline{\widehat g})\big)(k)\widehat h(k),
		\end{aligned}
	\end{equation}
    where the last equality follows from the inverse Fourier transform.
    Define the remainder $\widehat{\mathcal E}_0$ by
	\begin{equation}
		\begin{aligned}
			\widehat{\mathcal E}_0(t,k)
			:={}&(2\pi)^{-2d}t^{-1}\int \ee^{\ii(k-r)\cdot\eta}\widehat V(\eta)\widehat f(r)
            \Big(\overline{\widehat g(r-\eta/t)}\widehat h(k-\eta/t)
			-\overline{\widehat g(r)}\widehat h(k)\Big) \,\dd r\,\dd\eta.
		\end{aligned}
	\end{equation}
	Thus, the right-hand side of \eqref{eqn:profile-RHS} equals $\widehat{\mathcal R}_0+\widehat{\mathcal E}_0$.

    To estimate the remainder in Section~\ref{sec:decayestimates}, we write it in physical space. Expanding the Fourier transforms, we find
	\begin{equation}
		\begin{aligned}
			&\widehat f(r)\Big(
			\overline{\widehat g(r-\eta/t)}\widehat h(k-\eta/t)
			-\overline{\widehat g(r)}\widehat h(k)\Big)\\
			&\quad=\int \ee^{-\ii r\cdot(x_1-x_2)}\ee^{-\ii k\cdot x_3}
			\Big( \ee^{\ii\eta\cdot(x_3-x_2)/t}-1 \Big) f(x_1)\overline{g(x_2)}h(x_3)\,\dd x_1\,\dd x_2\,\dd x_3.
		\end{aligned}
	\end{equation}
    Substituting this expansion into the remainder and integrating in $r$, we obtain
    \begin{equation}\label{eq:error-before-r-integration}
		\begin{aligned}
			&\widehat{\mathcal E}_0(t,k) \\
			&=(2\pi)^{-2d}t^{-1}\int
			\ee^{-\ii k\cdot(x_3-\eta)}
			\ee^{-\ii r\cdot(x_1-x_2+\eta)}\widehat V(\eta)
            \Big(\ee^{\ii\eta\cdot(x_3-x_2)/t}-1\Big)
			f(x_1)\overline{g(x_2)}h(x_3)
			\,\dd r\,\dd\eta\,\dd x_1\,\dd x_2\,\dd x_3 \\
            &=(2\pi)^{-d}t^{-1}\int
			\ee^{-\ii k\cdot(x_1-x_2+x_3)}\widehat V(x_2-x_1)\Big(\ee^{\frac{\ii}{t}(x_1-x_2)\cdot(x_2-x_3)}-1\Big)
			f(x_1)\overline{g(x_2)}h(x_3)\,\dd x_1\,\dd x_2\,\dd x_3.
		\end{aligned}
	\end{equation}
    For fixed $x_1,x_2$, set $x=x_1-x_2+x_3$. Then we obtain
	\begin{equation}\label{9/8-3}
		\widehat{\mathcal{E}}_0 (t,k) = (2\pi)^{-d} t^{-1} \int \ee^{-\ii k \cdot x} I_0(t,x,x_1, x_2) \, \dd x_1 \, \dd x_2 \, \dd x,
	\end{equation}
	where
	\begin{equation}
		I_0(t,x,x_1,x_2):=
		\Big(\ee^{\frac{\ii}{t}(x_1-x)\cdot(x_1-x_2)}-1\Big)
		\widehat V(x_2-x_1)
		f(x_1)\overline{g(x_2)}h(x-x_1+x_2).
	\end{equation}

    For the direct term in \eqref{eqn:profile}, apply this calculation with
    $(f,g,h) = \bigl(\lambda(t,\cdot,a),\lambda(t,\cdot,a),\lambda(t,\cdot,y)\bigr)$.
    For the exchange term, use
    $(f,g,h) = \bigl(\lambda(t,\cdot,y),\lambda(t,\cdot,a),\lambda(t,\cdot,a)\bigr)$.
	Now \eqref{eqn:profile} becomes
	\begin{equation}\label{eqn:profile-RandE}
		\begin{aligned}
			\ii \partial_t \widehat{\lambda}(t,k,y) &= t^{-1}\mathcal{R}[\widehat{\lambda}](t,k,y) + \mathcal{E}[\widehat{\lambda}](t,k,y)
		\end{aligned}
	\end{equation}
	where
	\begin{align}
		\mathcal{R}[\Flam](t,k,y) &:= (2\pi)^{-d}  (|\cdot|^{-1} \ast \| \Flam(t,\cdot) \|_{L^2_a}^2)(k) \Flam(t,k,y) \notag \\
		&\quad - (2\pi)^{-d} \int |k-\ell|^{-1} \Flam(t, \ell ,y) \overline{\Flam(t,\ell, a)} \Flam(t,k,a) \, \dd a \, \dd \ell, \label{def:RFlam} \\[1ex]
		\mathcal{E}[\Flam](t,k,y) &:= (2\pi)^{-d} t^{-1}\int \ee^{-\ii k \cdot x} I_{\lambda}(t,x,y,x_1, x_2) \, \dd x_1 \, \dd x_2 \, \dd x, \label{def:EFlam}
	\end{align}
	where
	\begin{equation}
		\begin{aligned}
			&I_\lambda(t,x,y,x_1, x_2)
			=  \left( \ee^{\frac{\ii}{t} (x_1 - x) \cdot (x_1 - x_2)} -1 \right)  \widehat{V}\left(x_2 - x_1 \right) \\
			&\qquad \times \int \Big( \lambda(t,x_1, a) \overline{\lambda(t,x_2,a)} \lambda(t, x-x_1 + x_2, y) - \lambda(t,x_1, y) \overline{\lambda(t,x_2,a)} \lambda(t, x-x_1 + x_2, a)  \Big) \, \dd a.
		\end{aligned}
	\end{equation} 

    \subsection{Resonant dynamics}
	The interaction phase $(k-r)\cdot\ell$ in \eqref{eqn:profile} vanishes at $\ell=0$, where the multiplier $\widehat V(\ell)$ is singular. The leading term $t^{-1}\mathcal R[\widehat\lambda]$ extracted above describes the Coulomb interaction near this time resonance. Omitting the remainder, we obtain the equation for the resonant profile $\widehat\nu$:
	\[
	\ii \partial_t \widehat{\nu}(t,k,y) = t^{-1}\mathcal{R}[\widehat{\nu}](t,k,y),
	\]
	or equivalently,
	\begin{equation}\label{eqn:res-eqn}
		\begin{aligned}
			\ii \partial_t \widehat{\nu}(t,k,y) &= (2\pi)^{-d} t^{-1} (|\cdot|^{-1} \ast_k \| \widehat{\nu}(t,\cdot) \|_{L^2_y}^2)(k) \widehat{\nu}(t,k,y) \\
			&\quad - (2\pi)^{-d} t^{-1} \int |k-\ell|^{-1} \overline{\widehat{\nu}(t,\ell,a)} \widehat{\nu}(t,\ell,y) \widehat{\nu}(t,k,a) \, \dd \ell \, \dd a.
		\end{aligned}
	\end{equation}
    Recall the space $X$ defined in \eqref{def:X}. 
	For fixed $f \in X$ and $k$, define $A[f](k):L^2_y \rightarrow L^2_y$ by
	\begin{equation}\label{def:A}
		\begin{aligned}
			A[f](k)h(y) &= (2\pi)^{-d} (|\cdot|^{-1} \ast_k \| f \|_{L^2_y}^2)(k) h(y) \\
			&\quad - (2\pi)^{-d}  \int |k-\ell|^{-1} \overline{f(\ell,a)} f(\ell,y) h(a) \, \dd \ell \, \dd a.
		\end{aligned}
	\end{equation}
	Then the resonant equation can be written as
	\begin{equation}\label{eqn:res-eqn-02}
		\ii \partial_t \widehat{\nu}(t,k,y) = t^{-1} A[\widehat{\nu}(t)](k) \widehat{\nu}(t,k,y).
	\end{equation}
	We first establish bounds and self-adjointness for $A$.
	\begin{lemma}\label{lem:A}
		Let $f,g \in X$ and $h \in L^2_y$.
        Then \eqref{def:A} defines a bounded self-adjoint operator $A[f](k)$ on $L^2_y$ for every $k$. The map $k \mapsto A[f](k)$ is continuous in $\mathcal{B}(L^2_y)$.
		We also have
		\begin{align}
			\| A[f](k)h \|_{L^2_y} &\lesssim \left( |\cdot|^{-1} \ast_k \rho_{ff^*} \right)\! (k)  \|h \|_{L^2_y}, \label{est:Afh-y}\\
			\| (A[f](k) - A[g](k))h \|_{L^2_y} &\lesssim \left(|\cdot|^{-1} \ast_k \left[(\rho_{ff^*}^{1/2} + \rho_{gg^*}^{1/2}) \rho_{(f-g)(f-g)^*}^{1/2}\right] \right) \!(k)  \|h \|_{L^2_y} \label{est:Afh-Agh-y}
		\end{align}
		for every $k \in \mathbb{R}^d$.
		If $h=h(k,y) \in X$, then $A[f]h$ and $(A[f]-A[g])h$ belong to $X$, and 
		\begin{align}
			\|A[f]h\|_{X} &\lesssim \| f \|_{X}^2 \| h \|_{X}, \label{est:Afh-ky}\\
			\| (A[f]-A[g])h\|_{X} &\lesssim \left( \| f \|_{X} + \| g \|_{X} \right) \| f-g \|_{X} \| h \|_{X}.\label{est:Afh-Agh-ky}
		\end{align}
	\end{lemma}
	\begin{proof}
		Fix $k\in\mathbb R^d$. The map $\ell\mapsto |f(\ell)\rangle\langle f(\ell)|$ is continuous in $\mathcal B(L^2_y)$, with $\| |f(\ell)\rangle\langle f(\ell)|\|_{\mathcal B(L^2_y)}=\|f(\ell)\|_{L^2_y}^2$.
        Since $\rho_{ff^*}\in L_k^1\cap L_k^\infty$ and $d\ge2$, Lemma~\ref{lem:potential} implies that the Bochner integral defining the exchange term of \eqref{def:A} converges in $\mathcal B(L^2_y)$, with
        \[
            \left\|(2\pi)^{-d}\int |k-\ell|^{-1}|f(\ell)\rangle\langle f(\ell)|\,\dd\ell\right\|_{\mathcal B(L^2_y)}
            \le (2\pi)^{-d}(|\cdot|^{-1}*\rho_{ff^*})(k).
        \]
        Combining this estimate with \eqref{def:A}, we obtain \eqref{est:Afh-y}.
        Finally, each rank-one operator $|f(\ell)\rangle\langle f(\ell)|$ is self-adjoint, and the direct term is a real scalar multiple of the identity. Hence $A[f](k)$ is self-adjoint.

        To prove continuity in $k$, choose a smooth cutoff $\chi\in C_c^\infty(\mathbb R^d)$ with $\chi=1$ near the origin, and write
        $K_0(k)=\chi(k)|k|^{-1}$ and $K_\infty(k)=(1-\chi(k))|k|^{-1}$.
        Since $d\ge 2$, we have $K_0\in L^1$, while $K_\infty$ is bounded and uniformly continuous. By Young's inequality,
        \[
        \begin{aligned}
            &\sup_k\|A[f](k+z)-A[f](k)\|_{\mathrm{op}}
            \lesssim \|K_0(\cdot+z)-K_0\|_{L^1}\|f\|_{L_k^\infty L_y^2}^2
            +\|K_\infty(\cdot+z)-K_\infty\|_{L^\infty}\|f\|_{L_{k,y}^2}^2
            \to 0
        \end{aligned}
        \]
        as $z\to0$.
		For the direct term of $A[f](k)-A[g](k)$, we use
		\[
		|\rho_{ff^*} - \rho_{gg^*}| \le (\rho_{ff^*}^{1/2} + \rho_{gg^*}^{1/2}) \left| \rho_{ff^*}^{1/2} - \rho_{gg^*}^{1/2} \right| \le (\rho_{ff^*}^{1/2} + \rho_{gg^*}^{1/2}) \rho_{(f-g)(f-g)^*}^{1/2}.
		\]
		For the exchange term, we use
		\begin{equation*}
			\begin{aligned}
				f(\ell, y) \langle f(\ell), h\rangle_{L^2_y} - g(\ell,y) \langle g(\ell), h \rangle_{L^2_y} 
				= (f(\ell, y) - g(\ell,y)) \langle f(\ell),h  \rangle_{L^2_y} + g(\ell, y) \langle f(\ell)- g(\ell),h  \rangle_{L^2_y}.
			\end{aligned}
		\end{equation*}
		By Minkowski's inequality and Cauchy--Schwarz, we obtain \eqref{est:Afh-Agh-y}.

		For $h=h(k,y) \in X$, the continuity of $A[f](k)$ in operator norm and of $h(k)$ in $L^2_y$ implies that $A[f](k)h(k)$ is continuous in $k$. The same holds for $(A[f](k)-A[g](k))h(k)$.
		By Lemma~\ref{lem:potential} and \eqref{def:X},
		\[
		\left\| |\cdot|^{-1} \ast_k \rho_{ff^*} \right\|_{L^{\infty}_k} \lesssim \| \rho_{ff^*} \|_{L^1} + \| \rho_{ff^*} \|_{L^{\infty}} \lesssim \| f \|_{X}^2.
		\]
		Applying this estimate to \eqref{est:Afh-y}, we get \eqref{est:Afh-ky}.
		Similarly, Lemma~\ref{lem:potential} and Cauchy--Schwarz imply
		\begin{equation*}
			\begin{aligned}
				&\left\| |\cdot|^{-1} \ast_k \left[ (\rho_{ff^*}^{1/2} + \rho_{gg^*}^{1/2}) \rho_{(f-g)(f-g)^*}^{1/2} \right]\right\|_{L^{\infty}_k}
				\lesssim \left( \| f\|_{X} + \|g \|_{X} \right) \| f-g\|_{X}.
			\end{aligned}
		\end{equation*}
		Combining this estimate with \eqref{est:Afh-Agh-y}, we obtain \eqref{est:Afh-Agh-ky}.
		This completes the proof of Lemma \ref{lem:A}.
	\end{proof}

	Next, we establish the global well-posedness of this resonant equation.
	\begin{lemma}\label{lem:resonant-GWP}
		For every $\widehat{\nu}_1 \in X$, there exists a unique global-in-time solution $\widehat{\nu} \in C^1([1,\infty);X)$ to \eqref{eqn:res-eqn} with $\widehat{\nu}(1) = \widehat{\nu}_1$.
        The equation holds in $L^2_y$ for every $k \in \mathbb{R}^d$ and $t \ge 1$, and
        \[
            \| \widehat{\nu}(t,k) \|_{L^2_y} = \| \widehat{\nu}_1 (k) \|_{L^2_y}, \qquad \| \widehat{\nu}(t) \|_X = \| \widehat{\nu}_1 \|_X.
        \]
	\end{lemma}
	\begin{proof}
		Lemma~\ref{lem:A} shows that $f\mapsto A[f]f$ maps $X$ into itself. Using
        \[
            A[f]f-A[g]g= (A[f]-A[g])f + A[g](f-g),
        \]
        we obtain
        \begin{equation}\label{est:A-diff}
           \| A[f]f-A[g]g\|_X
           \lesssim
           (\|f\|_X + \|g\|_X)^2 \|f-g\|_X .
        \end{equation}
		If $\widehat{\nu}_1 =0$, then the zero solution is unique by \eqref{est:A-diff} and Gronwall's inequality. Otherwise, let
		\[
		B_T = \Bigl\{ f \in C \bigl([1,1+T], X \bigr) : \sup_{t \in [1,1+T]} \| f(t) \|_{X} \le 2 \| \widehat{\nu}_1 \|_{X} \Bigr\}
		\]
		be the closed ball.
		Define a map $\Phi$ on $B_T$ by
        \[
            \Phi(v)(t):= \widehat{\nu}_1 - \ii \int_1^t \tau^{-1} A[v(\tau)] v(\tau) \, \dd \tau,
            \qquad
            v \in B_T.
        \]
        By Lemma~\ref{lem:A} and \eqref{est:A-diff},
        \[
        \begin{aligned}
            \| \Phi(v) \|_{C([1,1+T];X)}
            &\le 
            \| \widehat{\nu}_1 \|_X
            +
            C \| \widehat{\nu}_1 \|_X^3 \log (1+T) , \\
            \| \Phi(v)-\Phi(w) \|_{C([1,1+T];X)}
            &\le C \| \widehat{\nu}_1 \|_X^2 \log (1+T) \| v-w \|_{C([1,1+T];X)},
        \end{aligned}
        \]
        where $C>0$ only depends on $d$. Choose $T>0$ so that $C\| \widehat{\nu}_1 \|_X^2 \log (1+T) \le 1/2$. Then $\Phi$ maps $B_T$ into itself and is a contraction. Its fixed point belongs to $C^1([1,1+T];X)$ since the integrand is continuous in $X$.

        For any starting time $t_0 \ge 1$, the same argument applies since
        \[
            \int_{t_0}^{t_0+T} \tau^{-1} \, \dd \tau = \log (1+T/t_0) \le \log (1+T).
        \]
        We thus obtain a unique maximal solution $\widehat{\nu} \in C^1 ([1,T_*);X)$. If $T_* <\infty$, then
		\begin{equation}\label{eqn:resonant-blowup}
			\limsup_{t \rightarrow T_*} \| \widehat{\nu}(t) \|_{X} = \infty.
		\end{equation}
        Since $\| f(k)\|_{L^2_y} \le \| f \|_X$ for each $k\in \mathbb{R}^d$, evaluation at $k$ is a bounded map from $X$ to $L^2_y$. Hence,
        \[
            \widehat\nu(t,k)=\widehat\nu_1(k)
            -\ii\int_1^t s^{-1}A[\widehat\nu(s)](k)\widehat\nu(s,k)\,\dd s.
        \]
        Moreover, $\widehat{\nu}(\cdot, k)$ is $C^1$ in $L^2_y$ and satisfies \eqref{eqn:res-eqn-02} for every $k$. Since $A[\widehat\nu(t)](k)$ is self-adjoint, we compute
		\begin{equation*}
        \begin{aligned}
			\frac{1}{2}\frac{\dd}{\dd t} \| \widehat{\nu}(t,k) \|_{L^2_y}^2 
            &=  \Im \langle \widehat{\nu}(t,k), \, \ii\partial_t \widehat{\nu}(t,k)\rangle_{L^2_y}
            = t^{-1} \Im \langle \widehat{\nu}(t,k), A[\widehat{\nu}(t)](k) \widehat{\nu}(t,k) \rangle_{L^2_y} =0 .
        \end{aligned}
		\end{equation*}
		Therefore $\| \widehat{\nu}(t,k)\|_{L^2_y} = \|\widehat{\nu}_1(k)\|_{L^2_y}$
		for every $k\in \mathbb{R}^d$ and $1 \le t <T_*$. In particular, $\| \widehat{\nu}(t) \|_X = \| \widehat{\nu}_1 \|_X$. This rules out \eqref{eqn:resonant-blowup}, so the solution is global.
		This completes the proof of Lemma \ref{lem:resonant-GWP}.
	\end{proof}

    \subsection{Unitary propagator}
    For the modified scattering statement, we also need the linear evolution generated by $t^{-1}A[f(t)](k)$.
	For $f\in C([1,\infty);X)$ and every $k \in \mathbb{R}^d$, define the associated propagator on $L^2_y$ by
    \begin{equation}\label{def:Uf-propagator}
		\ii\partial_t\mathcal U_f(t,s;k) =t^{-1}A[f(t)](k)\mathcal U_f(t,s;k),\qquad \mathcal U_f(s,s;k)=I.
	\end{equation}
    \begin{lemma}\label{lem:unitarypropagator}
		Let $f\in C([1,\infty);X)$. The propagator in \eqref{def:Uf-propagator} exists uniquely for every $t,s\ge1$ and $k\in \mathbb{R}^d$, and is unitary on $L_y^2$.
        For $r\ge1$, it satisfies
		\begin{equation*}
		\begin{gathered}
		    \mathcal U_f(t,r;k)\mathcal U_f(r,s;k)=\mathcal U_f(t,s;k),\qquad \mathcal U_f(t,s;k)^{-1}=\mathcal U_f(s,t;k)=\mathcal U_f(t,s;k)^*,\\
            \ii\partial_s\mathcal U_f(t,s;k) =-s^{-1}\mathcal U_f(t,s;k)A[f(s)](k).
		\end{gathered}
		\end{equation*}
        For every $1\le a<b<\infty$, the map $(t,s)\mapsto\mathcal U_f(t,s;\cdot)$ belongs to $C^1\! \left([a,b]^2;C_b\!\left(\mathbb R^d; \mathcal B(L_y^2)\right)\right)$.
	\end{lemma}
	\begin{proof}
        Fix $1\le a<b<\infty$. By Lemma~\ref{lem:A}, the map $t\mapsto t^{-1}A[f(t)]$ is continuous on $[a,b]$ with values in $C_b\!\left(\mathbb R^d;\mathcal B(L_y^2)\right)$,
        and
        \[
            \sup_{t\in[a,b]}\sup_k t^{-1}\|A[f(t)](k)\|_{\mathrm{op}}
            \lesssim a^{-1}\sup_{t\in[a,b]}\|f(t)\|_X^2<\infty.
        \]
        The standard existence and uniqueness theorem for linear ODEs therefore provides a unique solution of
        \[
            \mathcal U_f(t,s;k)
            =I-\ii\int_s^t \tau^{-1}A[f(\tau)](k)\mathcal U_f(\tau,s;k)\,\dd\tau
        \]
        for $t,s\in[a,b]$. For each $s$, the map $t\mapsto\mathcal U_f(t,s;\cdot)$ is continuously differentiable in the norm $\sup_k\|\cdot\|_{\mathrm{op}}$, and \eqref{def:Uf-propagator} holds in this norm.

        For $r\in[a,b]$, uniqueness implies
        \[
            \mathcal U_f(t,r;k)\mathcal U_f(r,s;k)=\mathcal U_f(t,s;k),
            \qquad
            \mathcal U_f(t,s;k)^{-1}=\mathcal U_f(s,t;k).
        \]
        For $g,h\in L_y^2$, self-adjointness of $A[f(t)](k)$ implies
        \[
            \frac{\dd}{\dd t}
            \langle\mathcal U_f(t,s;k)g,\mathcal U_f(t,s;k)h\rangle_{L_y^2}=0.
        \]
        Thus $\mathcal U_f(t,s;k)$ is an isometry, and its invertibility implies unitarity.

        Finally, differentiating
        \[
            \mathcal U_f(t,s;k)=\mathcal U_f(t,a;k)\mathcal U_f(s,a;k)^*
        \]
        with respect to $s$ and using self-adjointness, we obtain
        \[
            \ii\partial_s\mathcal U_f(t,s;k)
            =-s^{-1}\mathcal U_f(t,s;k)A[f(s)](k).
        \]
        The representation above establishes joint $C^1$ regularity in $(t,s)$ in the norm $\sup_k\|\cdot\|_{\mathrm{op}}$.
        Since $[a,b]$ is arbitrary, the lemma follows.
    \end{proof}

    \begin{remark}\label{rem:unitary-preserves-X}
        Assume $f\in C([1,\infty);X)$, as in
        Lemma~\ref{lem:unitarypropagator}.
        For $t,s\ge1$ and $h\in X$, define
        \[
            (\mathcal U_f(t,s)h)(k):=\mathcal U_f(t,s;k)h(k).
        \]
        Continuity in $k$ and unitarity on $L_y^2$ give
        $\mathcal U_f(t,s)h\in X$ and
        $\|\mathcal U_f(t,s)h\|_X=\|h\|_X$.
        Its inverse on $X$ is $\mathcal U_f(s,t)$.
        It is also unitary on $L_{k,y}^2$ and an isometry on every
        $L_k^pL_y^2$, $2\le p\le\infty$.
        For every $1\le a<b<\infty$, the map
        $(t,s)\mapsto\mathcal U_f(t,s)$ belongs to
        $C^1([a,b]^2;\mathcal B(X))$, with
        \begin{equation}\label{eqn:ODE-Uf-X}
        \begin{aligned}
            \ii\partial_t\mathcal U_f(t,s)
            =t^{-1}A[f(t)]\mathcal U_f(t,s),
            \qquad
            \ii\partial_s\mathcal U_f(t,s)
            =-s^{-1}\mathcal U_f(t,s)A[f(s)].
        \end{aligned}
        \end{equation}
        Indeed, for $B\in C_b(\mathbb R^d;\mathcal B(L_y^2))$ and $h \in X$, we have
        $\|Bh\|_X\le\sup_k\|B(k)\|_{\mathrm{op}}\|h\|_X$.
        Thus the $C^1$ regularity in Lemma~\ref{lem:unitarypropagator} also holds in $\mathcal{B}(X)$.

        For $1\le a<b<\infty$ and $f\in C([a,b];X)$, the proof of Lemma~\ref{lem:unitarypropagator} constructs $\mathcal U_f(t,s;k)$ for $t,s\in[a,b]$.
        All the preceding conclusions hold for these times.
    \end{remark}

	\section{Error and decay estimates}\label{sec:decayestimates}
	In this section, we estimate the remainder in the profile equation and then use this bound to prove the dispersive estimate for $\kappa$. This proves Proposition~\ref{prop:bootstrap}.
    We recall that
	\begin{equation}\label{eqn:profile-decay}
		\ii \partial_t \Flam(t,k,y) = t^{-1}\mathcal{R}[\Flam](t,k,y) + \mathcal{E}[\Flam](t,k,y)
	\end{equation}
	where $\mathcal{R}$ and $\mathcal{E}$ are defined in \eqref{def:RFlam} and \eqref{def:EFlam}, respectively.
	The following lemma gives a decay estimate for the error term $\mathcal{E}[\Flam](t)$ in $X$, defined in \eqref{def:X}. When $\theta > 3\delta$, this estimate is integrable in time and yields a uniform bound for $\|\Flam(t)\|_X$.
	\begin{lemma}\label{lem:error}
		Let $d\ge 2$, $\sigma = \lfloor d/2 \rfloor + 1$, and $T >1$. Let $\kappa$ be a solution to \eqref{eqn:kappa} on $[0,T]$, supplied by Lemma~\ref{lem:LWP}.
		Set $\lambda(t) = \ee^{-\ii t \Delta/2} \kappa(t)$.
		With $\delta = (\sigma-d/2)/24$ as in \eqref{def:bootstrapnorm}, for $0 \le \theta < \sigma - d/2$, we have 
		\[
		\| \mathcal{E}[\Flam](t) \|_{X} \lesssim t^{-1-\theta + 3 \delta} \mathcal{B}_{T}(\kappa)^3
		\]
		for $1 \le t \le T$.
		Moreover, for every $k \in \RR^d$ and $1\le t_1 \le t_2 \le T$,
        \[
            \| \Flam(t_2,k)\|_{L^2_y}
            \lesssim
            \| \Flam(t_1,k)\|_{L^2_y}
            +\mathcal{B}_T(\kappa)^3 \int_{t_1}^{t_2} \tau^{-1-\theta +3\delta} \, \dd \tau.
        \]
		In particular, taking $\theta> 3\delta$, we obtain
		\[
		\sup_{t \in [1,T]} \| \widehat{\lambda}(t) \|_{X} \lesssim \|\widehat\lambda(1)\|_{X} + \mathcal B_T(\kappa)^3.
		\]
	\end{lemma}

	\begin{proof}
		We first estimate the remainder in $L^p_k L^2_y$ for $2 \le p \le \infty$.
		By \eqref{def:EFlam},
		\[
		\mathcal{E}[\widehat{\lambda}](t,k,y) = (2\pi)^{-d} t^{-1} \int \mathcal{F}_x [I_{\lambda}(t,\cdot,y,x_1, x_2)](k) \, \dd x_1 \, \dd x_2.
		\]
		For $p \in [2,\infty]$ with $1/p + 1/p'= 1$, the Hausdorff--Young and Minkowski inequalities imply
		\begin{equation}\label{est:E-I}
			\| \mathcal{E}[\widehat{\lambda}](t) \|_{L^p_k L^2_y}
            \lesssim t^{-1} \| I_{\lambda}(t) \|_{  L^1_{x_1, x_2} L^{p'}_x L^2_y}.
		\end{equation}
		Recall that
		\begin{equation*}
			\begin{aligned}
				&I_{\lambda}(t,x,y,x_1, x_2) 
				=  \left( \ee^{\frac{\ii}{t} (x_1 - x) \cdot (x_1 - x_2)} -1 \right)  \widehat{V}\left( x_2 - x_1 \right) \\
				&\quad \times \int \Big( \lambda(t,x_1, a) \overline{\lambda(t,x_2,a)} \lambda(t, x-x_1 + x_2, y) - \lambda(t,x_1, y) \overline{\lambda(t,x_2,a)} \lambda(t, x-x_1 + x_2, a)  \Big) \, \dd a.
			\end{aligned}
		\end{equation*}
        Since $0 \le \theta < \sigma-d/2 \le 1$, we have
        \[
            \left| \ee^{\frac{\ii}{t} (x_1 - x) \cdot (x_1 -x_2)} -1 \right|
            \lesssim
            t^{-\theta}
            |x_1-x|^{\theta}
            |x_1-x_2|^{\theta}.
        \]
		Since $\widehat{V}(z) = c_d |z|^{1-d}$ for some constant $c_d>0$, we get
		\begin{equation}\label{est:I}
			\begin{aligned}
				&t^{-1}\|I_{\lambda}(t)\|_{L^1_{x_1,x_2} L^{p'}_x L^2_y} \\
				&\lesssim t^{-1-\theta} \int  |x_1 - x_2|^{-d+1+\theta} \|\lambda(t,x_1) \|_{L^2_y} \|\lambda(t,x_2) \|_{L^2_y} \Big\| |x_1- x|^{\theta} \|\lambda(t,x-x_1 +x_2)\|_{L^2_y} \Big\|_{L^{p'}_x} \, \dd x_1 \, \dd x_2.
			\end{aligned}
		\end{equation}
		After the change of variables $z= x-x_1 + x_2$, the triangel inequality gives
		\[
		\Big\|  |x_1 - x|^{\theta} \|\lambda(t,x-x_1 + x_2) \|_{L^2_y} \Big\|_{L^{p'}_x} \lesssim \langle x_2 \rangle^{\theta} \| \lambda(t) \|_{L^{p'}_x L^2_y} + \| \langle x \rangle^{\theta} \lambda(t) \|_{L^{p'}_x L^2_y}.
		\]
		Meanwhile, using the Hardy--Littlewood--Sobolev inequality, we get
		\[
		\int_{\RR^d \times \RR^d} |x_1 - x_2|^{-d+1+\theta} f(x_1) g(x_2) \, \dd x_1 \, \dd x_2 \lesssim \|f \|_{L^q_x} \| g \|_{L^q_x}
		\]
		where $q= 2d/(d+1+\theta) \in (1,2)$ satisfies $2/q + (d-1-\theta)/d =2$.
		Applying these two bounds to \eqref{est:I}, we find that
		\begin{equation}
			\begin{aligned}
				t^{-1} \| I_{\lambda}(t) \|_{L^1_{x_1, x_2} L^{p'}_x L^2_y} &\lesssim t^{-1-\theta} \left( \| \lambda(t) \|_{L^{p'}_x L^2_y} \| \lambda(t) \|_{L^q_x L^2_y} \| \langle x \rangle^{\theta} \lambda(t) \|_{L^q_x L^2_y} + \| \langle x \rangle^{\theta} \lambda(t) \|_{L^{p'}_x L^2_y} \| \lambda(t) \|_{L^q_x L^2_y}^2 \right) \\
				&\lesssim t^{-1-\theta} \| \langle x \rangle^m \lambda(t) \|_{L^2_{x,y}}^3,
			\end{aligned}
		\end{equation}
		for any $m> d/2 + \theta$, where the last inequality follows from H\"older's inequality. 
		Since $\theta < \sigma - d/2$, we may choose $m = \sigma$. Then we obtain from \eqref{est:E-I} and \eqref{def:bootstrapnorm} that
		\[
		\| \mathcal{E}[\widehat{\lambda}](t) \|_{L^p_k L^2_y} \lesssim t^{-1-\theta +3\delta} \mathcal{B}_{T}(\kappa)^3
		\]
		uniformly in $2 \le p \le \infty$ and $1 \le t \le T$.
        By \eqref{def:EFlam} and the preceding estimates,
        $\mathcal E[\widehat\lambda](t)\in C_b(\mathbb R^d;L_y^2)$.
        With $p=2$ and $p=\infty$, this proves the desired bound in $X$.

        By Remark~\ref{rmk:lambda-C1}, Lemma~\ref{lem:A}, and \eqref{eqn:profile-decay}, we have
        \[
            \mathcal E[\widehat\lambda]
            =\ii\partial_t\widehat\lambda
                -t^{-1}A[\widehat\lambda]\widehat\lambda
            \in C([1,T];X).
        \]

        By Remark~\ref{rem:unitary-preserves-X}, we obtain for $ 1 \le t_1 \le t_2 \le T$ that
        \[
            \Flam(t_2) = \mathcal{U}_{\Flam}(t_2,t_1)\Flam(t_1) - \ii \int_{t_1}^{t_2} \mathcal{U}_{\Flam}(t_2,\tau) \mathcal{E}[\Flam](\tau) \, \dd \tau.
        \]
        Unitarity of the propagator in $L^2_y$ implies
        \begin{equation}\label{est:profile-integral-ineq}
            \| \Flam(t_2,k)\|_{L^2_y} \le \|\Flam(t_1,k)\|_{L^2_y}
            +
            \int_{t_1}^{t_2} \| \mathcal{E}[\Flam](\tau,k)\|_{L^2_y} \, \dd \tau
        \end{equation}
        for every $k \in \RR^d$.

        Finally, we take the $L_k^2$ and $L_k^\infty$ norms in \eqref{est:profile-integral-ineq} and apply Minkowski's inequality to obtain
        \[
            \begin{aligned}
                \|\Flam(t)\|_{X}
                &\le\|\Flam(1)\|_{X}+
                \int_1^t\|\mathcal E[\Flam](s)\|_{X}\,\dd s
                \lesssim \|\Flam(1)\|_{X}+
                \mathcal B_T(\kappa)^3\int_1^t s^{-1-\theta+3\delta}\,\dd s
            \end{aligned}
        \]
        for $1 \le t \le T$.
        If $3\delta<\theta$, the last integral is finite, uniformly in $t$ and $T$.
		This completes the proof of Lemma \ref{lem:error}.

	\end{proof}

	\subsection{Proof of Proposition~\texorpdfstring{\ref{prop:bootstrap}}{\ref*{prop:bootstrap}}}
	We are ready to prove the bootstrap proposition.
	\begin{proof}
		For $0\le t\le1$, the local bounds following Lemma~\ref{lem:LWP} and Sobolev embedding yield
		\[
		\langle t \rangle^{d/2} \| \kappa(t) \|_{L^{\infty}_x L^2_y} \lesssim \| \kappa(t) \|_{H^{\sigma}_x L^2_y} \lesssim \epsilon_0 + \mathcal{B}_T(\kappa)^3.
		\]
		For $1\le t\le T$, set $\theta=6\delta$, where $\delta = (\sigma -d/2)/24$.
        Then $0<3\delta<\theta<(\sigma-d/2)/2$ and $\theta-3\delta=3\delta>0$.
        Note that
        \[
            \|\Flam(1)\|_{X}
            \lesssim
            \|\langle x\rangle^\sigma\lambda(1)\|_{L_{x,y}^2}
            \lesssim\epsilon_0.
        \]
		By Lemma \ref{lem:error}, we have
		\[
		\sup_{1 \le t \le T} \| \widehat{\lambda}(t) \|_{L^{\infty}_k L^2_y} \lesssim \epsilon_0 + \mathcal{B}_T(\kappa)^3.
		\]
		Moreover, the local bound at $t=1$, mass conservation, and Lemma \ref{lem:energy} imply
		\[
		\| \kappa(t) \|_{H^{\sigma}_x L^2_y} + \| \langle x \rangle^{\sigma} \lambda(t) \|_{L^2_{x,y}} \lesssim \epsilon_0 + \mathcal{B}_T(\kappa)^3 \int_1^t \tau^{-1+\delta} \, \dd \tau
        \lesssim
        \epsilon_0 + t^{\delta}\mathcal{B}_T(\kappa)^3.
		\]
		for $1\le t \le T$, where we used $\| |x|^{\sigma} \lambda(t) \|_{L^2_{x,y}} =\| |\mathrm{J}_x (t) |^{\sigma} \kappa(t) \|_{L^2_{x,y}}$.
		Applying Lemma \ref{lem:decay} with $m = \sigma$, we obtain
		\begin{align*}
			t^{d/2} \| \kappa(t) \|_{L^{\infty}_x L^2_y} &\lesssim \| \widehat{\lambda}(t) \|_{L^{\infty}_k L^2_y} + t^{-\theta} \| \langle x \rangle^{\sigma} \lambda(t) \|_{L^2_{x,y}} \\
			&\lesssim \epsilon_0 + \mathcal{B}_T(\kappa)^3 + t^{-\theta} \left(\epsilon_0 + t^{\delta} \mathcal{B}_T (\kappa)^3 \right) \\
			&\lesssim \epsilon_0 + \mathcal{B}_T (\kappa)^3
		\end{align*}
		uniformly in $1\le t \le T$, since $\theta > \delta$.

		Combining these estimates, we obtain
		\[
		\mathcal{B}_T (\kappa) \lesssim \epsilon_0 + \mathcal{B}_T (\kappa)^3,
		\]
		which completes the proof of Proposition \ref{prop:bootstrap}.
	\end{proof}

	\section{Modified scattering}

	In this section, we prove modified scattering and thereby complete the proofs of Theorems \ref{thm:main-kappa} and \ref{thm:main}.

	\subsection{Proof of Theorem~\texorpdfstring{\ref{thm:main-kappa}}{\ref*{thm:main-kappa}}}
	\begin{proof}
		Fix $\delta=(\sigma-d/2)/24$ and set $\theta=6\delta$, as in the proof of Proposition \ref{prop:bootstrap}.
        Then $t^{-1-\theta +3\delta} = t^{-1-3\delta}$ in Lemma~\ref{lem:error}.
		The local theory gives a solution on $[0,2]$ with $\mathcal{B}_2(\kappa) \lesssim \epsilon_0$. 
		By Proposition \ref{prop:bootstrap}, $$\mathcal B_T(\kappa)\le C_0\epsilon_0+C_1\mathcal B_T(\kappa)^3.$$
		Take $C_0$ large enough to include the bound from the local theory and choose $\epsilon_0$ so that $8C_1C_0^2\epsilon_0^2<1$.
        A standard continuity argument then shows that $$\mathcal{B}_T(\kappa) \le 2 C_0 \epsilon_0$$ throughout the lifespan of the solution. 
		The continuation criterion in Lemma \ref{lem:LWP} then extends the solution globally, with $$\mathcal B_T(\kappa)\lesssim\epsilon_0$$ for every $T \ge 1$.
		Interpolating the $L^{\infty}_x L^2_y$ bound with mass conservation gives the dispersive estimates, with the bounds on $0\le t\le1$ supplied by the local theory. 
        The estimate \eqref{eq:T-Hsigma-local} makes the nonlinearity locally Lipschitz in $H^{\sigma}_xL^2_y$. Uniqueness in $C_t H^{\sigma}_x L^2_y$ follows from Duhamel's formula and Gr\"onwall's inequality.

        Combining Remark~\ref{rmk:lambda-C1} and  Lemma~\ref{lem:error}, we obtain
        $\widehat\lambda\in C^1([1,\infty);X)$, with
        \[
            \ii\partial_t\widehat\lambda(t)
            =t^{-1}A[\widehat\lambda(t)]\widehat\lambda(t)
                +\mathcal E[\widehat\lambda](t)
        \]
        in $X$.
        Moreover, $\sup_{t\ge1}\|\widehat\lambda(t)\|_X\lesssim\epsilon_0$
        and $\|\mathcal E[\widehat\lambda](t)\|_X
            \lesssim\epsilon_0^3t^{-1-3\delta}$ for $t\ge1$.

        Set $\widehat\omega(t)=\mathcal U_{\widehat\lambda}(1,t)\widehat\lambda(t)$.
        By Remark~\ref{rem:unitary-preserves-X}, $\widehat\omega\in C^1([1,\infty);X)$ and
        \[
            \ii\partial_t\widehat\omega(t)
            =\mathcal U_{\widehat\lambda}(1,t)\mathcal E[\widehat\lambda](t).
        \]
        Integrating this equation and using the isometry of
        $\mathcal U_{\widehat\lambda}(1,s)$ on $X$, we obtain, for $t_2\ge t_1\ge1$,
        \[
            \|\widehat\omega(t_2)-\widehat\omega(t_1)\|_X
            \le\int_{t_1}^{t_2}\|\mathcal E[\widehat\lambda](s)\|_X\,\dd s
            \lesssim \epsilon_0^3(t_1^{-3\delta}-t_2^{-3\delta}).
        \]
		Hence $\widehat\omega(t)$ has a limit $\widehat\omega_\infty\in X$, with $\|\widehat\omega_\infty\|_{X}\lesssim\epsilon_0$. We also have
		\begin{equation}\label{est:modified-01}
			\|\widehat\omega(t)-\widehat\omega_\infty\|_X
            \le C_d \epsilon_0^3 \,  t^{-3\delta}.
		\end{equation}

		We now construct a solution of the resonant equation that approaches $\widehat\lambda(t)$. Fix $\widehat\lambda$ and $\widehat\omega_\infty$.
        Consider the complete metric space \[
            \begin{aligned}
                Y:=\left\{f\in C([1,\infty);X):
                \sup_{t\ge1}t^{3\delta}\|f(t)-\Fomega_\infty\|_{X}\le  \epsilon_0 \right\}
            \end{aligned}
        \]
        with $$d_Y(f,g):=\sup_{t\ge1}t^{3\delta}\|f(t)-g(t)\|_{X}.$$

        For $f \in Y$, set $\widehat\nu_f(t):= \mathcal U_{\Flam}(t,1)f(t)$ and define
		\[
		\Phi(f)(t)=\widehat\omega_\infty+ \ii\int_t^\infty s^{-1}\mathcal U_{\widehat\lambda}(1,s) (A[\widehat\nu_f(s)]-A[\widehat\lambda(s)])\widehat\nu_f(s)\,\dd s.
		\]
		Using the isometry in Remark~\ref{rem:unitary-preserves-X} and \eqref{est:modified-01}, we estimate
        \[
            \|\widehat\nu_f(t)\|_{X}\lesssim\epsilon_0,
            \qquad \|\widehat\nu_f(t)-\widehat\lambda(t)\|_{X}
            \le( \epsilon_0 +C_d \epsilon_0^3)t^{-3\delta}.
        \]
        Hence Lemma \ref{lem:A} implies
        \[
            \|(A[\widehat\nu_f(t)]-A[\widehat\lambda(t)])\widehat\nu_f(t)\|_{X}
            \lesssim\epsilon_0^2(\epsilon_0 + C_d \epsilon_0^3 )t^{-3\delta}.
        \]
        We thus obtain
		\[
		\sup_{t\ge1}t^{3\delta}\|\Phi(f)(t)-\widehat\omega_\infty\|_X \le C_d \epsilon_0^2 ( \epsilon_0 + C_d \epsilon_0^3).
		\]
        For $f,g\in Y$, use the decomposition
        \[
        \begin{aligned}
        &(A[\widehat\nu_f]-A[\widehat\lambda])\widehat\nu_f
        -(A[\widehat\nu_g]-A[\widehat\lambda])\widehat\nu_g\\
        &\quad=(A[\widehat\nu_f]-A[\widehat\nu_g])\widehat\nu_f
        +(A[\widehat\nu_g]-A[\widehat\lambda])
        (\widehat\nu_f-\widehat\nu_g).
        \end{aligned}
        \]
        Lemma \ref{lem:A} bounds its $X$ norm by $C\epsilon_0^2\|\widehat\nu_f-\widehat\nu_g\|_{X}$.
        Applying the propagator isometry and integrating, we find
        \[
            d_Y(\Phi(f),\Phi(g))\le C_d\epsilon_0^2d_Y(f,g).
        \]

		Choose $\epsilon_0$ sufficiently small that $C_d \epsilon_0^2 \le1/4$.
		Then $\Phi$ maps $Y$ into itself and is a contraction.
		Its unique fixed point $\widehat\eta$ satisfies
		\begin{equation}\label{est:modified-02}
			\|\widehat\eta(t)-\widehat\omega_\infty\|_{X}\le \epsilon_0 t^{-3\delta}.
		\end{equation}
		The integral equation implies $\widehat\eta\in C^1([1,\infty);X)$.
		Define $\widehat\nu(t):=\widehat\nu_{\widehat\eta}(t)=\mathcal U_{\widehat\lambda}(t,1)\widehat\eta(t)$.
		Differentiating in $t$, we compute
		\[
		\ii\partial_t\widehat\eta =t^{-1}\mathcal U_{\widehat\lambda}(1,t)(A[\widehat\nu]-A[\widehat\lambda])\widehat\nu, \qquad \ii\partial_t\widehat\nu=t^{-1}A[\widehat\nu]\widehat\nu.
		\]
		In particular $\widehat\nu_1:=\widehat\nu(1)=\widehat\eta(1)\in X$, and \[
            \|\widehat\nu_1\|_{X}
            \le\|\widehat\omega_\infty\|_{X}+\epsilon_0
            \lesssim\epsilon_0.
        \]
        Uniqueness of the resonant evolution with this initial value follows from Lemma \ref{lem:resonant-GWP}. By the isometry, \eqref{est:modified-01} and \eqref{est:modified-02},
		\[
		\|\widehat\lambda(t)-\widehat\nu(t)\|_{X} \le\|\widehat\omega(t)-\widehat\omega_\infty\|_{X} +\|\widehat\eta(t)-\widehat\omega_\infty\|_{X} \lesssim\epsilon_0t^{-3\delta}.
		\]
		Set $\nu=\mathcal F^{-1}\widehat\nu$.
		By Plancherel,
		\begin{equation}\label{est:kappa-scattering}
			\|\kappa(t)-\ee^{\ii t\Delta/2}\nu(t)\|_{\mathfrak S^2} =\|\lambda(t)-\nu(t)\|_{L_{x,y}^2} =(2\pi)^{-d/2}\|\widehat\lambda(t)-\widehat\nu(t)\|_{L_{k,y}^2} \lesssim\epsilon_0t^{-3\delta}.
		\end{equation}
		This proves the half-density statements.

	\end{proof}

	\subsection{Proof of Theorem~\texorpdfstring{\ref{thm:main}}{\ref*{thm:main}}}
	\begin{proof}
		Let $\kappa_0=\sqrt{\gamma_0}$.
		Observe that
		\[
		\|\langle\nabla\rangle^\sigma\kappa_0\|_{\mathfrak S^2}^2 =\|\langle\nabla\rangle^\sigma\gamma_0 \langle\nabla\rangle^\sigma\|_{\mathfrak S^1},\qquad \|\langle x\rangle^\sigma\kappa_0\|_{\mathfrak S^2}^2 =\|\langle x\rangle^\sigma\gamma_0 \langle x\rangle^\sigma\|_{\mathfrak S^1}.
		\]
		The sum of the two half-density norms is at most $\sqrt{2\epsilon}$.
		Thus Theorem \ref{thm:main-kappa} applies with $\epsilon_0=\sqrt{2\epsilon}$ after choosing $\epsilon_*$ sufficiently small.
		Set $\gamma=\kappa\kappa^*$.
        By Theorem~\ref{thm:main-kappa} and the Schatten--H\"older inequality, $\gamma \in C([0,\infty);\mathfrak{S}^{\sigma,1})$. 
		Since $\sigma\ge2$, it also follows from the half-density equation that $\kappa\in C^1_t\mathfrak S^2$, and therefore $\gamma\in C^1_t\mathfrak S^1$.
		By \eqref{eq:gamma-solves-HF}, $\gamma$ satisfies \eqref{eq:HF}.
		Moreover, $\rho_\gamma(x)=\|\kappa(x,\cdot)\|_{L^2_y}^2$, so $\|\rho_\gamma(t)\|_{L^p_x}=\|\kappa(t)\|_{L_x^{2p}L_y^2}^2$ implies the desired decay estimate.

		We next prove uniqueness for the Cauchy problem.
        Let $J=\langle\nabla\rangle^\sigma$ and
        $B_T=V*\rho_T-\mathsf X_T$.
        For self-adjoint $T\in\mathfrak S^{\sigma,1}$, we factor $T=fg^*$
        with $f,g\in H_x^\sigma L_y^2$ and
        $\|Jf\|_{\mathfrak S^2}\|Jg\|_{\mathfrak S^2}
            =\|T\|_{\mathfrak S^{\sigma,1}}$.
        To obtain this factorization, write the polar decomposition
        $K=U|K|$ of $K=JTJ$ and take
        $f=J^{-1}U|K|^{1/2}$ and $g=J^{-1}|K|^{1/2}$.
        Applying Lemmas~\ref{lem:direct} and \ref{lem:exchange-term}
        to $f$, $g$, we get 
        \[
            \|JB_TJ^{-1}\|_{\mathrm{op}}
            \lesssim\|T\|_{\mathfrak S^{\sigma,1}}.
        \]
        Morevoer, $B_T$ is bounded and self-adjoint on $L^2$, so the same bound holds for $J^{-1}B_T J$ by taking adjoints. The Schatten--H\"older inequality and linearity of $T \mapsto B_T$ therefore imply, for self-adjoint $T,S\in \mathfrak{S}^{\sigma, 1}$,
        \[
            \|[B_T,T]-[B_S,S]\|_{\mathfrak S^{\sigma,1}}
            \lesssim
            \bigl(\|T\|_{\mathfrak S^{\sigma,1}}
                +\|S\|_{\mathfrak S^{\sigma,1}}\bigr)
            \|T-S\|_{\mathfrak S^{\sigma,1}}.
        \]

        Let $T_0>0$, and let
        $\gamma^{(1)},\gamma^{(2)}\in C([0,T_0];\mathfrak S^{\sigma,1})$
        be self-adjoint solutions with the same initial data.
        Since $\sigma \ge 2$ and the free group commutes with $J$, their Duhamel formulas hold in $\mathfrak{S}^{\sigma, 1}$.
        Applying the preceding estimate, we obtain, for $0\le t\le T_0$,
        \[
        \begin{aligned}
            \|\gamma^{(1)}(t)-\gamma^{(2)}(t)\|_{\mathfrak S^{\sigma,1}}
            &\lesssim
                \int_0^t
                \bigl(\|\gamma^{(1)}(s)\|_{\mathfrak S^{\sigma,1}}
                    +\|\gamma^{(2)}(s)\|_{\mathfrak S^{\sigma,1}}\bigr)
                \|\gamma^{(1)}(s)-\gamma^{(2)}(s)\|_{\mathfrak S^{\sigma,1}}\,\dd s.
        \end{aligned}
        \]
        By Gr\"onwall's inequality, $\gamma^{(1)}(t)=\gamma^{(2)}(t)$
        for every $t\in[0,T_0]$.
        Since $T_0$ is arbitrary, this proves uniqueness in
        $C([0,\infty);\mathfrak S^{\sigma,1})$ among self-adjoint solutions.

		Let $\widehat\nu$ be the solution of the resonant equation supplied by Theorem \ref{thm:main-kappa}, and define
		\[
		\nu=\mathcal F^{-1}\widehat\nu,\qquad \Gamma(\tau)=\nu(\ee^\tau)\nu(\ee^\tau)^*,\qquad \tau\ge0.
		\]
		Then $\Gamma(0)=\Gamma_*:=\nu(1)\nu(1)^*$.
        By \eqref{def:Fourier-operator},
		\[
		\begin{aligned}
            G(\tau)
            :=
            \mathfrak F(\Gamma(\tau))
            =
            (2\pi)^{-d}\widehat\nu(\ee^\tau)\widehat\nu(\ee^\tau)^*,\qquad
            \rho_{G(\tau)}(k)
            =
            (2\pi)^{-d}\|\widehat\nu(\ee^\tau,k,\cdot)\|_{L_y^2}^2.
        \end{aligned}
		\]
		The density $\rho_{G(\tau)}$ lies in $L^1\cap L^\infty$ and is independent of $\tau$ by Lemma \ref{lem:resonant-GWP}.
		By Lemma~\ref{lem:potential} and Cauchy--Schwarz, $\| \mathsf{X}_G \|_{\mathrm{op}} \le \| V \ast \rho_G \|_{L^{\infty}}<\infty$.
		Both $V*\rho_G$ and $\mathsf X_G$ are bounded and self-adjoint, so $[V \ast \rho_G - \mathsf{X}_G, G] \in \mathfrak{S}^1$.

		By the definition of $\mathsf{X}_G$, we rewrite \eqref{eqn:res-eqn-02} for $t \ge 1$ as
        \[
            \ii\partial_t\widehat\nu(t)
            =t^{-1}\bigl(V*\rho_{G(\log t)}-\mathsf X_{G(\log t)}\bigr)
                \widehat\nu(t).
        \]
        Here the operator on the right hand side acts on $L^2_k$. 
        Since $\widehat\nu\in C^1([1,\infty);X)$, both $G$ and $\Gamma$ belong to $C^1([0,\infty);\mathfrak S^1)$. Differentiating the definition of $G$, we get \eqref{eqn:res-eqn-gamma}.

		Finally, conservation of the $\mathfrak S^2$ norms of
		$\kappa$ and $\nu$ implies
		$\|\kappa(t)\|_{\mathfrak S^2}+\|\nu(t)\|_{\mathfrak S^2}\lesssim\sqrt\epsilon$.
		By Schatten--H\"older and \eqref{est:kappa-scattering},
		\[
		\begin{aligned}
		\|\gamma(t)-\ee^{\ii t\Delta/2}\Gamma(\log t )\ee^{-\ii t\Delta/2}\|_{\mathfrak S^1}
		&\le\|\kappa(t)-\ee^{\ii t\Delta/2}\nu(t)\|_{\mathfrak S^2}
		\bigl(\|\kappa(t)\|_{\mathfrak S^2}+\|\nu(t)\|_{\mathfrak S^2}\bigr)
		\lesssim\epsilon t^{-3\delta}.
		\end{aligned}
		\]
		Both $\operatorname{Tr}\Gamma(\log t )=\|\nu(t)\|_{\mathfrak S^2}^2$ and $\operatorname{Tr}\gamma(t)=\operatorname{Tr}\gamma_0$ are constant in time. By the preceding estimate, $\operatorname{Tr}\Gamma(\tau)=\operatorname{Tr}\gamma_0$ for every $\tau \ge 0$.
	\end{proof}

	\section{Absence of linear scattering}\label{sec:no-scattering}

    To prove Theorem~\ref{thm:no-linear-scattering}, we first classify the stationary nonnegative trace-class operators with bounded density.
    Under the assumption on the time decay of density in that theorem, linear scattering implies that $\mathfrak F(\gamma_+)$ has bounded density and is stationary.
    Proposition~\ref{prop:stationary-coulomb} then gives $\operatorname{rank}\gamma_+\le1$, contradicting conservation of the eigenvalues when $\operatorname{rank}\gamma_0\ge2$.

	\subsection{Stationary solutions of the resonant equation}

	\begin{proposition}
		\label{prop:stationary-coulomb}
		Let $d\ge2$, $0<\theta\le1$, and $V(x)=|x|^{-\theta}$. Let $G\ge0$ be a trace-class operator on $L^2(\mathbb R^d)$ such that $\rho_G \in L^{\infty}(\mathbb{R}^d)$.
		Then $V*\rho_G-\mathsf X_G$ is bounded and self-adjoint, and
		\[
		[V*\rho_G-\mathsf X_G,G]=0\quad\Longleftrightarrow\quad \operatorname{rank}G\le1.
        \]
	\end{proposition}

	\begin{proof}
		The case $G=0$ is immediate. Write $\rho =\rho_G$. Since $\rho \in L^1 \cap L^{\infty}$, we argue as in the proof of Lemma~\ref{lem:potential} to obtain
        \[
            \| V \ast \rho \|_{L^{\infty}} \lesssim \| \rho\|_{L^1}^{1-\theta/d} \| \rho \|_{L^{\infty}}^{\theta/d}.
        \]
        Since $G$ is nonnegative, we have $|G(x,y)|^2 \le \rho(x) \rho(y)$. Thus, for $f \in L^2$,
        \begin{equation*}
        \begin{aligned}
            | \langle f , \mathsf{X}_G f \rangle |
            \le
            \iint V(x-y) \sqrt{\rho(x)\rho(y)}  \, |f(x) f(y)| \, \dd x \, \dd y 
            \le
            \int (V \ast \rho)(x) |f(x)|^2 \, \dd x,
        \end{aligned}
        \end{equation*}
        where we used $2\sqrt{\rho(x)\rho(y)}|f(x)f(y)| \le\rho(y)|f(x)|^2+\rho(x)|f(y)|^2$. 
        Thus $\mathsf{X}_G$ is bounded and self-adjoint, with $\|\mathsf X_G\|_{\mathrm{op}}\le\|V*\rho\|_\infty <\infty$.
		Define
		\[
		W(x,y)=\rho(x)\rho(y)-|G(x,y)|^2\ge0,\qquad Q=\iint V(x-y)W(x,y)\,\dd x\,\dd y.
		\]
        Then $ Q \le \|\rho\|_{L^1}\|V\ast\rho\|_{L^{\infty}} <\infty$, and
        \[
            \iint W(x,y) \, \dd x \, \dd y
            =
            (\operatorname{Tr} G)^2 - \operatorname{Tr}(G^2)
            =
            \sum_{j\neq k} \lambda_j \lambda_k,
        \]
        where $\lambda_j$ are the positive eigenvalues of $G$.
        Since $V>0$ almost everywhere and $W\ge0$, we have $Q=0$ if and only if $\operatorname{rank}G\le1$.

		Assume $[V*\rho-\mathsf X_G,G]=0$.
		Each eigenspace of $G$ corresponding to a positive eigenvalue is finite dimensional and invariant under $V*\rho-\mathsf X_G$.
		Thus we may choose a countable orthonormal family $(\phi_j)_{j \in J}$ such that
        \[
            G = \sum_{j\in J} \lambda_j \lvert \phi_j \rangle \langle \phi_j \rvert,
            \qquad
            \bigl(V*\rho-\mathsf X_G\bigr)\phi_j=\mu_j\phi_j,
            \qquad
            \lambda_j >0,
        \]
        where $\mu_j \in \mathbb{R}$ and $|\mu_j|\le\|V*\rho-\mathsf X_G\|_{\mathrm{op}}$.
		Let $\mathcal H=\ell^2(J)$ and set
		\[
		\begin{aligned}
			u(x)&=(\sqrt{\lambda_j}\phi_j(x))_{j \in J}, & R(x)&=|u(x)\rangle\langle u(x)|,\\
			T(x)&=\int V(x-y)R(y)\,\dd y, & M&=\operatorname{diag}(\mu_j).
		\end{aligned}
		\]
        For almost every $x$, we have
        \[
            \|u(x)\|_{\mathcal H}^2
            =\|R(x)\|_{\mathfrak S^1(\mathcal H)}
            =\|R(x)\|_{\mathfrak S^2(\mathcal H)}
            =\rho(x).
        \]
        In particular, $u\in L_x^2(\mathcal H)\cap L_x^\infty(\mathcal H)$, and $\|T(x)\|_{\mathfrak S^1(\mathcal H)}\le(V*\rho)(x)$.
        Moreover, for almost every $x$, $(T(x)u(x))_j=\sqrt{\lambda_j}(\mathsf X_G\phi_j)(x)$.
        Consequently, $C(x):=(V*\rho)(x)I-T(x)-M$ is bounded and self-adjoint for almost every $x$. By the eigenvalue equations, for almost every $x$,
        \begin{equation}\label{eq:stationary-eigval-eqn}
            C(x)u(x)=0.
        \end{equation}

		We next establish the weak derivatives needed to use this identity. For $\varepsilon>0$, set $V_\varepsilon(z)=(|z|^2+\varepsilon^2)^{-\theta/2}$, and $T_\varepsilon=V_\varepsilon*R$.
        If $d>\theta+1$, then $\nabla V$ is locally integrable and
        \[
            \int|x-y|^{-\theta-1}\rho(y)\,\dd y
            \lesssim \|\rho\|_\infty+\|\rho\|_1.
        \]
        Since $\|R(y)\|_{\mathfrak S^2(\mathcal H)}=\rho(y)$, the same bound justifies differentiation of both convolutions. Dominated convergence then proves \eqref{eqn:derivative-convergence} in this case.

        For $d=2$ and $\theta=1$, we instead use
        \[
            \widehat{\partial_jV_\varepsilon}(\xi)
            =2\pi\ii\frac{\xi_j}{|\xi|}\ee^{-\varepsilon|\xi|}.
        \]
        Since $\rho\in L^2$ and $R\in L_x^2(\mathfrak S^2(\mathcal H))$, the regularized derivatives converge in the corresponding $L^2$ spaces by Plancherel and dominated convergence.
        Also, $\|V-V_\varepsilon\|_{L^1(\mathbb R^2)}=2\pi\varepsilon$, so the potentials converge uniformly. Passing to the limit in the distributions identifies the $L^2$ limits as their weak derivatives.

        Thus, in both cases, $V \ast \rho \in W^{1,2}_{\mathrm{loc}}(\mathbb{R}^d)$ and $T \in W^{1,2}_{\mathrm{loc}}(\mathbb{R}^d;\mathfrak{S}^2(\mathcal{H}))$, and the regularized derivatives satisfy
        \begin{equation}\label{eqn:derivative-convergence}
        \begin{aligned}
            \partial_j(V_\varepsilon*\rho)&\rightarrow\partial_j(V*\rho)
            &&\text{in }L^2_{\mathrm{loc}},\\
            \partial_jT_\varepsilon&\rightarrow\partial_jT
            &&\text{in }L^2_{\mathrm{loc}}(\mathfrak S^2(\mathcal H)).
        \end{aligned}
        \end{equation}
        Since $C=(V*\rho)I-T-M$, the scalar and Hilbert--Schmidt derivative bounds imply $\|\partial_jC(\cdot)\|_{\mathrm{op}}\in L^2_{\mathrm{loc}}$.

        Fix $j\in\{1,\ldots,d\}$ and write
        \[
            u_h(x)=u(x+he_j),
            \qquad D_hC(x)=\frac{C(x+he_j)-C(x)}h,
            \qquad h\ne0.
        \]
        For each fixed $h\ne0$, self-adjointness and \eqref{eq:stationary-eigval-eqn} yield
        \[
        \begin{aligned}
            \langle u_h,(D_hC)u\rangle_{\mathcal H}
            &=h^{-1}\bigl(\langle C(x+he_j)u_h,u\rangle_{\mathcal H}
            -\langle u_h,C(x)u\rangle_{\mathcal H}\bigr)
            =0
        \end{aligned}
        \]
        for almost every $x$.
        Since $u$ and $u_h$ are uniformly bounded, Cauchy--Schwarz implies
        \[
            \langle u_h, (D_h C)u \rangle_{\mathcal{H}}
            \to
            \langle u , (\partial_j C)u \rangle_{\mathcal{H}}
        \]
        in $L^1_{\mathrm{loc}}$.
        Hence, for almost every $x$,
        \begin{equation*}
        \begin{aligned}
            F_j&:=\rho\,\partial_j(V*\rho)-\operatorname{Tr}_{\mathcal H}(R\,\partial_jT)
            =\langle u,(\partial_jC)u\rangle_{\mathcal H}=0.
        \end{aligned}
        \end{equation*}

        Let $b\in C_c^1(\mathbb R^d;\mathbb R^d)$ and define $F_\varepsilon$ by its components
        \[
            F_{\varepsilon,j}
            =\rho\,\partial_j(V_\varepsilon*\rho)
            -\operatorname{Tr}_{\mathcal H}(R\,\partial_jT_\varepsilon).
        \]
        By \eqref{eqn:derivative-convergence} and boundedness of $\rho$ and $R$, we have $F_\varepsilon\rightarrow F=0$ in $L^1_{\mathrm{loc}}$.
        Since $\nabla V_\varepsilon$ is bounded and $\operatorname{Tr}_{\mathcal H}(R(x)R(y))=|G(x,y)|^2$ for almost every $(x,y)$, we symmetrize in $x$ and $y$ to obtain
        \begin{equation}\label{eqn:Qvanishes}
            \int b(x)\cdot F_\varepsilon(x)\,\dd x
            =\frac12\iint(b(x)-b(y))\cdot\nabla V_\varepsilon(x-y)W(x,y)\,\dd x\,\dd y.
        \end{equation}
        The left-hand side tends to zero. On the right-hand side,
        \[
            |(b(x)-b(y))\cdot\nabla V_\varepsilon(x-y)|
            \le\theta\|\nabla b\|_\infty V(x-y).
        \]
        Since $V(x-y)W(x,y)$ is integrable, dominated convergence in \eqref{eqn:Qvanishes} yields
        \[
            0=\frac12\iint(b(x)-b(y))\cdot\nabla V(x-y)W(x,y)\,\dd x\,\dd y.
        \]
        Take $b_L(x)=x\chi(x/L)$, where $\chi\in C_c^\infty(\mathbb R^d)$ equals one near the origin. These vector fields converge pointwise to $x$ and have uniformly bounded gradients.
        Applying dominated convergence and using $z\cdot\nabla V(z)=-\theta V(z)$, we obtain
        \[
            0=\frac12\iint(x-y)\cdot\nabla V(x-y)W(x,y)\,\dd x\,\dd y
            =-\frac\theta2Q.
        \]
        Thus $Q=0$, and $\operatorname{rank}G\le1$.

        Conversely, if $\operatorname{rank}G=1$, write $G=\lambda|\phi\rangle\langle\phi|$, with $\lambda>0$ and $\|\phi\|_2=1$. Then $\mathsf X_G\phi=(V*\rho_G)\phi$, and hence
        \[
            (V*\rho_G-\mathsf X_G)G=G(V*\rho_G-\mathsf X_G)=0.
        \]
	\end{proof}

\subsection{Proof of Theorem~\texorpdfstring{\ref{thm:no-linear-scattering}}{\ref*{thm:no-linear-scattering}}}
\begin{proof}
For $t > 0$, define
\[
 (M_tf)(x)=\ee^{\ii|x|^2/(2t)}f(x),
 \qquad
 (D_tf)(x)=(\ii t)^{-d/2}f(x/t).
\]
Assume, for contradiction, that $\gamma_+\in\mathfrak S^1$ satisfies
\eqref{eq:LS}, and set $\Lambda(t):=\ee^{-\ii t\Delta/2}\gamma(t)\ee^{\ii t\Delta/2}$. Then
\begin{equation}
 \|\Lambda(t)-\gamma_+\|_{\mathfrak S^1}
 =\|\gamma(t)-\ee^{\ii t\Delta/2}\gamma_+\ee^{-\ii t\Delta/2}\|_{\mathfrak S^1}\to 0.
 \label{equiv:LS}
\end{equation}
Since $\gamma(t) \ge 0$, we also have $\gamma_+ \ge 0$. The mild formulation gives
\begin{equation}
\begin{aligned}
 \Lambda(t)-\Lambda(s)
 &= - \ii \int_s^t \ee^{-\ii\tau\Delta/2}
   [V*\rho_{\gamma(\tau)}-\mathsf X_{\gamma(\tau)},\gamma(\tau)]
   \ee^{\ii\tau\Delta/2}\, \dd\tau,
 \qquad 0\le s\le t.
 \end{aligned}
 \label{eq:profile-duhamel}
\end{equation}
From the factorization $\ee^{\ii t \Delta/2} = M_t D_t \mathcal{F} M_t$, we  may write
\begin{equation}
\begin{aligned}
\Upsilon(t)
:={}&\mathfrak F\bigl(M_t\Lambda(t)M_t^{-1}\bigr)
=D_t^{-1}M_t^{-1}\gamma(t)M_tD_t,\\
\Upsilon_\infty
:={}&\mathfrak F(\gamma_+),\\
\gamma(t)
={}&M_tD_t\Upsilon(t)D_t^{-1}M_t^{-1}.
\end{aligned}
\label{eq:rescaled-state}
\end{equation}
Since $M_t$ is unitary and $M_t \to I$ strongly on $L^2$, \eqref{equiv:LS} implies
\begin{equation}
\begin{aligned}
\|\Upsilon(t)-\Upsilon_\infty\|_{\mathfrak S^1}
&\le 
\|\Lambda(t)-\gamma_+\|_{\mathfrak S^1}
+\|M_t\gamma_+M_t^{-1}-\gamma_+\|_{\mathfrak S^1}
\rightarrow0.
\end{aligned}
\label{eq:Upsilon-convergence}
\end{equation}
Moreover, \eqref{eq:rescaled-state} and
\eqref{eq:optimal-density-decay} give
\begin{equation}
\rho_{\Upsilon(t)}(v)=t^d\rho_{\gamma(t)}(tv),
\qquad
\sup_{t\ge1}\|\rho_{\Upsilon(t)}\|_{L^\infty}\le C.
\label{eq:rescaled-density}
\end{equation}
We also have
$\|\rho_{\Upsilon(t)}-\rho_{\Upsilon_\infty}\|_{L^1}
\le\|\Upsilon(t)-\Upsilon_\infty\|_{\mathfrak S^1}\rightarrow0$.
It follows that
\begin{equation}
\rho_{\Upsilon_\infty}\in L^\infty,
\qquad
\|\rho_{\Upsilon_\infty}\|_{L^\infty}\le C.
\label{eq:limiting-density}
\end{equation}
We claim that, if
$A,B\ge0$ belong to $\mathfrak S^1$ and
$\|\rho_A\|_{L^\infty},\|\rho_B\|_{L^\infty}\le M$, then, for every $R>0$,
\begin{equation}
 \|V*(\rho_A-\rho_B)-\mathsf X_A+\mathsf X_B\|_{\mathrm{op}}
 \lesssim MR^{d-1}+R^{-1}\|A-B\|_{\mathfrak S^1}.
 \label{eq:field-continuity}
\end{equation}
As in Lemma~\ref{lem:potential}, write $V=V_{<R}+V_{\ge R}$.
For the direct term, Young's inequality implies
\[
\begin{aligned}
 \|V_{<R}*(\rho_A-\rho_B)\|_{L^\infty}
 &\le \|V_{<R}\|_{L^1}\|\rho_A-\rho_B\|_{L^\infty}
 \lesssim MR^{d-1},\\
 \|V_{\ge R}*(\rho_A-\rho_B)\|_{L^\infty}
 &\le \|V_{\ge R}\|_{L^\infty}\|\rho_A-\rho_B\|_{L^1}
 \le R^{-1}\|A-B\|_{\mathfrak S^1}.
\end{aligned}
\]
For the exchange term, let $\mathsf X_D^{<R}$ and
$\mathsf X_D^{\ge R}$ denote the operators with kernels
$V_{<R}(x-y)D(x,y)$ and $V_{\ge R}(x-y)D(x,y)$, respectively.
As in the proof of Proposition~\ref{prop:stationary-coulomb}, we have
$\|\mathsf X_A^{<R}\|_{\mathrm{op}}\le\|V_{<R}*\rho_A\|_{L^\infty}$.
The same bound holds for $B$, so
\[
 \|\mathsf X_{A-B}^{<R}\|_{\mathrm{op}}
 \le \|V_{<R}*\rho_A\|_{L^\infty}
      +\|V_{<R}*\rho_B\|_{L^\infty}
 \lesssim MR^{d-1}.
\]
For the other part, we estimate the Hilbert--Schmidt norm using
$\|V_{\ge R}\|_{L^\infty}=R^{-1}$:
\[
 \|\mathsf X_{A-B}^{\ge R}\|_{\mathrm{op}}
 \le\|\mathsf X_{A-B}^{\ge R}\|_{\mathfrak S^2}
 \le R^{-1}\|A-B\|_{\mathfrak S^2}
 \le R^{-1}\|A-B\|_{\mathfrak S^1}.
\]
Combining these estimates, we obtain \eqref{eq:field-continuity}.

It follows from \eqref{eq:field-continuity} and \eqref{eq:Upsilon-convergence} that
\begin{equation}
 \left\|
 V*(\rho_{\Upsilon(t)}-\rho_{\Upsilon_\infty})
 -\mathsf X_{\Upsilon(t)}+\mathsf X_{\Upsilon_\infty}
 \right\|_{\mathrm{op}}\to0.
 \label{eq:field-limit}
\end{equation}
By \eqref{eq:Upsilon-convergence} and \eqref{eq:field-limit},
\begin{equation}
 \begin{aligned}
 &\Bigl\|[V*\rho_{\Upsilon(t)}-\mathsf X_{\Upsilon(t)},\Upsilon(t)]
       -[V*\rho_{\Upsilon_\infty}-\mathsf X_{\Upsilon_\infty},
          \Upsilon_\infty]\Bigr\|_{\mathfrak S^1}\rightarrow0.
 \end{aligned}
 \label{eq:commutator-limit}
\end{equation}

Using the definitions of $M_t$, $D_t$, and $\Upsilon(t)$, together with
$V(tz)=t^{-1}V(z)$, we obtain
\[
 D_t^{-1} M_t^{-1}
 \bigl(V*\rho_{\gamma(t)}-\mathsf X_{\gamma(t)}\bigr)M_tD_t
 =t^{-1}\bigl(V*\rho_{\Upsilon(t)}-\mathsf X_{\Upsilon(t)}\bigr).
\]
Hence
\begin{equation}
 \begin{aligned}
 &t\ee^{-\ii t\Delta/2}
 [V*\rho_{\gamma(t)}-\mathsf X_{\gamma(t)},\gamma(t)]
 \ee^{\ii t\Delta/2}
 =M_t^{-1} \mathcal F^{-1}
 [V*\rho_{\Upsilon(t)}-\mathsf X_{\Upsilon(t)},\Upsilon(t)]
 \mathcal F M_t.
 \end{aligned}
 \label{eq:exact-commutator}
\end{equation}
Using \eqref{eq:profile-duhamel} and \eqref{eq:exact-commutator}, we obtain
\begin{equation}
 \ii\partial_t\Lambda(t)
 =\frac1t \mathcal F^{-1}
 [V*\rho_{\Upsilon_\infty}-\mathsf X_{\Upsilon_\infty},\Upsilon_\infty]
 \mathcal F+\frac1t \mathcal E(t).
 \label{eq:profile-resonant}
\end{equation}

By \eqref{eq:commutator-limit} and the strong convergence $M_t\rightarrow I$ on $L^2$, the remainder satisfies
$\|\mathcal E(t)\|_{\mathfrak S^1}\rightarrow0$.
Integrating \eqref{eq:profile-resonant} over $[T,2T]$, for $T\ge1$, we obtain
\begin{equation}
\begin{aligned}
\Lambda(2T)-\Lambda(T)
&=-\ii (\log2)\mathcal F^{-1}
[V*\rho_{\Upsilon_\infty}-\mathsf X_{\Upsilon_\infty},\Upsilon_\infty]
\mathcal F
-\ii \int_T^{2T}\mathcal E(s)\,\frac{ \dd s}{s}.
\end{aligned}
\label{eq:dyadic-identity}
\end{equation}
The error satisfies
\[
 \left\|\int_T^{2T}\mathcal E(s)\,\frac{\dd s}{s}\right\|_{\mathfrak S^1}
 \lesssim \sup_{s\ge T}\|\mathcal E(s)\|_{\mathfrak S^1}
 \rightarrow0.
\]
Letting $T\rightarrow\infty$ in \eqref{eq:dyadic-identity} and using \eqref{equiv:LS}, we obtain
\begin{equation}
 [V*\rho_{\Upsilon_\infty}-\mathsf X_{\Upsilon_\infty},\Upsilon_\infty]=0.
 \label{eq:asymptotic-stationarity}
\end{equation}
By \eqref{eq:asymptotic-stationarity}, \eqref{eq:limiting-density} and Proposition~\ref{prop:stationary-coulomb},
\begin{equation}
 \operatorname{rank}\gamma_+
 =\operatorname{rank}\Upsilon_\infty\le1.
 \label{eq:rank-obstruction}
\end{equation}

Let $\lambda_2(A)$ denote the second largest eigenvalue of a positive
trace-class operator $A$, counted with multiplicity. We set $\lambda_2(A)=0$ when $\operatorname{rank} A\le1$. By conservation of the spectrum,
$\lambda_2(\Lambda(t))
=\lambda_2(\gamma(t))
=\lambda_2(\gamma_0)>0$
for $t\ge0$.
The min--max principle and \eqref{equiv:LS} give
\[
 |\lambda_2(\Lambda(t))-\lambda_2(\gamma_+)|
 \le\|\Lambda(t)-\gamma_+\|_{\mathrm{op}}
 \le\|\Lambda(t)-\gamma_+\|_{\mathfrak S^1}
 \rightarrow0.
\]
Consequently $\lambda_2(\gamma_+)=\lambda_2(\gamma_0)>0$, contradicting
\eqref{eq:rank-obstruction}.
\end{proof}    

\begin{remark}[Extension to $0<\theta<1$]
\label{rem:power-extension-proof}
Let $\gamma$ satisfy the hypotheses of Remark~\ref{rem:power-extension-intro}.
By Proposition~\ref{prop:stationary-coulomb}, for $A\ge0$ in $\mathfrak S^1$
with $\rho_A\in L^\infty$, we have
\begin{equation}
 [V_\theta*\rho_A-\mathsf X_A^{(\theta)},A]=0
 \quad\Longleftrightarrow\quad
 \operatorname{rank} A\le1.
 \label{eq:rigidity-theta}
\end{equation}
Assume, for contradiction, that $\gamma_+\in\mathfrak S^1$ satisfies
\eqref{eq:LS}, and define $\Lambda(t)$, $\Upsilon(t)$, and $\Upsilon_\infty$
as in the preceding proof. Then
\eqref{eq:Upsilon-convergence}--\eqref{eq:limiting-density} remain valid,
and $\gamma_+,\Upsilon_\infty\ge0$.

If $A,B\ge0$ belong to $\mathfrak S^1$ and
$\|\rho_A\|_{L^\infty},\|\rho_B\|_{L^\infty}\le M$, then, for every $R>0$,
\begin{equation}
 \|V_\theta*(\rho_A-\rho_B)
       -\mathsf X_A^{(\theta)}+\mathsf X_B^{(\theta)}\|_{\mathrm{op}}
 \lesssim MR^{d-\theta}
       +R^{-\theta}\|A-B\|_{\mathfrak S^1}.
 \label{eq:field-continuity-theta}
\end{equation}
Indeed,
\[
 \|V_\theta\mathbf1_{\{|x|<R\}}\|_{L^1}
 \lesssim R^{d-\theta},
 \qquad
 \|V_\theta\mathbf1_{\{|x|\ge R\}}\|_{L^\infty}=R^{-\theta},
\]
so the proof of \eqref{eq:field-continuity} applies with
$V_\theta$ in place of $V$.
Estimate \eqref{eq:field-continuity-theta} then implies the analogues of
\eqref{eq:field-limit} and \eqref{eq:commutator-limit}, as in the preceding proof.

Using the definitions of $M_t$, $D_t$, and $\Upsilon(t)$, together with
$V_\theta(tz)=t^{-\theta}V_\theta(z)$, we obtain
\[
 D_t^{-1}M_t^{-1}
 \bigl(V_\theta*\rho_{\gamma(t)}
       -\mathsf X_{\gamma(t)}^{(\theta)}\bigr)M_tD_t
 =t^{-\theta}
 \bigl(V_\theta*\rho_{\Upsilon(t)}
       -\mathsf X_{\Upsilon(t)}^{(\theta)}\bigr).
\]
Hence
\begin{equation*}
 t^\theta\ee^{-\ii t\Delta/2}
 [V_\theta*\rho_{\gamma(t)}-\mathsf X_{\gamma(t)}^{(\theta)},\gamma(t)]
 \ee^{\ii t\Delta/2}
 =M_t^{-1}\mathcal F^{-1}
 [V_\theta*\rho_{\Upsilon(t)}-\mathsf X_{\Upsilon(t)}^{(\theta)},
  \Upsilon(t)]\mathcal F M_t.
\end{equation*}
The mild formulation of \eqref{eq:HF-theta} yields, for almost every $t \ge 1$,
\begin{equation}
 \ii\partial_t\Lambda(t)
 =t^{-\theta}\mathcal F^{-1}
 [V_\theta*\rho_{\Upsilon_\infty}
       -\mathsf X_{\Upsilon_\infty}^{(\theta)},\Upsilon_\infty]\mathcal F
 +t^{-\theta}\mathcal E_\theta(t),
 \label{eq:profile-resonant-theta}
\end{equation}
where $\|\mathcal E_\theta(t)\|_{\mathfrak S^1}\to0$,
by the same argument as in the preceding proof.

For $T\ge1$, set $J_{\theta}(T) := \int_T^{2T} s^{-\theta} \, \dd s$.
Integrating \eqref{eq:profile-resonant-theta} over $[T,2T]$, we obtain
\begin{equation}
    \Lambda(2T) - \Lambda(T)
    =
    -\ii J_{\theta}(T) \mathcal{F}^{-1} [V_{\theta} \ast \rho_{\Upsilon_{\infty}} - \mathsf{X}_{\Upsilon_{\infty}}^{(\theta)} , \Upsilon_{\infty} ] \mathcal{F}
    - \ii \int_T^{2T} s^{-\theta} \mathcal{E}_{\theta}(s) \, \dd s.
\label{eq:normalized-dyadic-theta}
\end{equation}
Since $T^{1-\theta}\lesssim J_\theta(T)$, \eqref{equiv:LS} and
\eqref{eq:normalized-dyadic-theta} imply
\[
 \left\|\mathcal F^{-1}
 [V_\theta*\rho_{\Upsilon_\infty} -\mathsf X_{\Upsilon_\infty}^{(\theta)},\Upsilon_\infty]\mathcal F
 \right\|_{\mathfrak S^1}
 \le 
 J_{\theta}(T)^{-1} \|\Lambda(2T)-\Lambda(T)\|_{\mathfrak S^1}
 +\sup_{s\ge T}\|\mathcal E_\theta(s)\|_{\mathfrak S^1}
 \to0.
\]
We thus obtain
\[
 [V_\theta*\rho_{\Upsilon_\infty}
       -\mathsf X_{\Upsilon_\infty}^{(\theta)},\Upsilon_\infty]=0.
\]
By \eqref{eq:rigidity-theta},
$\operatorname{rank}\gamma_+=\operatorname{rank}\Upsilon_\infty\le1$.
Conservation of the spectrum under \eqref{eq:HF-theta}, the min--max principle,
and \eqref{equiv:LS} imply
$\lambda_2(\gamma_+)=\lambda_2(\gamma_0)>0$,
which is a contradiction.
\end{remark}

\appendix
\section{Subcritical Hartree--Fock Equations}\label{app:Subcritical HFE}
In this appendix, we consider \eqref{eq:HF} with $V(x) = |x|^{-\theta}$, where $d\ge 2$ and $1 < \theta <d$.

\begin{theorem}\label{sr:thm:main}
Let $d\ge2$, $1<\theta<d$, and $\sigma=\lfloor d/2\rfloor+1$.
There exists $\epsilon_*=\epsilon_*(d,\theta)>0$ such that, if
$\gamma_0\ge0$ and
\[
    \epsilon:=\|\langle\nabla\rangle^\sigma\gamma_0
        \langle\nabla\rangle^\sigma\|_{\mathfrak S^1}
    +\|\langle x\rangle^\sigma\gamma_0
        \langle x\rangle^\sigma\|_{\mathfrak S^1}
    \le\epsilon_*,
\]
then the Cauchy problem \eqref{eq:HF} with $\gamma(0)=\gamma_0$
has a unique nonnegative global solution
$\gamma\in C([0,\infty);\mathfrak S^{\sigma,1})$, and satisfies $|\mathrm J|^\sigma\gamma|\mathrm J|^\sigma \in C([0,\infty);\mathfrak S^1)$.
Moreover, for $1\le p\le \infty$ and $t \ge 0$, 
\begin{equation}\label{sr:eq:density-decay}
    \|\rho_\gamma(t)\|_{L^p_x}
    \lesssim\epsilon\langle t\rangle^{-d(1-1/p)},
\end{equation}
In addition, there exists a unique nonnegative $\gamma_{\infty}\in\mathfrak S^{\sigma,1}$,
with $\langle x\rangle^\sigma\gamma_{\infty}\langle x\rangle^\sigma\in\mathfrak S^1$,
such that
\begin{equation}\label{sr:eq:density-scattering}
    \|\gamma(t)-\ee^{\ii t\Delta/2}\gamma_{\infty}\ee^{-\ii t\Delta/2}\|_{\mathfrak S^{\sigma,1}}
    \lesssim\epsilon^2 \, t^{1-\theta}
\end{equation}
for $t \ge 1$.
\end{theorem}

We prove Theorem~\ref{sr:thm:main} using the factorization $\gamma=\kappa\kappa^*$, where $\kappa(t,x,y)$
satisfies \eqref{eqn:kappa}.

\begin{theorem}\label{sr:thm:main-kappa}
Let $d\ge2$, $1<\theta<d$, and $\sigma=\lfloor d/2\rfloor+1$.
If
$\epsilon_0:=\|\kappa_0\|_{\Sigma_x^\sigma L_y^2}$
is sufficiently small, then the Cauchy problem \eqref{eqn:kappa} with
$\kappa(0)=\kappa_0$ has a unique global solution
$\kappa\in C([0,\infty);\Sigma_x^\sigma L_y^2)$.
Moreover, for $2\le p\le\infty$ and $t\ge0$,
\begin{equation}\label{sr:eq:decay}
    \|\kappa(t)\|_{L_x^pL_y^2}
    \lesssim\epsilon_0\langle t\rangle^{-d(1/2-1/p)}.
\end{equation}
In addition, there exists a unique $\kappa_{\infty}\in\Sigma_x^\sigma L_y^2$ such that
\begin{equation}\label{sr:eq:scattering}
    \|\ee^{-\ii t\Delta/2}\kappa(t)-\kappa_{\infty}\|_{\Sigma_x^\sigma L_y^2}
    \lesssim\epsilon_0^3 \, t^{1-\theta}
\end{equation}
for $t \ge 1$.
\end{theorem}

For $R>0$, write
\begin{equation}\label{sr:eq:split}
    V_{<R}=V\mathbf1_{\{|x|<R\}},\qquad
    V_{\ge R}=V\mathbf1_{\{|x|\ge R\}},\qquad
    V=V_{<R}+V_{\ge R}.
\end{equation}
Thus $V_{<1}\in L^1$ and $V_{\ge1}\in L^\infty$.

For local well-posedness, we use the Sobolev estimates \eqref{est:direct-Hsigma} and \eqref{est:exchange-Hsigma}, whose proofs depend on \(V\) only through these bounds.
Thus \eqref{eq:T-Sigma-local} and the conclusions of
Lemma~\ref{lem:LWP} and the following two remarks also hold for this potential.
Mass conservation follows from self-adjointness of $H_\gamma$,
as in the proof of Lemma~\ref{lem:energy}.

\begin{lemma}\label{sr:lem:dispersion}
Let $m>d/2$ and $\langle x\rangle^m f\in L^2_{x,y}$. Then, for $t>0$,
\begin{equation}\label{sr:eq:profile-dispersion}
    \|\ee^{\ii t\Delta/2}f\|_{L_x^\infty L_y^2}
    \lesssim_m t^{-d/2}\|\langle x\rangle^m f\|_{L^2_{x,y}}.
\end{equation}
Consequently, for $0<t<T_*$,
\begin{equation}\label{sr:eq:dispersion-energy}
    t^{d/2}\|\kappa(t)\|_{L_x^\infty L_y^2}
    \lesssim\|\kappa(t)\|_{L^2_{x,y}}
    +\||\mathrm J_x(t)|^\sigma\kappa(t)\|_{L^2_{x,y}}.
\end{equation}
\end{lemma}

\begin{proof}
The free Schr\"odinger kernel and Cauchy--Schwarz give
\[
    \|\ee^{\ii t\Delta/2}f\|_{L_x^\infty L_y^2}
    \lesssim t^{-d/2}\|f\|_{L_x^1L_y^2}
    \lesssim_m t^{-d/2}\|\langle x\rangle^m f\|_{L^2_{x,y}}.
\]
Taking $f=\lambda(t)$ and $m=\sigma>d/2$ proves
\eqref{sr:eq:dispersion-energy}.
\end{proof}

\begin{lemma}\label{sr:lem:nonlinear}
For $u\in H_x^\sigma L_y^2$ and $R>0$,
\begin{equation}\label{sr:eq:nonlinear-bounds}
\begin{aligned}
    \|\mathcal T(u,u,u)\|_{L^2_{x,y}}
    &\lesssim
    \left(R^{d-\theta}\|u\|_{L_x^\infty L_y^2}^2
    +R^{-\theta}\|u\|_{L^2_{x,y}}^2\right)\|u\|_{L^2_{x,y}},\\
    \|\mathcal T(u,u,u)\|_{\dot H_x^\sigma L_y^2}
    &\lesssim
    \left(R^{d-\theta}\|u\|_{L_x^\infty L_y^2}^2
    +R^{-\theta}\|u\|_{L^2_{x,y}}^2\right)
    \|u\|_{\dot H_x^\sigma L_y^2}.
\end{aligned}
\end{equation}
\end{lemma}

\begin{proof}
Since $\|V_{<R}\|_{L^1}\lesssim R^{d-\theta}$ and $\|V_{\ge R}\|_{L^\infty}=R^{-\theta}$, the argument in the proof of Lemma~\ref{lem:trilinear-cutoff} proves \eqref{sr:eq:nonlinear-bounds}.
\end{proof}

\begin{proof}[Proof of Theorem~\ref{sr:thm:main-kappa}]
By the local theory, we have
\[
    \sup_{0\le t\le2}\|\lambda(t)\|_{\Sigma_x^\sigma L_y^2}\lesssim\epsilon_0.
\]
By \eqref{eq:free-Sigma-bound}, \eqref{eq:T-Sigma-local}, and the
profile equation in Remark~\ref{rmk:lambda-C1}, we obtain for $0 \le t \le 1$ that
\[
    \|\partial_t\lambda(t)\|_{\Sigma_x^\sigma L_y^2}
    \lesssim\epsilon_0^3.
\]
For $1\le t<T_*$, take $R=t$ in
Lemma~\ref{sr:lem:nonlinear}.
By \eqref{sr:eq:dispersion-energy},
\begin{equation}\label{sr:eq:T-energy}
    \|\mathcal T(\kappa,\kappa,\kappa)(t)\|_{H_x^\sigma L_y^2}
    \lesssim t^{-\theta}\|\lambda(t)\|_{\Sigma_x^\sigma L_y^2}^{3}
\end{equation}
Applying Lemma~\ref{sr:lem:nonlinear} with $R=t$ after introducing $\widetilde\kappa(t,x,y) = \ee^{-\ii|x|^2/(2t)}\kappa(t,x,y)$, as in the proof of Lemma~\ref{lem:energy}, we have
\[
    \||\mathrm J_x(t)|^\sigma\mathcal T(\kappa,\kappa,\kappa)(t)\|_{L^2_{x,y}}
    \lesssim t^{-\theta} \|\lambda(t)\|_{\Sigma_x^\sigma L_y^2}^{3}.
\]
Together with \eqref{sr:eq:T-energy} and the profile equation in
Remark~\ref{rmk:lambda-C1}, this gives
\begin{equation}\label{sr:eq:profile-derivative}
    \|\partial_t\lambda(t)\|_{\Sigma_x^\sigma L_y^2}
    \lesssim t^{-\theta}\|\lambda(t)\|_{\Sigma_x^\sigma L_y^2}^{3},
    \qquad 1\le t\le T
\end{equation}
Fix $1<T<T_*$. Integrating \eqref{sr:eq:profile-derivative}, we obtain
\[
    \sup_{1\le t\le T}\|\lambda(t)\|_{\Sigma_x^\sigma L_y^2}
    \le C_0\epsilon_0+C_1 \sup_{1\le s\le T}\|\lambda(s)\|_{\Sigma_x^\sigma L_y^2}^{3},
\]
where $C_0,C_1 >1 $ depend only on $d,\theta$.

Take $\epsilon_0$ small enough that $8C_1C_0^2\epsilon_0^2<1$.
A standard continuity argument then implies
\[
    \sup_{1\le t\le T}\|\lambda(t)\|_{\Sigma_x^\sigma L_y^2}
    \le2C_0\epsilon_0,\qquad 1\le T<T_*.
\]
Since $\|\kappa(t)\|_{H_x^\sigma L_y^2}
=\|\lambda(t)\|_{H_x^\sigma L_y^2}$, the continuation criterion in
Lemma~\ref{lem:LWP} implies $T_*=\infty$. Consequently,
\begin{equation}\label{sr:eq:uniform-energy}
    \sup_{t\ge0}\left(
    \|\kappa(t)\|_{H_x^\sigma L_y^2}
    +\||\mathrm J_x(t)|^\sigma\kappa(t)\|_{L^2_{x,y}}\right)
    \lesssim\epsilon_0.
\end{equation}

The $L_x^\infty L_y^2$ decay follows from \eqref{sr:eq:dispersion-energy}.
Interpolation with mass conservation gives \eqref{sr:eq:decay}.
Sobolev embedding and the local theory cover $0\le t\le1$.

Since $\theta>1$, \eqref{sr:eq:profile-derivative} and
\eqref{sr:eq:uniform-energy} imply that $\lambda(t)$ converges in $\Sigma^\sigma_x L^2_y$ as $t \to \infty$. By Duhamel's formula, its limit is
\[
    \kappa_\infty
    =\kappa_0-\mathrm i\int_0^\infty
    \mathrm e^{-\mathrm is\Delta/2}
    \mathcal T(\kappa,\kappa,\kappa)(s)\, \dd s.
\]
For $t\ge1$,
\[
    \|\lambda(t)-\kappa_{\infty}\|_{\Sigma_x^\sigma L_y^2}
    \lesssim\epsilon_0^3\int_t^\infty s^{-\theta}\,\dd s
    \lesssim \epsilon_0^3 \, t^{1-\theta}.
\]
This proves \eqref{sr:eq:scattering}, completing the proof of Theorem~\ref{sr:thm:main-kappa}.
\end{proof}

\begin{proof}[Proof of Theorem~\ref{sr:thm:main}]
Let $\kappa_0=\sqrt{\gamma_0}$. The identities
\[
\begin{aligned}
    \|\kappa_0\|_{H_x^\sigma L_y^2}^2
    =\|\gamma_0\|_{\mathfrak S^{\sigma,1}},\qquad
    \|\langle x\rangle^\sigma\kappa_0\|_{L^2_{x,y}}^2
    =\|\langle x\rangle^\sigma\gamma_0\langle x\rangle^\sigma\|_{\mathfrak S^1}
\end{aligned}
\]
imply $\epsilon_0=\|\kappa_0\|_{\Sigma_x^\sigma L_y^2}\le\sqrt{2\epsilon}$.
Thus Theorem~\ref{sr:thm:main-kappa} applies if $\epsilon_*$ is
sufficiently small. Set
$\gamma(t)=\kappa(t)\kappa(t)^*$ and $\gamma_{\infty}=\kappa_{\infty}\kappa_{\infty}^*$.
Taking trace norms and using \eqref{sr:eq:uniform-energy}, we obtain
\begin{equation}\label{sr:eq:density-energy}
\begin{aligned}
    &\sup_{t\ge0}\left(
    \|\gamma(t)\|_{\mathfrak S^{\sigma,1}}
    +\||\mathrm J(t)|^\sigma\gamma(t)|\mathrm J(t)|^\sigma\|_{\mathfrak S^1}
    \right)
    =\sup_{t\ge0}\left(
    \|\kappa(t)\|_{H_x^\sigma L_y^2}^2
    + \||\mathrm J_x(t)|^\sigma\kappa(t)\|_{L^2_{x,y}}^2\right)\lesssim\epsilon_0^2\lesssim\epsilon.
\end{aligned}
\end{equation}
As in the proof of Theorem~\ref{thm:main}, we also have
$\gamma\in C([0,\infty);\mathfrak S^{\sigma,1})$
and $|\mathrm J|^\sigma\gamma|\mathrm J|^\sigma\in C([0,\infty);\mathfrak S^1)$.
Since $\kappa_{\infty}\in\Sigma_x^\sigma L_y^2$,
we also have $\gamma_{\infty}\in\mathfrak S^{\sigma,1}$ and
$\langle x\rangle^\sigma\gamma_{\infty}\langle x\rangle^\sigma\in\mathfrak S^1$.
Since $\sigma\ge2$, the argument in that proof also shows that
$\gamma\in C^1([0,\infty);\mathfrak S^1)$ and satisfies
\eqref{eq:HF} in $\mathfrak S^1$.
Since $\|\rho_\gamma(t)\|_{L_x^p}
=\|\kappa(t)\|_{L_x^{2p}L_y^2}^{2}$,
\eqref{sr:eq:density-decay} follows from \eqref{sr:eq:decay}.
Uniqueness among self-adjoint solutions in
$C([0,\infty);\mathfrak S^{\sigma,1})$ follows as in the proof of
Theorem~\ref{thm:main}, using $V_{<1}\in L^1$, $V_{\ge1}\in L^\infty$,
and the estimates \eqref{est:direct-Hsigma} and \eqref{est:exchange-Hsigma}.

Finally, by the Schatten--H\"older inequality and \eqref{sr:eq:scattering},
\[ 
\|\gamma(t)-e^{it\Delta/2}\gamma_\infty e^{-it\Delta/2}\|_{\mathfrak S^{\sigma,1}} 
\le 
\bigl(\|\kappa(t)\|_{H_x^\sigma L_y^2} +\|\kappa_\infty\|_{H_x^\sigma L_y^2}\bigr) \|\kappa(t)-e^{it\Delta/2}\kappa_\infty\|_{H_x^\sigma L_y^2} 
\lesssim
\epsilon_0^4t^{1-\theta} \lesssim\epsilon^2t^{1-\theta}
\]
for $t \ge 1$. 
Uniqueness of $\gamma_{\infty}$
follows from convergence of $\ee^{-\ii t\Delta/2}\gamma(t)\ee^{\ii t\Delta/2}$ in $\mathfrak S^{\sigma,1}$.
\end{proof}

\end{document}